\documentclass[11pt]{amsart}
\usepackage{amsmath,amssymb,amsfonts,exscale}

\usepackage[curve,matrix,arrow,cmtip]{xy}
\usepackage{verbatim}

\newdir^{ (}{{}*!/-3pt/\dir^{(}}    
\newdir_{ (}{{}*!/-3pt/\dir_{(}}

\def\Hom{\mathop{\rm Hom}\nolimits}

\def\Spec{\mathop{\rm Spec}\nolimits}

\def\deg{\mathop{\rm deg}\nolimits}

\def\Stab{\mathop{\rm Stab}\nolimits}

\def\rk{\mathop{\rm rk}\nolimits}
\def\spn{\mathop{\rm span}\nolimits}

\def\proj{\mathop{\rm pr}\nolimits}

\def\GL{\mathop{\rm GL}\nolimits}

\def\PGL{\mathop{\rm PGL}\nolimits}

\def\Bun{\mathop{\rm Bun}\nolimits}

\def\Maps{\mathop{\rm Maps}\nolimits}
\def\Dist{\mathop{\rm Dist}\nolimits}

\def\gr{{\rm gr}}

\newbox\starbox 
\setbox\starbox=\hbox{\vphantom{\underline{\ }}\rlap{\kern4pt\small?}\smash{\lower2pt\hbox{\Large$\square$}}}

\def\hatE{{\mathchoice
  {\hbox{\rlap{\smash{\kern1pt\lower1pt\hbox{$\widehat{\phantom{\hbox{$E$}}}$}}}$E$}}
  {\hbox{\rlap{\smash{\kern1pt\lower1pt\hbox{$\widehat{\phantom{\hbox{$E$}}}$}}}$E$}}
  {\widehat E}
  {\widehat E}}}

\def\hatW{\hbox{\rlap{\smash{\lower1pt\hbox{$\widehat{\phantom{\hbox{$W$}}}$}}}$W$}}

\def\tildeW{\hbox{\rlap{\smash{\lower1pt\hbox{$\widetilde{\phantom{\hbox{$W$}}}$}}}$W$}}

\newbox\checkWbox
\setbox\checkWbox\hbox{\rlap{\smash{\kern.8pt\lower4pt\hbox{\huge \v{}}}}$W$}

\def\Fpb{{\overline{\Fp}_P}}
\def\bBun{{\overline{\Bun}}}
\def\pr{{\rm pr}}
\def\proj{{\rm proj}}

\def\circV{{\mathchoice{\circVbig}{\circVbig}{\circVscript}{\circVscriptscript}}}
\def\circVbig{\hbox{\text{\it\r{V}}}}
\def\circVscript{\hbox{\scriptsize\text{\it\r{V}}}}
\def\circVscriptscript{\mbox{\tiny\text{\it\r{V}}}}
\def\circVprime{{\mathchoice{\circV\kern1.8pt{}^\prime}{\circV\kern1.8pt{}^\prime}
                            {\circVscript\kern1.3pt{}^\prime}{\circVscriptscript\kern1pt{}^\prime}}}

\def\circVpprime{{\mathchoice{\circV\kern1.8pt{}^{\prime \prime}}{\circV\kern1.8pt{}^{\prime \prime}}
                            {\circVscript\kern1.3pt{}^{\prime \prime}}{\circVscriptscript\kern1pt{}^{\prime \prime}}}}

\def\lambdach{\check\lambda}
\def\Lambdach{\check\Lambda}
\def\alphach{\check\alpha}

\let\epsilon\varepsilon
\let\setminus\smallsetminus

\let\leq\leqslant
\let\geq\geqslant

\newtheorem{theorem}[subsubsection]{Theorem}

\newtheorem{corollary}[subsubsection]{Corollary}

\newtheorem{proposition-definition}[subsubsection]{Proposition-Definition}
\newtheorem{theorem-definition}[subsubsection]{Theorem-Definition}
\newtheorem{example}[subsubsection]{Example}

\newtheorem{lemma}[subsubsection]{Lemma}

\newtheorem{proposition}[subsubsection]{Proposition}

\newcommand{\BA}{{\mathbb{A}}}

\newcommand{\BG}{{\mathbb{G}}}

\newcommand{\BP}{{\mathbb{P}}}
\newcommand{\BQ}{{\mathbb{Q}}}

\newcommand{\BZ}{{\mathbb{Z}}}

\newcommand{\Fp}{{\mathfrak{p}}}

\newcommand{\CI}{{\mathcal I}}

\newcommand{\CO}{{\mathcal O}}
\newcommand{\CP}{{\mathcal P}}

\newcommand{\CU}{{\mathcal U}}

\newcommand{\CX}{{\mathcal X}}
\newcommand{\CY}{{\mathcal Y}}
\newcommand{\CZ}{{\mathcal Z}}

\newcommand{\ssec}{\subsection}
\newcommand{\sssec}{\subsubsection}

\def\longto{\longrightarrow}
\def\into{\hookrightarrow}
\let\onto\twoheadrightarrow

\def\longinto{\lhook\joinrel\longrightarrow}
\def\longotni{\longleftarrow\joinrel\rhook}
\def\longonto{\ontoover{\ }}

\newbox\mybox
\def\arrover#1{\mathrel{
       \setbox\mybox=\hbox spread 1.4em
              {\hfil$\scriptstyle#1\vphantom{g}$\hfil}
       \vbox{\offinterlineskip\copy\mybox
             \hbox to\wd\mybox{\rightarrowfill}}}}
\def\larrover#1{\mathrel{
       \setbox\mybox=\hbox spread 1.4em{\hfil$\scriptstyle#1$\hfil}
       \vbox{\offinterlineskip\copy\mybox
             \hbox to\wd\mybox{\leftarrowfill}}}}

\def\ontoover#1{\mathrel{
       \setbox\mybox=\hbox spread 1.4em{\hfil$\scriptstyle#1$\hfil}
       \vbox{\offinterlineskip\copy\mybox
             \hbox to\wd\mybox{\rightarrowfill\hskip-2.8mm
                               $\rightarrow$}}}}
\def\leftontoover#1{\mathrel{
       \setbox\mybox=\hbox spread 1.4em{\hfil$\scriptstyle#1$\hfil}
       \vbox{\offinterlineskip\copy\mybox
             \hbox to\wd\mybox{$\leftarrow$\hskip-2.8mm
                               \leftarrowfill}}}}

\newbox\invlimsymbol
\setbox\invlimsymbol=\vtop{\hbox{\rm lim}\vskip-8pt
        \hbox{\hskip1pt$\scriptstyle\longleftarrow$}\vskip-1pt}

\newbox\dirlimsymbol
\setbox\dirlimsymbol=\vtop{\hbox{\rm lim}\vskip-8pt
        \hbox{\hskip1pt$\scriptstyle\longrightarrow$}\vskip-1pt}

\begin{document}

\title[The Harder-Narasimhan stratification of $\Bun_G$]{The Harder-Narasimhan stratification of the moduli stack of $G$-bundles via Drinfeld's compactifications}

\author{Simon Schieder}
\thanks{Dept. of Mathematics, Harvard University, Cambridge, MA 02138, USA}
\thanks{Supported by the International Fulbright Science and Technology Award of the U.S. Department of State}

\address{Dept. of Mathematics, Harvard University, Cambridge, MA 02138, USA}
\email{schieder@math.harvard.edu}

\maketitle

\begin{abstract}
We use Drinfeld's relative compactifications $\bBun_P$ and the Tannakian viewpoint on principal bundles to construct the Harder-Narasimhan stratification of the moduli stack $\Bun_G$ of $G$-bundles on an algebraic curve in arbitrary characteristic, generalizing the stratification for $G=\GL_n$ due to Harder and Narasimhan to the case of an arbitrary reductive group $G$. To establish the stratification on the set-theoretic level, we exploit a Tannakian interpretation of the Bruhat decomposition and give a new and purely geometric proof of the existence and uniqueness of the canonical reduction in arbitrary characteristic. We furthermore provide a Tannakian interpretation of the canonical reduction in characteristic $0$ which allows to study its behavior in families.
The substack structures on the strata are defined directly in terms of Drinfeld's compactifications~$\bBun_P$, which we generalize to the case where the derived group of $G$ is not necessarily simply connected. Using $\bBun_P$ we establish various properties of the stratification, including finer information about the structure of the individual strata and a simple description of the strata closures. Finally, we introduce a novel notion of slope for principal bundles which allows for a natural formulation of the reduction theory of arbitrary reductive groups in arbitrary characteristic.
\end{abstract}

\bigskip\bigskip\bigskip
\bigskip\bigskip\bigskip
\bigskip\bigskip\bigskip

\hspace{.28in} {\footnotesize {\bf Mathematics Subject Classification (MSC 2010):} }

\hspace{.28in} {\footnotesize Primary: 14D24}

\hspace{.28in} {\footnotesize Secondary: 14D20, 14D23}

\bigskip\bigskip\bigskip


\newpage
\tableofcontents

\newpage

\section{Introduction}
\label{intro}

\medskip

\ssec{Overview}
\label{intro overview}

\mbox{} \\

Let $X$ be a smooth complete curve over an algebraically closed field $k$ of arbitrary characteristic, and let $G$ be a reductive linear algebraic group over $k$. The main goal of this article is to generalize the \textit{Harder-Narasimhan stratification} of the moduli stack of vector bundles to the moduli stack $\Bun_G$ of principal $G$-bundles on $X$, in the form stated in our main theorem, Theorem~\ref{main theorem} below.

\medskip

The main technical tool in our construction of the stratification, as well as in establishing various properties, is Drinfeld's relative compactification~$\bBun_P$. The latter object first appeared in the context of geometric Eisenstein series (see \cite{GeometricEisenstein}) and has since been of great importance in many areas of geometric representation theory (see \cite[Sec. 0.1]{ICofDrinfeldCompactifications} for an overview).
Our interest in stratifying the moduli stack $\Bun_G$ stems from the Geometric Langlands Program (see for example \cite{DG:GL}), and more specifically the study of D-modules on $\Bun_G$. For instance, the main theorem of the present paper, Theorem \ref{main theorem}, is applied by V. Drinfeld and D. Gaitsgory in their forthcoming paper \cite{CompactGenerationBunG}.

\medskip

We now briefly describe the stratification on the set-theoretic level. It is known (and we will also prove) that every $G$-bundle on $X$ possesses a unique \textit{canonical reduction} to a unique parabolic subgroup $P$ of $G$, i.e., a reduction to $P$ such that its corresponding Levi bundle is semistable and such that its coweight degree~$\lambdach_P$ enjoys a certain regularity property. To describe the strata, we associate to any given $G$-bundle the pair $(P,\lambdach_P)$ obtained from its canonical reduction. Then on the level of $k$-points the strata are precisely those loci in $\Bun_G$ on which the pair $(P, \lambdach_P)$ remains constant. The open strata together comprise the locus of semistable $G$-bundles.

\medskip

In the case $G = \GL_n$ the idea of stratifying the moduli stack of vector bundles by loci of instability is due to G. Harder and M. S. Narasimhan. It was carried out set-theoretically by Harder and Narasimhan in \cite{HarderNarasimhan} and scheme-theoretically by S. Shatz in \cite{Shatz}. 
The study in the context of the geometric Langlands program and in the language of stacks was initiated by G.~Laumon in \cite{Laumon Eisenstein}.
The case of a general reductive group~$G$ in arbitrary characteristic has been considered by K. A. Behrend in his thesis \cite{Behrend thesis}, although the parts containing the stratification results obtained by his methods seem to remain unpublished. Our approach to stratifying $\Bun_G$ is however quite different, as will be explained in the next sections in more detail. One the one hand, our use of Drinfeld's compactifications $\bBun_P$ greatly simplifies the construction of the strata and the analysis of their geometry. On the other hand, our approach to the canonical reduction uses only basic algebraic geometry and representation theory and avoids any involved combinatorics; we also introduce and use throughout the article a novel notion of \textit{slope} for principal bundles which allows for natural formulations of all our reduction-theoretic definitions, statements, and proofs.

\bigskip\bigskip

\ssec{Main results}

\medskip

\sssec{Results about the canonical reduction}
\label{Results about the canonical reduction}

To establish the stratification on the set-theoretic level, we give a direct and -- to our knowledge -- new proof of the existence and uniqueness of the canonical reduction in arbitrary characteristic. Recall that in the case $G = \GL_n$ the canonical reduction reduces to the Harder-Narasimhan filtration of a vector bundle, as defined by Harder and Narasimhan in \cite{HarderNarasimhan}. For a general reductive group $G$ in characteristic~$0$, the existence and uniqueness of the canonical reduction were proven by A.~Ramanathan (\cite{Ramanathan1}, \cite{Ramanathan2}).

\medskip

In arbitrary characteristic, the existence and uniqueness of the canonical reduction are more involved, and were established much later by Behrend in~\cite{Behrend thesis} and \cite{Behrend semistability} by first passing from principal bundles to reductive group schemes over $X$ and then using the structure theory for the latter as well as a detailed analysis of certain combinatorial devices called \textit{complementary polyhedra on root systems}.

\medskip

Our present approach to the existence and uniqueness of the canonical reduction is different and proceeds in a fairly pedestrian manner, featuring algebraic geometry and representation theory instead of combinatorics. Namely, we give a direct bundle-theoretic and in some sense purely geometric proof which exploits the Bruhat decomposition of the double quotient $P_1 \backslash G / P_2$ for parabolics $P_1$ and $P_2$ of $G$, and consistently uses the Tannakian perspective on principal bundles.
Our use of the double quotient $P_1 \backslash G / P_2$ is to a certain extent reminiscent of the geometry underlying the ``Geometric Lemma'' of I.~N. Bernstein and A. V. Zelevinsky in the representation theory of reductive groups over non-archimedean local fields (see \cite{BernsteinZelevinsky}). The Tannakian viewpoint is employed to obtain a modular interpretation of the Bruhat decomposition, which yields finer information than the set-theoretic decomposition alone.

\medskip

Furthermore, we exhibit a certain extremal property of the canonical reduction in arbitrary characteristic (the \textit{comparison theorem}, Theorem \ref{comparison theorem}), which strengthens the uniqueness assertion and which also immediately yields a strategy to prove existence.

\medskip

In characteristic $0$, and for reductions to the Borel $B$ of $G$ in arbitrary characteristic, we obtain stronger results: We provide a Tannakian interpretation of the canonical reduction, from which we for example deduce that the canonical reduction is also ``unique in families''. The latter is in turn equivalent to Behrend's conjecture (see Section~\ref{Behrend's conjecture}).

\medskip

Throughout the article we employ a novel notion of \textit{slope} for $G$-bundles and their reductions (see Section \ref{Semistability of $G$-bundles} for the definition). It appears to us that our definitions (such as semistability of $G$-bundles), results (such as the comparison theorem), and proofs are most naturally phrased in terms of this notion. Furthermore, its use makes the more technical constructions in Sections \ref{Comparing two reductions} and \ref{The case of characteristic $0$} more transparent. We believe that our notion of slope is also rather intuitive to use, as it brings out the analogy of the case of an arbitrary reductive group with the $\GL_n$-case without resorting to any Tannakian formalism; the latter is not adequate for dealing with semistability in positive characteristic.

\medskip

\sssec{Construction of the strata via Drinfeld's compactifications}

To define our strata as locally closed substacks and not only on the level of $k$-points as above, we use the relative compactifications $\bBun_P$, which are due to V.~Drinfeld and were introduced by A. Braverman and D. Gaitsgory in \cite{GeometricEisenstein}.

\medskip

To motivate the use of $\bBun_P$ in the present context, let $P$ be a parabolic subgroup of~$G$, let $\Bun_P$ denote the moduli stack of $P$-bundles, and consider the natural projection map $\Fp_P: \Bun_P \to \Bun_G$. Then even though we can identify our desired strata on the level of $k$-points with the set-theoretic images of certain open substacks of $\Bun_P$, these images do a priori not carry a natural stack structure since the map $\Fp_P$ is not proper.

\medskip

This is remedied by Drinfeld's compactification $\bBun_P$, which contains $\Bun_P$ as an open dense substack and comes equipped with a proper map $\Fpb: \bBun_P \to \Bun_G$ extending the projection~$\Fp_P$. Thus we define the strata as the stack-theoretic images under $\Fpb$ of certain substacks of $\bBun_P$ and use some well-known properties of the latter to establish that these images indeed possess the desired $k$-points as described in Section \ref{intro overview} above.

\medskip

The definition of Drinfeld's compactification $\bBun_P$ in \cite{GeometricEisenstein} requires the derived group $[G,G]$ of $G$ to be simply connected in order for $\bBun_P$ to have the desired properties (which are stated in Section \ref{Drinfeld overview}). To avoid having a similar restriction in our main theorem, we show in the last section, Section~\ref{drinfeld}, how the definition of $\bBun_P$ can be modified so that it satisfies the desired properties for an arbitrary reductive group $G$. The same strategy applies to Drinfeld's compactifications $\widetilde{\Bun}_P$, which are however not used in the present article.

\medskip

\sssec{Properties of the stratification}

Under the assumption that the characteristic of $k$ is $0$ or that $P=B$, we use Drinfeld's compactifications $\bBun_P$ to show that each individual stratum is isomorphic to its corresponding locus of reductions in the stack $\Bun_P$. This implies for example that all strata are smooth in these cases. For a general parabolic $P$ in arbitrary characteristic we show that the strata are still almost-isomorphic (in a precise sense, see Section \ref{Almost-isomorphisms}) to their corresponding loci of reductions, which is sufficient for applications in the theory of D-modules or etale cohomology.

\medskip

Apart from various elementary properties, we also provide a formula for the closure of a stratum. This formula follows immediately from the construction of the strata via $\bBun_P$ and some well-known properties of the latter. In general the closure of a stratum need however not be a union of strata, as we illustrate with an example for $G=\GL_3$.

\bigskip\medskip

\ssec{Structure of the article}

\mbox{} \\

We now briefly discuss the content of the individual sections.

\medskip

Section \ref{The main theorem} can be viewed as an extension of the introduction; it contains no proofs. Its goal is to state the main theorem (Theorem \ref{main theorem}) after introducing the necessary notation and definitions; here we also define the aforementioned notion of slope for principal bundles.
We conclude the section with a series of remarks detailing those in the introduction, and provide some very basic examples in the case $G=\GL_n$.

\medskip

Section \ref{Preparations} is preparatory. Here we collect several combinatorial lemmas that are used throughout the article, and record some basic results from the Tannakian formalism for principal bundles.

\medskip

In Section \ref{Comparing two reductions} we first analyze the notion of \textit{relative position} of two reductions to parabolics $P_1$ and $P_2$ of the same $G$-bundle on a scheme, in terms of the Bruhat cells of the double quotient stack $P_1 \backslash G / P_2$ (see for example Corollary \ref{FQ1FQ2}). After specializing to the case of a curve and providing a Tannakian interpretation of the Bruhat decomposition of $P_1 \backslash G / P_2$ (Proposition \ref{proposition-modular-bruhat}), we prove the aforementioned comparison theorem (Theorem~\ref{comparison theorem}). Since in characteristic $0$ the comparison theorem will also be deduced from the results of Section \ref{The case of characteristic $0$}, the reader who is only interested in Theorem~\ref{main theorem} in characteristic $0$ can skip Section~\ref{Comparing two reductions} entirely.

\medskip

In Section \ref{The case of characteristic $0$} we provide the Tannakian characterization of the canonical reduction in characteristic $0$ (Proposition \ref{HN}). We furthermore prove the above-mentioned result about the uniqueness of the canonical reduction in families (Proposition~\ref{monoprop}), and give another proof of the comparison theorem in characteristic $0$.

\medskip

In Section \ref{Proof of the remaining parts of the main theorem} we define the strata using Drinfeld's compactifications $\bBun_P$ and combine the results of the previous sections to complete the proof of Theorem \ref{main theorem}.

\medskip

Finally, Section \ref{drinfeld} contains the aforementioned generalization of Drinfeld's compactifications to the case of an arbitrary reductive group, and the proof that it satisfies the desired properties. This section can be read independently from the rest of the article.

\bigskip\bigskip

\ssec{Acknowledgements}

\mbox{} \\

I would like to thank Vladimir Drinfeld and Dennis Gaitsgory for initiating the present article in parallel to their work \cite{CompactGenerationBunG}, which makes use of our main theorem, Theorem \ref{main theorem}. I~would furthermore like to thank my doctoral advisor Dennis Gaitsgory for numerous suggestions and useful conversations, as well as for helpful comments on an earlier draft of this paper. Finally, I would like to thank the referee for valuable comments, and in particular for bringing to my attention the article \cite{Laumon Eisenstein}.

\bigskip\bigskip\bigskip\bigskip

\section{The main theorem}
\label{The main theorem}

\medskip

\ssec{The setting}
\label{The setting}

\medskip

\sssec{Notation related to the group}

Let $k$ be an algebraically closed field of arbitrary characteristic, and let~$G$ be a connected reductive linear algebraic group over $k$. Fix a Borel subgroup $B$ of $G$ and let $T = B/U(B)$ denote the abstract Cartan, where $U(B)$ is the unipotent radical of $B$. We also fix a splitting of the surjection $B \onto T$, i.e. a realization of $T$ as a maximal torus in $B$. By a parabolic subgroup of $G$ we will always mean a standard parabolic subgroup, i.e., one containing the fixed Borel $B$. The connected component of the center of $G$ will be denoted by $Z_0(G)$, the collection of roots and coroots of $G$ by $R$ and $\check{R}$, their positive parts by $R_+$ and $\check{R}_+$, and the set of vertices of the Dynkin diagram by $\CI$. As usual the Weyl group of~$G$ will be denoted by $W$, and its longest element by $w_0$.

\medskip

Next let $\Lambdach_G$ denote the coweight lattice of $G$, let $\Lambda_G$ denote the weight lattice, and let 
$$\langle.,.\rangle : \ \Lambdach_G \times \Lambda_G \ \longto \ \BZ$$
denote the natural pairing between the two. Let furthermore $\Lambdach^{pos}_G \subset \Lambdach_G$ be the semigroup of positive coweights and let $\Lambdach^+_G \subset \Lambdach_G$ be the semigroup of dominant coweights, and similarly for $\Lambda_G$. Set $\Lambdach_G^\BQ := \Lambdach_G \otimes_\BZ \BQ$, and let~$\Lambdach_{G}^{\BQ,pos}$ and $\Lambdach_{G}^{\BQ,+}$ denote the rational cones of $\Lambdach^{pos}_G$ and $\Lambdach^+_G$, and analogously for the weight lattice. As usual, given two coweights $\lambdach, \check{\mu} \in \Lambdach_G^{\BQ}$ we write $\lambdach \geq \check{\mu}$ if the difference $\lambdach - \check{\mu}$ lies in $\Lambdach_G^{\BQ, pos}$.

\medskip

\sssec{Notation related to a parabolic}
\label{notation-parabolic}

Let $P$ be a parabolic subgroup of $G$, let $U(P)$ be its unipotent radical, and let $M = P/U(P)$ be its Levi quotient. Having fixed a splitting of the surjection $B \onto T$ we also obtain an induced splitting of the surjection $P \onto M$. The subset of vertices of $\CI$ corresponding to $P$ will be denoted by~$\CI_M$, the collection of roots of the Levi $M$ by $R_M$, and the root lattice of $M$ by
$$\BZ R_M \ := \ \spn_{\BZ}(R_M) \ \subset \ \Lambda_G \, .$$
Furthermore, for any vertex $i \in \CI$ we will denote by $P_i$ the corresponding maximal parabolic subgroup of $G$, i.e., the parabolic corresponding to the subset $\CI \setminus \{ i \} \, \subset \CI$.

\medskip

Next define the sublattice $\Lambda_{G,P} \subset \Lambda_G$ as
$$\Lambda_{G,P} := \{ \lambda \in \Lambda_G \mid \langle \alphach_i , \lambda \rangle = 0 \text{ for all } i \in \CI_M \} .$$
Thus for $P = B$ we have $\Lambdach_{G,B} = \Lambdach_G$.
Consider furthermore the sublattice $\Lambdach_{[M,M]_{sc}} \subset \Lambdach_G$ spanned by the simple coroots $\alphach_i$ for $i \in \CI_M$, and set 
$$\Lambdach_{G,P} := \Lambdach_G / \Lambdach_{[M,M]_{sc}}.$$
By definition we have $\Lambda_{G,P} = \Lambda_{M,M}$ and $\Lambdach_{G,P} = \Lambdach_{M,M}$.
Furthermore the pairing $\langle .,. \rangle$ above induces a pairing
$$\langle.,.\rangle : \Lambdach_{G,P} \times \Lambda_{G,P} \to \BZ \, ,$$
which becomes perfect after quotienting out by the torsion part of $\Lambdach_{G,P}$.
As before we let $\Lambdach_{G,P}^{\BQ} := \Lambdach_{G,P} \otimes_\BZ \BQ$ and similarly for $\Lambda_{G,P}^\BQ$. Finally,
we denote by $\Lambdach_{G,P}^{pos}$ the image of $\Lambdach_G^{pos}$ in $\Lambdach_{G,P}$, and analogously for
$\Lambdach_{G,P}^{\BQ, pos} \subset \Lambdach_{G,P}^{\BQ}$.

\medskip

\sssec{The slope map $\phi_P$}
\label{phiP}
We now define a map
$$\phi_P: \ \Lambdach_{G,P} \ \longto \ \Lambdach_G^{\BQ}$$
which we will call the \textit{slope map}. Its importance, as well as our choice of terminology, will become apparent in Sections \ref{Semistability of $G$-bundles} and \ref{Statement of the theorem} below.

\medskip

Let $Z_0(M)$ denote the connected component of the center of $M$. The splitting of the surjection $P \onto M$ gives rise to a natural inclusion
$$\Lambdach_{Z_0(M)}^{\BQ} \ \longinto \ \Lambdach_G^{\BQ} \, ,$$
and the composition of this inclusion with the projection $\Lambdach_G^{\BQ} \onto \Lambdach_{G,P}^{\BQ}$ is an isomorphism
$$\Lambdach_{Z_0(M)}^{\BQ} \ \stackrel{\cong}{\longto} \ \Lambdach_{G,P}^{\BQ} \, .$$
We then define the map $\phi_P$ as the composition
$$\Lambdach_{G,P} \ \longto \ \Lambdach_{G,P}^{\BQ} \cong \Lambdach_{Z_0(M)}^{\BQ} \ \longinto \ \Lambdach_G^{\BQ} \, .$$

\medskip

It follows directly from the definition of $\phi_P$ that for any elements $\lambdach_P \in \Lambdach_{G,P}$ and $\lambda_P \in \Lambda_{G,P}$ we have
\begin{equation}
\label{phi-compatibility}
\langle \lambdach_P, \lambda_P \rangle \ = \ \langle \phi_P(\lambdach_P), \lambda_P \rangle \, , \tag{\ref{phiP}}
\end{equation}
where the pairing on the right hand side is the natural one between $\Lambdach_G^\BQ$ and~$\Lambda_G^\BQ$.

\medskip

Finally, observe that for any element $\lambdach_P \in \Lambdach_{G,P}$ we have $\langle \phi_P(\lambdach_P), \alpha_i \rangle = 0$ for all $i \in \CI_M$. We call $\lambdach_P$ {\it dominant $P$-regular} if $\langle \phi_P(\lambdach_P), \alpha_i \rangle > 0$ for all $i \in \CI \setminus \CI_M$. 

\medskip\medskip\medskip

\ssec{ Slope and semistability of $G$-bundles}
\label{Semistability of $G$-bundles}

\sssec{Moduli stacks of bundles}

Let $X$ be a smooth and complete curve over $k$ and let $\Bun_G$ denote the moduli stack of principal $G$-bundles on $X$, and similarly for other linear algebraic groups. Extension of structure group along the homomorphisms $P \into G$ and $P \onto M$ defines maps of stacks
$$\xymatrix@+10pt{
\Bun_P \ar[r]^{\mathfrak{q}_P} \ar[d]^{\mathfrak{p}_P}   &   \Bun_M   \\
\Bun_G                                  &                   \\
}$$

It is well-known that the map $\mathfrak{q}_P$ induces a bijection on the sets of connected components of $\Bun_P$ and $\Bun_M$, and that
$$\pi_0(\Bun_P) \ \cong \ \pi_0(\Bun_M) \ \cong \ \Lambdach_{G,P} \, .$$

\sssec{Degree and slope of a bundle}
\label{Degree and slope of a bundle}

We denote the connected component of $\Bun_P$ and $\Bun_M$ corresponding to $\lambdach_P \in \Lambdach_{G,P}$ by $\Bun_{P,\lambdach_P}$ and $\Bun_{M,\lambdach_P}$, respectively. If $F_P$ is a $P$-bundle on $X$ which lies in $\Bun_{P,\lambdach_P}$ we will refer to the element $\lambdach_P$ as the \textit{degree} of $F_P$.

\medskip

Furthermore, we propose to call the element $\phi_P(\lambdach_P)$ the \textit{slope} of~$F_P$. We have several reasons for this choice of terminology.
First, for $G=\GL_n$ this notion indeed reduces to the usual notion of slope of vector bundles (see Section \ref{GL_n}).
Second, for an arbitrary reductive group~$G$, the element~$\phi_P(\lambdach_P)$ in fact determines the slopes of certain naturally associated vector bundles (see Proposition \ref{slope-of-associated-bundles} and Remark \ref{slope remark}); in other words, the element~$\phi_P(\lambdach_P)$ yields the correct Tannakian version of slope. Third, by analogy with the case $G = \GL_n$ (see Section \ref{GL_n}), the terminology is motivated by the re-definition of semistability of $G$-bundles that we introduce next.

\medskip

\sssec{Re-definition of semistability}
\label{definition of semistability}

Using the above notion of slope, we propose the following re-definition of semistability of $G$-bundles, which is equivalent, though not tautologically so, to the usual definition.

\medskip

Let $F_G \in \Bun_{G, \lambdach_G}$ be a $G$-bundle on $X$. Then we call~$F_G$ {\it semistable} if one of the following equivalent conditions holds:

\begin{itemize}
\item[]
\item[$(a)$] For every parabolic $P$ and for every element $\lambdach_P \in \Lambdach_{G,P}$ such that~$F_G$ admits a reduction $F_P \in \Bun_{P, \lambdach_P}$ of degree $\lambdach_P$ we have
$$\phi_P (\lambdach_P) \ \leq \ \phi_G (\lambdach_G) \, .$$
\item[$(b)$] For every element $\lambdach_B \in \Lambdach_{G,B} = \Lambdach_G$ such that $F_G$ admits a reduction $F_B \in \Bun_{B, \lambdach_B}$ of degree $\lambdach_B$ we have
$$\lambdach_B \ \leq \ \phi_G (\lambdach_G) \, .$$
\item[$(c)$] For every maximal parabolic $P_i$ and every element $\lambdach_{P_i} \in \Lambdach_{G,{P_i}}$ such that $F_G$ admits a reduction $F_{P_i} \in \Bun_{P_i, \lambdach_{P_i}}$ we have
$$\phi_{P_i} (\lambdach_{P_i}) \ \leq \ \phi_G (\lambdach_G) \, .$$
\end{itemize}

\medskip\medskip

The claimed equivalence of the three conditions follows easily from Proposition \ref{phialpha} below; the argument is carried out in Lemma \ref{semistability equivalence}. In the same lemma we also use Proposition \ref{phialpha} to show that our definition of semistability of $G$-bundles agrees with the usual definition from \cite{Ramanathan1}, \cite{Ramanathan3}, \cite{RamananRamanathan}; in the language of Section \ref{notation-parabolic} the usual definition can be expressed as follows:

\medskip

Note first that for every parabolic $P$ of $G$ the vector space $\Lambdach_{G,P}^\BQ$ is the direct sum of the subspaces $\Lambdach_{Z_0(G)}^\BQ$ and $\sum_{i \in \CI \setminus \CI_M} \BQ \alphach_i$.
We denote by $\proj_P$ the projection onto the second summand, i.e., the composition
$$\proj_P: \ \Lambdach_{G,P}^\BQ \ \longonto \ \sum_{i \in \CI \setminus \CI_M} \BQ \alphach_i \ \longinto \ \Lambdach_{G,P}^\BQ \, .$$

\medskip

Let now $F_G \in \Bun_{G, \lambdach_G}$ be a $G$-bundle on $X$. Then $F_G$ is semistable in the sense of the usual definition if it satisfies the following condition:

\medskip

\begin{itemize}
\item[$(d)$]
For every parabolic $P$ and every element $\lambdach_P \in \Lambdach_{G,P}$ such that $F_G$ admits a reduction $F_P \in \Bun_{P, \lambdach_P}$, the image of $\lambdach_P$ under the composition
$$\Lambdach_{G,P} \ \longto \ \Lambdach_{G,P}^\BQ \ \stackrel{\ \proj_P}{\longonto} \ \Lambdach_{G,P}^\BQ$$

\medskip

\noindent lies in the negative cone $- \Lambdach_{G,P}^{\BQ, pos}$.
\end{itemize}

\medskip\medskip\medskip

Our new definition turns out to be more convenient for our purposes (see for example Theorem \ref{comparison theorem}, Proposition \ref{semistable locus}, Lemma \ref{P-semistability}, Lemma~\ref{semistability difference lemma}, Proposition \ref{existence and uniqueness of the canonical reduction}), and appears to streamline various arguments in the reduction theory of a general reductive group. Furthermore, it closely resembles the well-known definition of slope-semistability of vector bundles, to which it reduces for $G = \GL_n$ and which we now recall.

\medskip

\sssec{The case $G = \GL_n$}
\label{GL_n}

It is easy to see that a $\GL_n$-bundle is semistable in the above sense if and only if the corresponding vector bundle is slope-semistable in the following sense. Recall first that the \textit{slope} $\mu(E)$ of a vector bundle~$E$ on the curve $X$ is defined as
$$\mu(E) := \frac{\deg E}{\rk E} \,$$
where $\deg E$ and $\rk E$ denote the degree and rank of $E$, respectively. Now~$E$ is called {\it semistable} if
$$\mu(F) \leq \mu(E)$$
for every locally free subsheaf $F \subseteq E$, or equivalently for every subbundle $F \subseteq E$.

\medskip

For $G=\GL_n$ the maps $\phi_G$ and $\phi_P$ have the following meaning. Let $E$ be a vector bundle on $X$ corresponding to a $G$-bundle of degree $\lambdach_G$, and identify $\Lambdach_{\GL_n}^\BQ$ with $\BQ^n$ in the canonical way. Then
$$\phi_G(\lambdach_G) \ = \ \bigl( \mu(E), \ldots, \mu(E) \bigr) \, \in \, \BQ^n.$$

\medskip

More generally, let
$$0 \neq E_1 \subsetneq E_2 \subsetneq \ldots \subsetneq E_m = E$$
be the flag of vector bundles corresponding to a $P$-bundle $F_P$ of degree $\lambdach_P$ for some parabolic $P \subset \GL_n$. Then
$$\phi_P(\lambdach_P) \ = \ \bigl( \mu(E_1), \ldots, \mu(E_1), \ldots\ldots, \mu(E_m/E_{m-1}), \ldots, \mu(E_m/E_{m-1}) \bigr) \, \in \, \BQ^n$$
where each $\mu(E_i/E_{i+1})$ is repeated $\rk(E_i)$ times.

\medskip\medskip\medskip
\bigskip\bigskip

\ssec{Statement of the theorem}
\label{Statement of the theorem}

\sssec{Loci of semistability}

It is well-known that for any reductive group $G$ the collection of semistable $G$-bundles constitutes an open substack $\Bun_G^{ss}$ of $\Bun_G$ (see also Proposition \ref{semistable locus} below, where we give a quick proof of this fact using Drinfeld's relative compactifications $\bBun_P$), and that its intersection with each connected component of $\Bun_G$ is quasicompact.

\medskip

For a parabolic $P$ of $G$, we denote by $\Bun_P^{ss}$ the inverse image of $\Bun_M^{ss}$ under the projection $\mathfrak{q}_P$, and by $\Bun_{P,\lambdach_P}^{ss}$ the intersection of $\Bun_P^{ss}$ with a given connected component $\Bun_{P,\lambdach_P}$. As the projection $\mathfrak{q}_P$ is quasicompact we see that the open substacks $\Bun_{P,\lambdach_P}^{ss}$ are quasicompact as well.

\medskip

A direct definition of $\Bun_P^{ss}$ in the spirit of Definition \ref{definition of semistability} can be found in Lemma \ref{P-semistability} below.

\medskip

\sssec{Almost-isomorphisms}
\label{Almost-isomorphisms}

We call a morphism of algebraic stacks $\CX \to \CY$ an \textit{almost-isomorphism} if it is schematic, finite, and if the fiber of every geometric point of $\CY$ consists, as a topological space, of precisely one point.

\medskip

This terminology stems from the fact that in certain contexts, such as in the theory of D-modules or etale cohomology, an almost-isomorphism can play the same role as an actual isomorphism.

\medskip\medskip\medskip

We can now formulate the main theorem:

\medskip

\begin{theorem}
\label{main theorem}

\begin{itemize}
\item[]
\item[]
\item[$(a)$]
Let $P$ be a parabolic in $G$ and let $\lambdach_P \in \Lambdach_{G,P}$ be dominant $P$-regular. Then there exists a unique reduced quasicompact locally closed substack $\Bun_G^{P,\lambdach_P}$ of $\Bun_G$ such that the projection $\Fp_P: \Bun_P \to \Bun_G$ induces an almost-isomorphism between $Bun_{P,\lambdach_P}^{ss}$ and $\Bun_G^{P,\lambdach_P}$.

\medskip

\item[$(b)$]
If the field $k$ is of characteristic $0$, or if $P=B$, then the map $\Fp_P$ in fact induces an isomorphism
$$\Bun_{P,\lambdach_P}^{ss} \ \stackrel{\cong}{\longto} \ \Bun_G^{P,\lambdach_P} \, .$$

\medskip

\item[$(c)$]
The substacks $\Bun_G^{P,\lambdach_P}$ for all pairs $(P, \lambdach_P)$ as in $(a)$ define a stratification of $\Bun_G$ in the sense that every $k$-point of $\Bun_G$ lies in a unique $\Bun_G^{P,\lambdach_P}$.

\medskip

\item[$(d)$]
Let $P'$ be another parabolic and let $\lambdach'_{P'}$ be any element of $\Lambdach_{G,P'}$. If the image of $\Bun_{P',\lambdach'_{P'}}$ in $\Bun_G$ meets a stratum
$\Bun_G^{P, \lambdach_P}$, then we have
$$\phi_P(\lambdach_P) \ \geq \ \phi_{P'}(\lambdach'_{P'}) \, .$$

\medskip

\item[$(e)$]
The stratification of $\Bun_G$ by the strata $\Bun_G^{P,\lambdach_P}$ is locally finite in the sense that there exists a covering of $\Bun_G$ by open substacks, each of which intersects only finitely many strata.

\medskip

\item[$(f)$]
On the level of $k$-points, the closure of the stratum $\Bun_G^{P,\lambdach_P}$ in $\Bun_G$ is equal to the (non-disjoint) union
$$\overline{\Bun_G^{P,\lambdach_P}} \ = \ \bigcup_{\check\theta \in \Lambdach_{G,P}^{pos}} \Fp_P \bigl( \Bun_{P, \lambdach_P + \check\theta} \bigr).$$

\medskip

\item[$(g)$]
The closure of a stratum $\Bun_G^{P,\lambdach_P}$ need not be a union of strata, or equivalently, it need not contain every stratum it meets. However, if the closure of $\Bun_G^{P,\lambdach_P}$ meets a stratum $\Bun_G^{P,\lambdach_P'}$ corresponding to the same parabolic $P$, then it contains that entire stratum.

\medskip

\end{itemize}
\end{theorem}

\bigskip

\ssec{Remarks and complements}

\sssec{The canonical reduction}
\label{The canonical reduction}

A reduction $F_P$ of a $G$-bundle on $X$ is called {\it canonical} if $F_P \in \Bun_{P,\lambdach_P}^{ss}$ for $\lambdach_P$ dominant $P$-regular. Thus parts $(a)$ and~$(c)$ of Theorem \ref{main theorem} together imply that every $G$-bundle on $X$ has a unique canonical reduction $F_P$ to a unique parabolic $P$. The strata of the theorem are precisely the loci in $\Bun_G$ obtained by fixing the numerical invariants of the canonical reduction.

\medskip

\sssec{The Harder-Narasimhan filtration}
\label{The Harder-Narasimhan filtration}

One sees easily that in the case $G = \GL_n$ the canonical reduction of a vector bundle $E$ is precisely the {\it Harder-Narasimhan filtration} of $E$, i.e., the unique filtration
$$0 \neq E_1 \subsetneq E_2 \subsetneq \ldots \subsetneq E_m = E$$
of $E$ by subbundles with the property that all quotients $E_{i+1}/E_i$ are semistable and that
$$\mu(E_1) > \mu(E_2/E_1) > \ldots > \mu(E_m/E_{m-1}).$$

\medskip

It is easy to see directly that every vector bundle on $X$ possesses a unique Harder-Narasimhan filtration. Its first term equals the {\it maximal destabilizing subsheaf} of $E$, i.e. the unique subbundle $F \subseteq E$ with the properties that $\mu(F) \geq \mu(F')$ for any other locally free subsheaf $F' \subseteq E$, and that $F' \subseteq F$ if $\mu(F') = \mu(F)$. It is the unique subbundle of $E$ of maximal rank among all subbundles of maximal slope. 

\medskip

For $G = \GL_n$, Theorem \ref{main theorem} yields a stratification of the moduli stack of rank $n$ vector bundles on $X$ by loci on which the ranks and slopes of the subbundles occurring in the Harder-Narasimhan filtration remain constant. In this case the stratification was already defined set-theoretically by G. Harder and M. S. Narasimhan in \cite{HarderNarasimhan}. By analogy with this special case one might call the stratification of $\Bun_G$ in Theorem \ref{main theorem} the \textit{Harder-Narasimhan stratification} of $\Bun_G$, and its strata the \textit{Harder-Narasimhan strata}. As has already been remarked in the introduction, several of the assertions of Theorem \ref{main theorem} can also be extracted, in different language and with different proofs, from Behrend's unpublished thesis \cite{Behrend thesis}.

\medskip

\sssec{Behrend's conjecture}
\label{Behrend's conjecture}

Let $P$ be any parabolic of $G$ and let $F_P$ be the canonical reduction of a $G$-bundle on the curve $X$. Denote the Lie algebras of $G$ and $P$ by $\frak g$ and $\frak p$, respectively, and consider the vector bundle $(\frak g / \frak p)_{F_P}$ obtained by twisting the $P$-representation $\frak g / \frak p$ by $F_P$. Then Behrend conjectured in \cite{Behrend thesis} that the space of global sections of this vector bundle vanishes:
$$H^0(X, (\frak g / \frak p)_{F_P}) \ = \ 0 \, .$$

\medskip

\noindent By considering the differential of the projection $\Fp_P: \Bun_P \to \Bun_G$ and using part $(a)$ of Theorem \ref{main theorem}, it is easy to show that Behrend's conjecture holds for the group $G$ if and only if the conclusion of part $(b)$ of Theorem \ref{main theorem} is valid, i.e., if and only if the map $\Bun_{P,\lambdach_P}^{ss} \to \Bun_G^{P,\lambdach_P}$ is an isomorphism for any parabolic $P$ and any degree $\lambdach_P$.

\medskip

However, in the present approach, part $(b)$ of Theorem \ref{main theorem} is proven directly, and thus Behrend's conjecture in characteristic $0$ follows. Since Behrend's conjecture is known to be false in positive characteristic (see for example \cite{Heinloth Behrend's conjecture}), the restriction to characteristic $0$ in part $(b)$ is necessary.

\medskip

\bigskip

\ssec{Examples for $G=\GL_n$}

\mbox{} \\

We now collect some very basic examples illustrating certain aspects of Theorem \ref{main theorem} in the case $G=\GL_n$.

\medskip

\sssec{Strata closure for vector bundles of rank $2$}

Consider the moduli stack $\Bun_{\GL_2}$ of vector bundles of rank $2$ on any smooth projective curve $X$, and let $\Bun_{\GL_2,d}$ denote the connected component consisting of all bundles of degree $d$. As the group $\GL_2$ has semisimple rank $1$, part $(g)$ of Theorem~\ref{main theorem} shows that the Harder-Narasimhan stratification of $\Bun_{\GL_2,d}$ has the property that the closure of a stratum is a union of strata.

\medskip

Using part $(f)$ of Theorem \ref{main theorem} one easily determines the strata occurring in this union. To do so, recall first that the open stratum $\Bun_{\GL_2,d}^{ss}$ consists of all semistable bundles, and that the remaining strata $\Bun_{GL_2}^{(\ell, d - \ell)}$ are parametrized by integers $\ell$ satisfying $\ell > \frac{d}{2}$. Namely, the stratum $\Bun_{GL_2}^{(\ell, d - \ell)}$ consists of precisely those bundles $E$ whose Harder-Narasimhan flag is of the form $0 \neq L \subsetneq E$ for a line bundle $L$ of degree $\ell$, which must then be larger than the slope $\mu(E)=\frac{d}{2}$ of $E$. Specializing part $(f)$ of Theorem \ref{main theorem} to the case $G=\GL_2$ then yields:

\medskip

\begin{example}
\label{GL2 strata closure}
The closure of the stratum $\Bun_{GL_2}^{(\ell, d - \ell)}$ is equal to the union of all strata $\Bun_{GL_2}^{(\ell', d - \ell')}$ for all $\ell' \geq \ell$.
\end{example}

\bigskip

\sssec{Specialization of vector bundles on the projective line}

Up to isomorphism every vector bundle $E$ on the projective line $\BP^1$ is of the form
$$E \ = \ \CO(d_1) \oplus \CO(d_2) \oplus \ldots \oplus \CO(d_n)$$
for integers $d_1 \geq d_2 \geq \ldots \geq d_n$. A natural deformation-theoretic question is to ask which vector bundles on $\BP^1$ the bundle $E$ can specialize to. This question has the following easy and well-known answer, which for simplicity we state in the case that $d_1 > d_2 > \ldots > d_n$; the general case is analogous.

\medskip

\begin{example}
\label{P1}
Let
$E' = \CO(d_1') \oplus \ldots \oplus \CO(d_n')$
be another vector bundle on~$\BP^1$ of the same rank and degree as $E$. Then $E$ specializes to $E'$ if and only if some permutation of the tuple $(d_1', \ldots, d_n')$ lies in the subset
$$(d_1, \ldots, d_n) \ + \ \sum_{i=1}^{n-1} \, \BZ_{\geq 0} \, (e_i - e_{i+1})$$
of the lattice $\BZ^n$.
\end{example}

\medskip

Though easy to prove directly, this fact is also a very special case of part~$(f)$ of Theorem \ref{main theorem}. Namely, for $X=\BP^1$ we let $G=\GL_n$ and let~$P$ be the standard Borel $B$ of $\GL_n$. Then the claim follows immediately from the formula in part $(f)$ and from the fact that for $G=\GL_n$ the simple coroots are given by $\alphach_i = e_i - e_{i+1}$ in the coweight lattice $\Lambdach_{\GL_n} = \BZ^n$. (The author would like to thank the referee for pointing out that Example \ref{P1} has already been discussed in great detail in Laumon's article \cite{Laumon Eisenstein}.)

\bigskip

\sssec{Counterexample to strata closure}

As asserted in part $(g)$ of Theorem~\ref{main theorem}, the closure of a stratum need not be a union of strata. By the second assertion of part~$(g)$, this however cannot happen if the semisimple rank of $G$ is $1$, as in Example \ref{GL2 strata closure} for $G=\GL_2$ above.

\medskip

In the general case, we provide the following counterexample in Section~\ref{An example of strata closure} below. Let $G=\GL_3$, let $B$ denote the standard Borel subgroup of $\GL_3$, and assume that the genus of the curve $X$ is at least~$2$. Then using part $(f)$ of Theorem \ref{main theorem} we show that the closure of the stratum $\Bun_{\GL_3}^{B,(2,1,0)}$, which consists of those vector bundles $E$ of rank $3$ whose Harder-Narasimhan flag
$$0 \neq L \subsetneq F \subsetneq E$$
is complete and satisfies
$$\deg(L) = 2, \ \ \deg(F/L) = 1, \ \ \text{and} \ \deg(E/F) = 0,$$
is not a union of strata.

\medskip

\sssec{The almost-isomorphism for $G=\GL_n$}

The almost-isomorphism
$$\Bun_{P,\lambdach_P}^{ss} \ \longto \ \Bun_G^{P,\lambdach_P}$$
from part $(a)$ of Theorem \ref{main theorem} is always an isomorphism for $G=\GL_n$, i.e., the assertion of part $(b)$ of the theorem holds for any parabolic and in any characteristic. This follows immediately, and exactly as in the proof of part~$(b)$ of the theorem in Section \ref{Construction of the strata and proof of $(a)$, $(b)$} below, once one establishes that Proposition \ref{monoprop} always holds for $G=\GL_n$; the latter states that the above map is a monomorphism. To prove this one only needs to adapt the usual proof of the uniqueness of the Harder-Narasimhan flag to families using the theorem on cohomology and base change, in the same way it is used in the proof of Proposition \ref{monoprop} below.

\bigskip\bigskip\bigskip\bigskip

\section{Preparations}
\label{Preparations}

\medskip

\ssec{Lemmas about the slope map $\phi_P$}

\mbox{} \\

In this preparatory section we collect several easy lemmas about the slope map $\phi_P: \Lambdach_{G,P} \to \Lambdach_G^\BQ$ defined in Section \ref{phiP} which will be used throughout the article, and expound on the definition of semistability in Section \ref{definition of semistability} above.

\medskip

\sssec{Preservation of the partial ordering}

Part $(a)$ of the following elementary proposition turns out to be of surprising importance and will be used frequently:

\medskip

\begin{proposition}
\label{phialpha}
\begin{itemize}
\item[]
\item[]
\item[$(a)$]
For any $j \in \CI \setminus \CI_M$ the element $\phi_P(\alphach_j)$ lies in~$\Lambdach_G^{\BQ, pos}$. In other words, the map $\phi_P$ preserves the partial orders ``$\leq$'' on $\Lambdach_{G,P}$ and~$\Lambdach_G^\BQ$ in the sense that it maps $\Lambdach_{G,P}^{pos}$ to $\Lambdach_G^{\BQ, pos}$.
\item[]
\item[$(b)$] For any $j \in \CI \setminus \CI_M$ we have $\langle \phi_P(\alphach_j), \alpha_j \rangle > 0$.
\end{itemize}
\end{proposition}

\medskip

\begin{proof}[Proof of Proposition \ref{phialpha}]
Fix a $W$-invariant scalar product $(. \, ,.)$ on $\Lambdach_G^\BQ$. By the definition of the map $\phi_P$ we can write $\phi_P(\alphach_j)$ as
$$\phi_P(\alphach_j) \ = \ \sum_{i' \in \CI_M'} c_{i'} \alphach_{i'} \ - \ \sum_{i'' \in \CI_M''} c_{i''} \alphach_{i''} \ + \ \alphach_j$$
where $\CI_M'$ and $\CI_M''$ are disjoint subsets of $\CI_M$ and all $c_{i'}$ and $c_{i''}$ are positive. Pairing both sides of this equation with the sum $\sum_{i''} c_{i''} \alphach_{i''}$, the left hand side becomes
$$\bigl( \phi_P(\alphach_j) , \sum_{i''} c_{i''} \alphach_{i''} \bigr) \ = \ 0.$$
On the right hand side, we find that
$$\bigl( \sum_{i'} c_{i'} \alphach_{i'} , \sum_{i''} c_{i''} \alphach_{i''} \bigr) \ \leq \ 0$$
and
$$\bigl( \alphach_j, \sum_{i''} c_{i''} \alphach_{i''} \bigr) \ \leq \ 0.$$
Together this implies that
$$\bigl( \sum_{i''} c_{i''} \alphach_{i''} , \sum_{i''} c_{i''} \alphach_{i''} \bigr) \ \leq \ 0,$$
but then positive definiteness forces $\sum_{i''} c_{i''} \alphach_{i''} = 0$, as desired.

\medskip

To prove the second assertion, we pair both sides of the equality
$$\phi_P(\alphach_j) \ = \ \sum_{i'} c_{i'} \alphach_{i'} \ + \ \alphach_j$$
with the sum $\sum_{i'} c_{i'} \alphach_{i'} + \alphach_j$. Then the left hand side equals $\bigl( \phi_P(\alphach_j) , \alphach_j \bigr)$ since
$$\bigl( \phi_P(\alphach_j) , \sum_{i'} c_{i'} \alphach_{i'} \bigr) \ = \ 0.$$
But by positive definiteness the right hand side must be positive, and the second assertion of the lemma follows.
\end{proof}

\medskip\medskip

\sssec{Comparison of slope maps}

We now record two simple lemmas relating the slope maps $\phi_P$ and $\phi_{P'}$ for different parabolics $P$ and $P'$ in two specific circumstances which arise frequently. To state the first lemma, we denote by
$$\phi_P^\BQ: \ \Lambdach_{G,P}^\BQ \ \longinto \Lambdach_G^\BQ$$
the natural map induced by $\phi_P$, i.e., the composition
$$\Lambdach_{G,P}^{\BQ} \ \stackrel{\cong}{\longto} \ \Lambdach_{Z_0(M)}^{\BQ} \ \longinto \ \Lambdach_G^{\BQ} \, .$$

\medskip

We then have:

\medskip

\begin{lemma}
\label{compatibility of slope maps}
Let $P$ and $P'$ be parabolics in $G$ and assume that $P \subset P'$. Let furthermore $\lambdach_{P'}$ be an element of $\Lambdach_{G,P'}$, and let the image of $\phi_{P'}(\lambdach_{P'})$ under the projection $\Lambdach_G^\BQ \onto \Lambdach_{G,P}^\BQ$ be denoted by $\overline{\phi_{P'}(\lambdach_{P'})}$. Then we have

$$\phi_P^\BQ \Bigl( \overline{\phi_{P'}(\lambdach_{P'})} \Bigr) \ = \ \phi_{P'}(\lambdach_{P'}) \, .$$
\end{lemma}

\medskip

\begin{proof}
Follows directly from the definitions and the fact that, denoting the Levi of~$P'$ by $M'$, the inclusion $\Lambdach_{Z_0(M')}^\BQ \subset \Lambdach_{Z_0(M)}^\BQ$ holds in $\Lambdach_G^\BQ$.
\end{proof}

\medskip\medskip

We now record the second lemma:

\medskip

\begin{lemma}
\label{difference lemma for two parabolics}
Let $P$ and $P'$ be parabolic subgroups of $G$, and let $\lambdach_P \in \Lambdach_{G,P}$ and $\lambdach_{P'} \in \Lambdach_{G,P'}$.
\begin{itemize}
\item[$(a)$] Assume that $\lambdach_P$ and $\lambdach_{P'}$ map to the same element of $\Lambdach_{G,G}$ under the natural projections. Then the difference
$$\phi_P(\lambdach_P) - \phi_{P'}(\lambdach_{P'})$$
lies in the subspace
$$\sum_{i \in \CI} \BQ \alphach_i \ \subset \ \Lambdach_G^\BQ \, .$$
\item[$(b)$] Assume now that $P' \subset P$, and that $\lambdach_{P'}$ maps to $\lambdach_P$ under the natural projection. Then the difference
$$\phi_P(\lambdach_P) - \phi_{P'}(\lambdach_{P'})$$
lies in the subspace
$$\sum_{i \in \CI_M} \BQ \alphach_i \ \subset \ \Lambdach_G^\BQ \, ,$$
where $\CI_M \subset \CI$ denotes the subset corresponding to $P$.
\end{itemize}
\end{lemma}

\medskip

\begin{proof}
Since the kernel of the projection $\pr: \Lambdach_G^\BQ \onto \Lambdach_{G,G}^\BQ$ is precisely the subspace $\sum_{i \in \CI} \BQ \alphach_i$, we can prove part $(a)$ by showing that
$$\pr (\phi_P(\lambdach_P)) \ = \ \pr (\phi_{P'}(\lambdach_{P'})).$$
Denoting the image of $\lambdach_G$ under the natural map $\Lambdach_{G,G} \to \Lambdach_{G,G}^\BQ$ by $\lambdach_G \otimes 1$, the commutativity of the diagram
$$\xymatrix{
\Lambdach_{G,P} \ar^{\phi_P}[r] \ar@{>>}[d]   &   \Lambdach_G^\BQ \ar@{>>}^\pr[d]   \\
\Lambdach_{G,G} \ar[r]   &   \Lambdach_{G,G}^\BQ   \\   
}$$
shows that $\pr (\phi_P(\lambdach_P)) = \lambdach_G \otimes 1$. The exact same reasoning for $P'$ instead of~$P$ yields that
$\pr (\phi_{P'}(\lambdach_{P'})) \ = \lambdach_G \otimes 1$ as well, finishing the proof of part~$(a)$.

\medskip

Part $(b)$ can be verified analogously. Alternatively, part $(b)$ also follows from part $(a)$. To see this, consider the parabolic subgroup $P'/U(P)$ of the Levi quotient $M$. We then have natural identifications
$$\Lambdach_{G,P} = \Lambdach_{M,M} \ \ \ \text{and} \ \ \ \Lambdach_{G,P'} = \Lambdach_{M, P'/U(P)} \, ,$$
and furthermore the slope maps $\phi_P$ and $\phi_{P'}$ for the group $G$ agree with the slope maps~$\phi_M$ and $\phi_{P'/U(P)}$ for the group $M$ under these identifications. Thus part $(b)$ follows from part $(a)$ applied to the reductive group $M$ and its parabolic subgroups $M$ and $P'/U(P)$.
\end{proof}

\medskip

\sssec{Lemmas about semistability}
\label{Lemmas about semistability}

We now establish the claimed equivalence of conditions $(a)$, $(b)$, $(c)$ in our re-definition of semistability in Section~\ref{definition of semistability} above, and furthermore show that our definition agrees with the usual definition of semistability, condition $(d)$, from \cite{Ramanathan1}, \cite{Ramanathan3}, \cite{RamananRamanathan}.

\medskip

\begin{lemma}
\label{semistability equivalence}
The four conditions $(a)$, $(b)$, $(c)$, $(d)$ of Definition~\ref{definition of semistability} are equivalent.
\end{lemma}

\medskip

\begin{proof}
The implications $(a) \Rightarrow (b)$ and $(a) \Rightarrow (c)$ are formal. We now prove the implication $(b) \Rightarrow (a)$. Thus assuming condition $(b)$, we need to show that given a reduction $F_P \in \Bun_{P, \lambdach_P}$ of $F_G$, the difference $\phi_G(\lambdach_G) - \phi_P(\lambdach_P)$ lies in the positive cone $\Lambdach_G^{\BQ, pos}$. To do so, choose a reduction $F_B$ of $F_P$ to the Borel $B$ and let $\lambdach_B \in \Lambdach_G$ denote its degree. Then the difference $\phi_G(\lambdach_G) - \lambdach_B$ lies in the positive cone $\Lambdach_G^{\BQ, pos}$, and thus its image in $\Lambdach_{G,P}^\BQ$ lies in $\Lambdach_{G,P}^{\BQ, pos}$. But since applying the slope map $\phi_P^\BQ$ to this image yields precisely the difference $\phi_G(\lambdach_G) - \phi_P(\lambdach_P)$ by Lemma \ref{compatibility of slope maps}, the claim follows from part $(a)$ of Proposition \ref{phialpha}.

\medskip

Next we prove the implication $(c) \Rightarrow (b)$. Thus assuming condition $(c)$, we have to prove that given a reduction $F_B \in \Bun_{B, \lambdach_B}$ of $F_G$, the difference $\phi_G(\lambdach_G) - \lambdach_B$ lies in $\Lambdach_G^{\BQ, pos}$. By part $(a)$ of Lemma \ref{difference lemma for two parabolics} we have that
$$\phi_G(\lambdach_G) - \lambdach_B \ \, = \ \ \sum_{i \in \CI} c_i \alphach_i$$
for certain rational numbers $c_i \in \BQ$. We hence need to show that $c_i \geq 0$ for any~$i$. To do so, consider the $P_i$-bundle $F_{P_i}$ induced from $F_B$, and let $\lambdach_{P_i} \in \Lambdach_{G,P_i}$ denote its degree. We now first project both sides of the last equality to $\Lambdach_{G,P_i}^\BQ$ and then apply the slope map $\phi_{P_i}$ to both projections. Using lemma \ref{compatibility of slope maps} this yields the equality
$$\phi_G(\lambdach_G) - \phi_{P_i}(\lambdach_{P_i}) \ \ = \ \ c_i \cdot \phi_{P_i} (\alphach_i).$$
But then condition $(c)$ applied to the reduction $F_{P_i}$ of $F_G$ and part $(a)$ of Proposition \ref{phialpha} together show that $c_i \geq 0$, as desired.

\medskip

We have now established the equivalence of the conditions $(a)$, $(b)$, $(c)$. We finish the proof by showing that conditions $(a)$ and $(d)$ are equivalent. Let~$P$ and $\lambdach_P$ be as in the formulation of these conditions.
Using part~$(a)$ of Lemma \ref{difference lemma for two parabolics} one checks easily that
$$\phi_P(\proj_P(\lambdach_P)) \ = \ \phi_P(\lambdach_P) - \phi_G(\lambdach_G).$$
Hence part $(a)$ of Proposition \ref{phialpha} implies that the element $\proj_P(\lambdach_P)$ lies in the negative cone $-\Lambdach_{G,P}^{\BQ,pos}$ if and only if the element $\phi_P(\lambdach_P) - \phi_G(\lambdach_G)$ lies in the negative cone $-\Lambdach_G^{\BQ,pos}$, and the equivalence follows.
\end{proof}

\medskip\medskip

We also record the following easy characterization of those $P$-bundles on the curve $X$ whose induced $M$-bundles are semistable, i.e., of those lying in~$\Bun_P^{ss}$.

\medskip

\begin{lemma}
\label{P-semistability}
A $P$-bundle $F_P \in \Bun_{P, \lambdach_P}$ on $X$ lies in $\Bun_P^{ss}$ if and only if for any smaller parabolic $P' \subset P$ and any element $\lambdach_{P'} \in \Lambdach_{G,P'}$ such that~$F_P$ admits a reduction $F_{P'} \in \Bun_{P', \lambdach_{P'}}$ we have
$$\phi_{P'} (\lambdach_{P'}) \ \leq \ \phi_P (\lambdach_P) \, .$$
In fact, it suffices to check this condition for reductions $F_{P'}$ to maximal proper sub-parabolics $P' \subset P$, i.e., to those parabolics corresponding to the subsets $\CI_M \setminus \{i\}$ for any $i \in \CI_M$.
\end{lemma}

\medskip

\begin{proof}
Recall that the assignment $P' \mapsto P'/U(P)$ defines a bijection between the collection of parabolic subgroups of $G$ contained in $P$ and the collection of parabolic subgroups of the Levi $M$, and furthermore that there is a natural one-to-one correspondence between reductions of $F_P$ to $P'$ and reductions of the induced Levi bundle $F_M$ to $P'/U(P)$. Finally, this correspondence induces the identity map
$$\Lambdach_{G,P'} \ = \ \Lambdach_{M,P'/U(P)}$$
on the level of degrees. Thus using Lemma \ref{difference lemma for two parabolics} $(b)$ we see that the first assertion of the lemma is nothing but a restatement of part $(a)$ of the definition of semistability \ref{definition of semistability} for the Levi bundle $F_M$. Similarly, the second assertion is a restatement of part $(c)$ of the same definition.
\end{proof}

\medskip
\bigskip

\ssec{Associated bundles and Tannakian formalism for reductions}
\label{Tannakian formalism for reductions}

\mbox{} \\

In this section we record some basic but frequently used lemmas about certain associated vector bundles of a given $G$-bundle $F_G$, as well as about certain subbundles induced by a given reduction of $F_G$ to a parabolic subgroup. For the majority of the article we will use bundles associated to an arbitrary $G$-representation of highest weight $\lambda$. However, in certain applications in Section \ref{The case of characteristic $0$} it will be essential to use specifically the Weyl modules of $G$, due to Proposition~\ref{Plucker redundancy} below.

\sssec{Notation}
\label{notation-reps-bundles}

For any linear algebraic group $H$ we will denote its algebra of distributions at the element $1 \in H$ by $\Dist(H)$. Its use in the sequel stems from the fact that in arbitrary characteristic the algebra of distributions of~$H$ retains many of the features that the Lie algebra of $H$ enjoys only in characteristic $0$. For an exposition of its basic properties see \cite[Sec. I.7 and Sec. II.1.12]{Jantzen representations of reductive groups}.

\medskip

Next, given a $G$-bundle $F_G$ on a scheme $S$ and a finite-dimensional representation $V$ of $G$ we denote by
$$V_{F_G} \ := \ V \overset{G}{\times} F_G$$
the associated vector bundle on $S$, and similarly for any linear algebraic group over $k$. For any weight $\lambda \in \Lambda_G$ we let $k_\lambda$ be the corresponding $1$-dimensional representation of the maximal torus $T$, and we let
$$\CO(\lambda) \ := \ k_\lambda \overset{B}{\times} G$$
be the associated line bundle on the flag variety $G/B$. Finally, for a dominant weight $\lambda \in \Lambda_G^+$ we denote by $V^\lambda$ the corresponding Weyl $G$-module
$$V^\lambda \ := \ H^0 (G/B, \CO(- w_0 \lambda))^* \, .$$

\medskip

\sssec{Representations of highest weight $\lambda$}

Fix a dominant weight $\lambda \in \Lambda_G^+$ and let $V$ be any finite-dimensional $G$-representation of highest weight~$\lambda$,~i.e., if
$$V \ = \ \, \bigoplus_\nu V[\nu]$$
is the weight decomposition of $V$, then $\lambda \geq \nu$ for all~$\nu$. Given any parabolic~$P$ of $G$, we define the subspace $V[\lambda + \BZ R_M] \subset V$ as the sum of weight spaces
$$V[\lambda + \BZ R_M] \ \ := \ \ \bigoplus_{\nu \in \, \lambda + \BZ R_M} V[\nu] \, .$$

\medskip

We record the following basic fact:

\medskip

\begin{lemma}
\label{the usual sum of weight spaces}
\begin{itemize}
\item[]
\item[]
\item[$(a)$] The subspace $V[\lambda + \BZ R_M]$ is a $P$-subrepresentation of $V$. Furthermore, the unipotent radical $U(P)$ of $P$ acts trivially on $V[\lambda + \BZ R_M]$, and hence the action of $P$ descends to an action of the Levi $M$.
\item[]
\item[$(b)$] Let now $V=V^\lambda$ be the Weyl module corresponding to the dominant weight $\lambda$. Then the space $V^\lambda[\lambda + \BZ R_M]$ is $1$-dimensional if and only if $\lambda$ lies in~$\Lambda_{G,P}$. In this case this space is then equal to the highest weight line $V^\lambda[\lambda] \subset V^\lambda$.
\end{itemize}
\end{lemma}

\medskip

\begin{proof}
It suffices to check the two assertions of part $(a)$ on the level of the algebras of distributions $\Dist(P)$ and $\Dist(U(P))$, for which they are immediate.

\medskip

To prove part $(b)$, note first that every weight $\nu$ of the $M$-representation $V^\lambda[\lambda + \BZ R_M]$ satisfies
$$\lambda \ \geq \ \nu \ \geq \ w_{0,M}(\lambda)$$
where $w_{0,M}$ denotes the longest element of the Weyl group $W_M$ of $M$. As the weight $\lambda$ occurs with multiplicity $1$ in $V^\lambda[\lambda + \BZ R_M]$, it follows that $V^\lambda[\lambda + \BZ R_M]$ is $1$-dimensional if and only if $w_{0,M}(\lambda) = \lambda$.

\medskip

If now $\lambda$ lies in $\Lambda_{G,P}$, then every simple reflection $s_{\alpha_i}$ with $i \in \CI_M$ leaves $\lambda$ invariant, and thus $w_{0,M}(\lambda) = \lambda$, proving one direction. Conversely, assume that $V^\lambda[\lambda + \BZ R_M]$ is $1$-dimensional. Then since $s_{\alpha_i}(\lambda)$ is again a weight of $V^\lambda[\lambda + \BZ R_M]$ for any $i \in \CI_M$, we conclude that $s_{\alpha_i}(\lambda) = \lambda$ and thus $\langle \alphach_i, \lambda \rangle = 0$ as desired.
\end{proof}

\medskip

\sssec{Subbundles induced by reductions}
\label{Subbundles induced by reductions}

Let $F_P$ be a reduction of a $G$-bundle $F_G$ on a scheme $S$ to a parabolic $P$, and let $F_M$ denote the corresponding Levi bundle. Then for any dominant weight $\lambda \in \Lambda_G^+$ and any $G$-representation $V$ of highest weight $\lambda$, the inclusion map $V[\lambda + \BZ R_M] \into V$ gives rise to a subbundle map
$$\kappa^\lambda: \ V[\lambda + \BZ R_M]_{F_M} \ \longinto \ V_{F_G}$$
between associated vector bundles.

\medskip

The maps $\kappa^\lambda$ will play a prominent role in the study of the collection of reductions of a given $G$-bundle. One instance of this is the use of Proposition~\ref{Plucker redundancy} below in Section \ref{The case of characteristic $0$}. Another instance is that we will use the subbundles $V[\lambda + \BZ R_M]_{F_M}$ to compare numerical data attached to two reductions of a $G$-bundle on the curve $X$ in Section \ref{Relative position in the curve case and a key inequality} below. In that section, as well as in other parts of the article, the following calculation will be used frequently:

\medskip

\begin{proposition}
\label{slope-of-associated-bundles}
Let $V$ be a finite-dimensional representation of $G$, let $V = \bigoplus_{\nu} V[\nu]$ be its weight decomposition with weights $\nu \in \Lambda_G$, and let $m_\nu$ denote the multiplicity of $\nu$. Furthermore let $F_G \in \Bun_{G,\lambdach_G}$ be a $G$-bundle on the curve~$X$ of degree $\lambdach_G \in \Lambdach_{G,G}$. Then we have:
\begin{itemize}
\item[]
\item[$(a)$]
The sum $\sum_\nu m_\nu \nu$ lies in $\Lambda_{G,G}$, and the degree of the associated vector bundle $V_{F_G}$ is
$$\deg V_{F_G} \ = \ \langle \lambdach_G, \sum_\nu m_\nu \nu \rangle \ = \ \langle \phi_G(\lambdach_G), \sum_\nu m_\nu \nu \rangle \, .$$

\item[$(b)$]
If $V$ is of highest weight $\lambda \in \Lambda_G^+$, then the slope of the associated vector bundle $V_{F_G}$ is
$$\mu(V_{F_G}) \ = \ \langle \phi_G(\lambdach_G) , \lambda \rangle \, .$$
\end{itemize}

\medskip

\noindent Next let $P$ be a parabolic in $G$ and let $F_P \in \Bun_{P,\lambdach_P}$ be a $P$-bundle of degree $\lambdach_P \in \Lambdach_{G,P}$ with corresponding Levi bundle $F_M$. Consider the $M$-representation $V[\lambda + \BZ R_M]$ obtained from a $G$-representation $V$ of highest weight $\lambda$. Then we have:
\begin{itemize}
\item[]
\item[$(c)$]
The slope of the associated vector bundle $V[\lambda + \BZ R_M]_{F_M}$ is
$$\mu \bigl( V[\lambda + \BZ R_M]_{F_M} \bigr) \ = \ \langle \phi_P(\lambdach_P) , \lambda \rangle \, .$$
\end{itemize}
\end{proposition}

\medskip

\begin{proof}
The first assertion of part $(a)$ holds since the weight $\sum_\nu m_\nu \nu$ corresponds to the action of the torus $T$ on the determinant representation~$\det(V)$ of $G$.
For the second assertion of part $(a)$, choose a reduction $F_B$ of $F_G$ to the Borel $B$, and let $\lambdach~\in~\Lambdach_G=\Lambdach_{G,B}$ denote its degree. Recall that if $k_\lambda$ denotes the $1$-dimensional representation of~$B$ corresponding to a character $\lambda~\in~\Lambda_G$, then the degree of the associated line bundle $(k_\lambda)_{F_B}$ is
$$\deg ((k_\lambda)_{F_B}) \ = \ \langle \lambdach, \lambda \rangle \, .$$

\medskip

Considering $V$ as a $B$-representation it possesses a filtration by $B$-sub- representations such that each successive quotient is $1$-dimensional. Each such quotient must be isomorphic to some $k^\nu$ for $\nu$ a weight of $V$, and the number of times $k^\nu$ appears among these quotients is precisely $m_\nu$. Hence
\begin{eqnarray}
\nonumber \deg (V_{F_G}) & = & \deg (V_{F_B}) \\
\nonumber & = & \sum_\nu m_\nu \deg (k^\nu_{F_B}) \\
\nonumber & = & \langle \lambdach , \sum_\nu m_\nu \nu \rangle \\
\nonumber & = & \langle \lambdach_G , \sum_\nu m_\nu \nu \rangle \, ,
\end{eqnarray}
proving the first half of the formula. The second half is just equation (\ref{phi-compatibility}) above.

\medskip

To prove part $(c)$, note that $\lambda - \nu$ lies in $\Lambda_M^{pos}$ for every weight $\nu$ of $V[\lambda + \BZ R_M]$, and hence
$$\langle  \phi_P (\lambdach_P) , \lambda \rangle \ = \ \langle \phi_P (\lambdach_P) , \nu \rangle$$
as $\langle \phi_P(\lambdach_P), \alpha_i \rangle = 0$ for all $i \in \CI_M$. Combined with part $(a)$ applied to the reductive group $M$ we obtain
$$\deg (V[\lambda + \BZ R_M]_{F_M}) \ = \ \bigl( \sum_\nu m_\nu \bigr) \cdot \langle  \phi_P (\lambdach_P) , \lambda \rangle \, ,$$
giving rise to the desired formula for the slope $\mu (V[\lambda + \BZ R_M]_{F_M})$.

\medskip

Finally, part $(b)$ is just the case $P=G$ in part $(c)$.
\end{proof}

\medskip

\sssec{Remark}
\label{slope remark}

Proposition \ref{slope-of-associated-bundles} provides another reason for calling the element $\phi_P(\lambdach_P)$ the ``slope'' of a $P$-bundle $F_P$ of degree $\lambdach_P$. Namely, by part~$(c)$ the function
$$\langle \phi_P(\lambdach_P) , - \rangle: \ \Lambda_G^\BQ \ \longto \ \BQ$$
sends any weight $\lambda \in \Lambda_G$ to the slope of the subbundle corresponding to the inclusion $\kappa^\lambda$.

\medskip

\sssec{The case of Weyl modules}

We now consider the induced subbundles~$\kappa^\lambda$ in the case where the representation $V$ of highest weight $\lambda$ is specifically the Weyl module $V^\lambda$ of $G$. First recall from Lemma \ref{the usual sum of weight spaces} $(b)$ above that $V^\lambda[\lambda + \BZ R_M] = V^\lambda[\lambda]$ if the weight $\lambda$ lies in $\Lambda_{G,P}$, and hence the inclusions $\kappa^\lambda$ correspond to line subbundles for such $\lambda$. These line subbundles are of great importance in the study of reductions of $G$-bundles to parabolic subgroups (see \cite[Ch. 1]{GeometricEisenstein}, \cite[Ch. 1]{ICofDrinfeldCompactifications}). We will only need the following fact:

\medskip

\begin{proposition}
\label{Plucker redundancy}
Let $F_G$ be a $G$-bundle on a scheme $S$ and let $F_P$ and $\tilde{F}_P$ be two reductions of $F_G$ to the same parabolic $P$. Then if the line subbundles
$$V^\lambda[\lambda]_{F_M} \ \stackrel{\kappa^{\lambda}}{\longinto} \ V^\lambda_{F_G} \ \stackrel{\tilde \kappa^{\lambda}}\longotni \ V^\lambda[\lambda]_{\tilde F_M}$$
are equal for all $\lambda \in \Lambda_G^+ \, \cap \, \Lambda_{G,P}$, then the reductions $F_P$ and $\tilde{F}_P$ are already equal.
\end{proposition}

\begin{proof}
Follows from the description of Pl\"ucker data in for example \cite[Sec. 1.3.2]{GeometricEisenstein} or \cite[Sec. 1.1]{ICofDrinfeldCompactifications}, taking into account that the action of $M$ on $V^\lambda[\lambda]$ factors through the torus $M/[M,M]$ and that the induced $M/[M,M]$-bundle $F_{M/[M,M]}$ is uniquely determined
by the inclusions of line subbundles $\kappa^\lambda$
for $\lambda \in \Lambda_G^+ \, \cap \, \Lambda_{G,P}$.
\end{proof}

\medskip

According to the proposition, a reduction of $F_G$ to a parabolic $P$ is already uniquely determined by the subbundle maps $\kappa^\lambda$ for $\lambda \in \Lambda_G^+ \, \cap \, \Lambda_{G,P}$. However, unlike in Proposition \ref{Plucker redundancy}, we will be interested in comparing reductions to possibly different parabolics of $G$, and hence it is important to consider the subbundles $\kappa^\lambda$ for all $\lambda \in \Lambda_G^+$ above.

\bigskip\bigskip\bigskip\bigskip\bigskip\bigskip

\section{Comparing two reductions}
\label{Comparing two reductions}

\medskip

Fix two parabolic subgroups $P_1$ and $P_2$ of $G$. The purpose of this section is to compare two given reductions $F_{P_1}$ and $F_{P_2}$ of the same $G$-bundle on the curve $X$. Its main goal, and its sole application towards the proof of Theorem \ref{main theorem}, is to prove the comparison theorem, Theorem \ref{comparison theorem} below.

\medskip

The comparison theorem is used in the proofs of almost all parts of Theorem \ref{main theorem} in section \ref{Proof of the remaining parts of the main theorem} below. For example, it immediately implies the uniqueness of the canonical reduction and thus establishes half of part $(c)$. It furthermore directly implies part $(d)$, and is also used in the proofs of parts $(a)$ and $(g)$.
As we include a simpler proof of the comparison theorem in characteristic $0$ in Section \ref{The case of characteristic $0$} below, the reader only interested in that case can skip the present Section \ref{Comparing two reductions} entirely.

\bigskip\bigskip\bigskip

\ssec{Relative position of two reductions}
\label{Relative position of two reductions}

\mbox{} \\

\medskip

Consider the double quotient stack $P_1 \backslash G / P_2$. By definition, it parametrizes triples $(F_{P_1}, F_{P_2}, \gamma)$ consisting of a $P_1$-bundle $F_{P_1}$, a $P_2$-bundle $F_{P_2}$, and an isomorphism
$$\gamma: \ (F_{P_1})_G \ \cong \ (F_{P_2})_G$$
of their induced $G$-bundles. Thus given two reductions $F_{P_1}$ and $F_{P_2}$ of the same $G$-bundle on a scheme $S$, we obtain a map
$$S \ \longto \ P_1 \backslash G / P_2 \, .$$
The main idea of the present Section \ref{Comparing two reductions} is to use the geometry of the stack $P_1 \backslash G / P_2$ -- namely, a Tannakian interpretation of its Bruhat stratification -- to compare the two reductions. We begin by introducing the notion of relative position of two reductions.

\medskip

\sssec{Notation}

The Levi quotients of $P_1$ and $P_2$ will be denoted by $M_1$ and $M_2$ and will also be considered as subgroups via the fixed splitting of $B \onto T$, see Section \ref{notation-parabolic} above. Furthermore, the subsets of the set of vertices of the Dynkin diagram $\CI$ corresponding to $M_1$ and $M_2$ will be denoted by
$$\CI_{M_1} \ \subset \ \CI \ \supset \ \CI_{M_2} \, ,$$
the roots of $M_1$ and $M_2$ by
$$R_{M_1} \ \subset \ R \ \supset \ R_{M_2} \, ,$$
and their corresponding subgroups of the Weyl group $W$ of $G$ by
$$W_{M_1} \ \subset W \ \supset W_{M_2} \, .$$

\medskip

Now consider the set of double cosets $W_{M_1} \backslash W / W_{M_2}$ for the action of $W_{M_1}$ on $W$ from the left and of $W_{M_2}$ from the right. We obtain a canonical set of representatives by taking the unique shortest element in each coset.
In other words, we consider the subset
$$W_{1,2} \ = \ \{ w \in W \mid \ell (w) \leq \ell (w_1 w w_2) \ \text{for all} \ w_1 \in W_{M_1}, w_2 \in W_{M_2} \}$$
of $W$. Furthermore we also choose representatives in $G$ for each $w \in W_{1,2}$ and again denote them by $w$ by abuse of notation.

\medskip

\sssec{Stratification of $P_1 \backslash G / P_2$}
\label{Bruhat-stratification}

By the Bruhat decomposition, the group~$G$ is the disjoint union of double cosets
$$G \ = \ \bigcup_{w \in W_{1,2}} P_1 w P_2 \, ,$$
and the cosets $P_1 w P_2$ are locally closed subvarieties of $G$. Quotienting out by the left action of $P_1$ and the right action of $P_2$ thus yields a stratification of the double quotient stack
$$
P_1 \backslash G / P_2 \ = \ \bigcup_{w \in W_{1,2}} (P_1 \backslash G / P_2)_w
$$
by locally closed substacks
$$(P_1 \backslash G / P_2)_w \ := \ P_1 \backslash (P_1 w P_2) / P_2 \, .$$

\medskip

\sssec{Definition of relative position}

Let $F_{P_1}$ and $F_{P_2}$ be two reductions of the same $G$-bundle on a scheme $S$ and let
$$\psi: \ S \ \longto \ P_1 \backslash G / P_2$$
be the induced map. Moreover let $w \in W_{1,2}$. Then we say that $F_{P_1}$ \textit{is in relative position $w$ with respect to} $F_{P_2}$ if the map $\psi$ factors through the substack $(P_1 \backslash G / P_2)_w$ of $P_1 \backslash G / P_2$.

\medskip

We furthermore say that $F_{P_1}$ \textit{is generically in relative position $w$ with respect to} $F_{P_2}$ if there exists an open dense subscheme $U$ of $S$ such that the restriction of $\psi$ to $U$ factors through $(P_1 \backslash G / P_2)_w$.

\medskip

For example we have the following immediate lemma:

\medskip

\begin{lemma}
\label{generic-relative-position}
Let $F_{P_1}$ and $F_{P_2}$ be two reductions of the same $G$-bundle on an integral scheme $S$. Then there exists a unique element $w \in W_{1,2}$ such that $F_{P_1}$ is generically in relative position $w$ with respect to $F_{P_2}$.
\end{lemma}

\bigskip

\ssec{Deeper reductions}
\label{Deeper reductions}

\sssec{Overview}

Let $F_{P_1}$ and $F_{P_2}$ be two reductions of the same $G$-bundle $F_G$ on a scheme $S$ and assume that $F_{P_1}$ is in relative position $w \in W_{1,2}$ with respect to $F_{P_2}$. In this section we construct certain ``deeper'' reductions, i.e., reductions $F_{Q_1}$ of $F_{P_1}$ and $F_{Q_2}$ of $F_{P_2}$ to smaller parabolics $Q_1 \subset P_1$ and $Q_2 \subset P_2$.

\medskip

The parabolics $Q_1$ and $Q_2$ only depend on the element $w$ and have the property that their corresponding Levi subgroups $L_1$ and $L_2$ are $w$-conjugate inside the group $G$. Furthermore, the reductions~$F_{Q_1}$ and $F_{Q_2}$ have the property that their induced Levi bundles $F_{L_1}$ and $F_{L_2}$ are naturally isomorphic when $L_1$ and $L_2$ are identified via the element $w$ (see Corollary \ref{FQ1FQ2} below):

$$\xymatrix@+10pt{
F_{Q_1} \ar@{|->}[r]\ar@{|->}[d] & F_{L_1} \cong F_{L_2} & F_{Q_2}\ar@{|->}[l]\ar@{|->}[d] \\
F_{P_1} \ar@{|->}[dr]    &                                                                &   F_{P_2} \ar@{|->}[dl] \\
            &                              F_G                               &      \\
}$$

\medskip

The geometry we employ in this section to construct the deeper reductions~$F_{Q_1}$ and $F_{Q_2}$ is to a certain extent reminiscent of the geometry underlying the ``Geometric Lemma'' of I.~N. Bernstein and A. V. Zelevinsky in the representation theory of reductive groups over non-archimedean local fields (see \cite{BernsteinZelevinsky}).

\medskip

\sssec{Lemmas about the Weyl group}

We start with the following two easy lemmas about the subset $W_{1,2}$ of the Weyl group $W$:

\medskip

\begin{lemma}
\label{W-positivity-lemma}
Let $w \in W_{1,2}$. Then we have:
\begin{align*}
\forall i \in \CI_{M_1} & : \ w^{-1}(\alpha_i) \in R_+ \\
\forall i \in \CI_{M_2} & : \ w(\alpha_i) \in R_+
\end{align*}
\end{lemma}

\medskip

\begin{proof}
To prove the second assertion, suppose there exists some $i \in \CI_{M_2}$ such that $w(\alpha_i) \in -R_+$. Then as the simple reflection $s_{\alpha_i} \in W_{M_2}$ permutes the set $R_+ \setminus \{ \alpha_i \}$, the element $w \, s_{\alpha_i}$ is shorter than $w$, in contradiction to the fact that $w \in W_{1,2}$. The first assertion is checked analogously.
\end{proof}

\medskip

\begin{lemma}
\label{W-simple-roots-lemma}
Let $w \in W_{1,2}$. If $i \in \CI_{M_1}$ and $w^{-1}(\alpha_i) \in R_{\CI_{M_2}}$, then $w^{-1}(\alpha_i)$ is again a simple root. Similarly, if $i \in \CI_{M_2}$ and $w(\alpha_i) \in R_{\CI_{M_1}}$, then $w(\alpha_i)$ is again a simple root.
\end{lemma}

\medskip

\begin{proof}
For $i \in \CI_{M_1}$ the assumption together with the first half of Lemma \ref{W-positivity-lemma} implies that $w^{-1}(\alpha_i)$ lies in the positive part $R_{\CI_{M_2},+}$. Now if $w^{-1}(\alpha_i)$ was not simple, then by the second half of Lemma \ref{W-positivity-lemma} its image under $w$ could not be simple either, but of course $w(w^{-1}(\alpha_i)) = \alpha_i$. The second statement is verified analogously.
\end{proof}

\medskip\medskip

\sssec{The parabolics $Q_1$ and $Q_2$}
\label{The parabolics $Q_1$ and $Q_2$}

Fix an element $w \in W_{1,2}$. We now define the parabolics $Q_1 \subset P_1$ and $Q_2 \subset P_2$. Namely, we define the following sets of vertices:
\begin{align*}
\CI_{L_1} \ & := \ \{ i \in \CI_{M_1} \mid \exists j \in \CI_{M_2} : w(\alpha_j) = \alpha_i \}   \\
\CI_{L_2} \ & := \ \{ i \in \CI_{M_2} \mid \exists j \in \CI_{M_1} : w^{-1}(\alpha_j) = \alpha_i \}
\end{align*}

\medskip

Let $Q_1$ and $Q_2$ denote the parabolic subgroups of $G$ corresponding to the sets of vertices $\CI_{L_1}$ and $\CI_{L_2}$, and let $L_1$ and $L_2$ denote their Levi quotients, which as before we simultaneously regard as subgroups. Since $\CI_{L_1} \subset \CI_{M_1}$ and $\CI_{L_2} \subset \CI_{M_2}$ we have $Q_1 \subset P_1$ and $Q_2 \subset P_2$. Furthermore, we denote by $R_{L_1} \subset R_{M_1}$ and $R_{L_2} \subset R_{M_2}$ the corresponding sets of roots. Finally, it follows directly from the definition of $\CI_{L_1}$ and $\CI_{L_2}$ that conjugation by $w$ maps $L_2$ isomorphically onto $L_1$:
$$w L_2 w^{-1} \ = \ L_1$$

\medskip

\sssec{Lemmas about $Q_1$ and $Q_2$}

Next we record two lemmas about the parabolics $Q_1$ and $Q_2$ that will be needed for the construction of the deeper reductions $F_{Q_1}$ and $F_{Q_2}$ in Propositions \ref{Q1Q2P1P2 relative positions} and \ref{Q1Q2 factorization} below.

\medskip

First, the element $w$ was chosen from the set of coset representatives $W_{1,2} \subset W$ in order to make the following lemma hold true:

\medskip

\begin{lemma}
\label{RL1RL2}
\begin{align*}
R_{L_1} \ & = \ R_{M_1} \ \cap \ w(R_{M_2})   \\
R_{L_2} \ & = \ w^{-1}(R_{M_1}) \ \cap \ R_{M_2}
\end{align*}
\end{lemma}

\medskip

\begin{proof}
We prove only the first assertion, the proof of the second one being similar. It is clear from the definition of $\CI_{L_1}$ that $R_{L_1}$ is contained in the intersection $R_{M_1} \cap \, w(R_{M_2})$. To prove the converse inclusion, it suffices
to show that every positive root $\alpha$ in this intersection can be written as a sum of simple roots corresponding to elements $i \in \CI_{L_1}$. So let $\alpha = \sum_{i \in \CI_{M_1}} n_i \alpha_i$ with integers $n_i \geq 0$. Then
$$w^{-1}(\alpha) \ = \ \sum_{i \in \CI_{M_1}} n_i \, w^{-1}(\alpha_i)$$
lies in $R_{M_2}$ by assumption, but at the same time all $w^{-1}(\alpha_i)$ lie in $R_+$ by Lemma \ref{W-positivity-lemma}. Hence all $w^{-1}(\alpha_i)$ have to lie in $R_{M_2} \cap R_+$, and thus have to be simple roots by Lemma \ref{W-simple-roots-lemma}. Applying $w$ again we obtain a presentation of $\alpha$ as a sum of simple roots of the desired form.
\end{proof}

\medskip

Next, let $U(Q_1)$ and $U(Q_2)$ denote the unipotent radicals of $Q_1$ and $Q_2$. For the proofs of Propositions \ref{Q1Q2P1P2 relative positions} and \ref{Q1Q2 factorization} below we will furthermore need the following technical lemma:

\medskip

\begin{lemma}
\label{technical-lemma-about-Q1Q2}
The scheme-theoretic intersections
$$P_1 \cap (w P_2 w^{-1}), \ \ P_2\cap(w^{-1} P_1 w), \ \ Q_1 \cap (w Q_2 w^{-1}), \ \ (w^{-1} U(Q_1) w) \cap Q_2$$
are reduced and connected. We furthermore have the following inclusions:
\begin{itemize}

\medskip

\item[$(a)$] \ \ $P_1 \cap (w P_2 w^{-1}) \ \ \subset \ Q_1$

\medskip

\item[$(b)$] \ \ $P_2 \cap (w^{-1} P_1 w) \ \ \subset \ Q_2$

\medskip

\item[$(c)$] \ \ $(w^{-1} U(Q_1) w) \cap Q_2 \ \ \subset \ U(Q_2)$
\end{itemize}
\end{lemma}

\medskip

\begin{proof}
For reducedness we need to show that in each of the cases the Lie algebra of the intersection has the minimal dimension it can possibly have, namely the dimension of the intersection itself. This is in turn a consequence of the fact that all subgroups under consideration are defined combinatorially, i.e., on the level of roots.
Namely, for every root $\alpha$ occurring in the root decomposition of the Lie algebra of an intersection as above, the intersection contains the root subgroup $U_\alpha \cong \BG_a$ corresponding to the root~$\alpha$. Thus the claim about the dimension follows from the fact that the multiplication map
$$\prod_{\alpha \in R_+} U_{\alpha} \, \times \, T \, \times \prod_{\alpha \in R_+} U_{-\alpha} \ \ \longto \ G$$
is an open immersion, namely an isomorphism onto the open Bruhat cell of~$G$.

\medskip

Next, the intersection of any two parabolics is connected (see for example \cite[Prop. 14.22]{Borel linear algebraic groups}), establishing the claimed connectedness for the first three intersections. The connectedness of the fourth intersection follows from the following general fact about reductive groups (see for example \cite[Prop. 14.4]{Borel linear algebraic groups}), applied to the conjugate Borel subgroup $w^{-1} B w$:
If a closed subgroup of the unipotent part of a Borel subgroup is stable under conjugation by a maximal torus of this Borel, then it must be connected.

\medskip

We now show that the three inclusions $(a)$, $(b)$, $(c)$ hold. Since all subgroups under consideration are connected, the inclusions can be checked on the level of algebras of distributions. For this in turn it suffices to prove the three inclusions on the level of roots.
The assertion of part $(a)$ then translates to the claim that
$$(R_{M_1} \cup R_+) \cap w(R_{M_2} \cup R_+) \ \ \ \subset \ \ R_+ \cup R_{L_1}.$$
Proving this inclusion reduces to showing that the sets
$R_{M_1} \cap w(R_{M_2})$ and $R_{M_1} \cap w(R_+)$ are both contained in the right hand side. But the first set is equal to $R_{L_1}$ by Lemma \ref{RL1RL2}, and the second set is equal to $R_{M_1} \cap R_+$ by Lemma \ref{W-positivity-lemma}, thus completing the proof of part $(a)$. Part $(b)$ is proven analogously.

\medskip

Similarly, the assertion of part $(c)$ translates on the level of roots to the claim that
$$w^{-1}(R_+ \setminus R_{L_1}) \cap (R_+ \cup R_{L_2}) \ \ \ \subset \ \ R_+ \setminus R_{L_2},$$
which in turn follows immediately from the fact that $w^{-1}(R_{L_1}) = R_{L_2}$.
\end{proof}

\medskip

\sssec{Construction of deeper reductions}

We now turn to the construction of the deeper reductions $F_{Q_1}$ and $F_{Q_2}$. We begin by proving:

\medskip

\begin{proposition}
\label{Q1Q2P1P2 relative positions}
The natural map of stacks
$$Q_1 \backslash (Q_1 w Q_2) / Q_2 \ \ \longto \ \ P_1 \backslash (P_1 w P_2) / P_2$$
is an isomorphism.
\end{proposition}

\medskip

\begin{proof}
The double quotient $P_1 \backslash (P_1 w P_2) / P_2$ is equal to the classifying stack of the scheme-theoretic stabilizer $\Stab_{P_1 \times P_2} (w)$ of $w$ under the action of the product $P_1 \times P_2$, and similarly the stack $Q_1 \backslash (Q_1 w Q_2) / Q_2$ is equal to the classifying stack of the scheme-theoretic stabilizer $\Stab_{Q_1 \times Q_2}(w)$. Furthermore both stacks are naturally pointed, i.e., equipped with a map from the point $\Spec (k)$, and the natural map between them is in fact a map of pointed stacks. Thus to show that this map is an isomorphism reduces to showing that the inclusion of stabilizers
$$\Stab_{Q_1 \times Q_2} (w) \ \longinto \ \Stab_{P_1 \times P_2} (w)$$
is an isomorphism. By definition, this is equivalent to showing that the inclusion of intersections
$$Q_1 \cap (w Q_2 w^{-1}) \ \longinto \ P_1 \cap (w P_2 w^{-1})$$
is an isomorphism. As both intersections are reduced by Lemma \ref{technical-lemma-about-Q1Q2}, we only need to show that the inclusion map is surjective on $k$-points. But this follows immediately from parts $(a)$ and $(b)$ of the same lemma.
\end{proof}

\medskip

Next, recall that the groups $L_1$ and $L_2$ are identified with each other under conjugation by $w$. Thus we occasionally denote both groups simply by $L$. For example, in the following proposition we consider the fiber product $$\cdot / Q_1 \underset{\cdot / L}{\times} \cdot / Q_2$$ of the classifying stacks of $Q_1$, $Q_2$, and $L$.

\medskip

\begin{proposition}
\label{Q1Q2 factorization}
The forgetful map
$$Q_1 \backslash (Q_1 w Q_2) / Q_2 \ \ \longto \ \ \cdot / Q_1 \times \cdot / Q_2$$
canonically factors as
$$\xymatrix@+10pt{
Q_1 \backslash (Q_1 w Q_2) / Q_2 \ar[rr] \ar@{..>}[dr]  & &    \cdot / Q_1 \times \cdot / Q_2   \\
   &   \cdot / Q_1 \underset{\cdot / L}{\times} \cdot / Q_2 \ar[ur]  &   \\
}$$
\end{proposition}

\medskip

\begin{proof}
All three stacks being naturally pointed, we in fact claim that a canonical factorization as above exists {\it as pointed stacks}. Since the double quotient $Q_1 \backslash (Q_1 w Q_2) / Q_2$ is equal to the classifying stack of the stabilizer $\Stab_{Q_1 \times Q_2} (w)$,
this assertion in turn follows once we show that the following diagram of groups factors:

$$\xymatrix@+10pt{
\Stab_{Q_1 \times Q_2} (w)  \ar@{^{ (}->}[rr] \ar@{..>}[dr]  & &   Q_1 \times Q_2   \\
   &   Q_1 \underset{L}{\times} Q_2 \ar@{^{ (}->}[ur]  &   \\
}$$

\medskip

\noindent Here we are using that the classifying stack of the fiber product of groups $Q_1\underset{L}{\times}Q_2$ is equal to the fiber product of their classifying stacks since $Q_1$ and~$Q_2$ surject onto $L$.

\medskip

As all three groups are reduced, it suffices to prove that the diagram of groups factors as claimed on the level of $k$-points. Thus let $(x_1, x_2)$ be an element of $\Stab_{Q_1 \times Q_2}(w)$, i.e., let $x_1 \in Q_1$ and $x_2 \in Q_2$ such that $x_1 = w x_2 w^{-1}$. Then we need to show that if
$$x_1 = \ell_1 \cdot u_1 \ \ \ \ \text{and} \ \ \ \ x_2 = \ell_2 \cdot u_2$$
are the Levi decompositions of $x_1$ in $Q_1 = L_1 \cdot U(Q_1)$ and $x_2$ in $Q_2 = L_2 \cdot U(Q_2)$, then
$$w^{-1} \ell_1 w \ = \ \ell_2 \, .$$

\medskip

\noindent To see this we write the element $x_2 = w^{-1} x_1 w$ as

$$x_2 \ = \ w^{-1} \ell_1 w \cdot w^{-1} u_1 w \, .$$

\medskip

\noindent We claim that this is in fact the Levi decomposition of $x_2$ in $Q_2$. Indeed, the first factor lies in $L_2$, and thus also in $Q_2$. Since the product lies in $Q_2$ as well, the same must hold for the second factor $w^{-1} u_1 w$. Hence $w^{-1} u_1 w$ lies in the intersection $(w^{-1} U(Q_1) w) \cap Q_2$ and therefore in $U(Q_2)$ by part~$(c)$ of Lemma \ref{technical-lemma-about-Q1Q2}, proving the claim about the Levi decomposition of $x_2$. By uniqueness of the latter we conclude that $w^{-1} \ell_1 w = \ell_2$ as desired.
\end{proof}

\medskip

Combining Propositions \ref{Q1Q2P1P2 relative positions} and \ref{Q1Q2 factorization} yields the desired construction of the deeper reductions:

\medskip

\begin{corollary}
\label{FQ1FQ2}
The forgetful map
$$(P_1 \backslash G / P_2)_w \ = \ P_1 \backslash (P_1 w P_2) / P_2 \ \ \longto \ \ \cdot / P_1 \times \cdot / P_2$$
canonically factors as

$$\xymatrix@+10pt{
\ \ P_1 \backslash (P_1 w P_2) / P_2 \ \ \ar^{\cong}[rr] \ar[d] & & \ \ Q_1 \backslash (Q_1 w Q_2) / Q_2 \ \ \ar[d]   \\
\ \ \cdot / P_1 \times \cdot / P_2 \ \ & &  \ \ \cdot / Q_1 \underset{\cdot / L}{\times} \cdot / Q_2 \ \ \ar[ll]   \\
}$$

\medskip

\noindent In other words, if $F_{P_1}$ and $F_{P_2}$ are two reductions of the same $G$-bundle on a scheme $S$ such that $F_{P_1}$ is in relative position $w$ with respect to $F_{P_2}$, then there exist naturally defined reductions $F_{Q_1}$ of ${F_{P_1}}$ and $F_{Q_2}$ of $F_{P_2}$ to $Q_1$ and $Q_2$ such that $F_{Q_1}$ is still in relative position $w$ with respect to $F_{Q_2}$ and such that their Levi bundles $F_{L_1}$ and $F_{L_2}$ are naturally isomorphic when $L_1$ and $L_2$ are identified via conjugation by $w$.
\end{corollary}

\bigskip\bigskip

\ssec{Tannakian interpretation of Bruhat decomposition}
\label{Tannakian-Bruhat}

\mbox{} \\

In this section we give a Tannakian description of the Bruhat decomposition of the stack $P_1 \backslash G / P_2$. More precisely, we analyze what the notion of relative position of two reductions $F_{P_1}$ and $F_{P_2}$ translates to on the level of vector bundles associated to $G$-representations of highest weight $\lambda$. If the parabolics $P_1$ and $P_2$ are both equal to the Borel $B$, our result is stated directly in terms of the associated bundles of $F_{P_1}$ and $F_{P_2}$. For two arbitrary parabolics $P_1$ and $P_2$, our result is phrased in terms of the associated bundles of the deeper reductions~$F_{Q_1}$ and $F_{Q_2}$ from the previous Section \ref{Deeper reductions}.

\medskip

We emphasize again that while it is essential to use Weyl modules in Proposition \ref{Plucker redundancy} above, all assertions in the current section, as well as in its application to the proof of the comparison theorem via Proposition \ref{w Q_1 Q_2 proposition} below, hold for an arbitrary finite-dimensional $G$-representation of highest weight $\lambda$, and will thus be stated and proved in this generality.

\medskip

\sssec{The case $P_1 = B = P_2$}
\label{The case $P_1 = B = P_2$}

We shall first state the result in the special case $P_1 = B = P_2$, which is significantly simpler since then $Q_1 = B = Q_2$ and $L_1 = T = L_2$ as well as $W_{1,2} = W$, so that the deeper reductions of Section \ref{Deeper reductions} above do not appear explicitly.

\medskip

Let $V$ be any finite-dimensional $G$-representation of highest weight $\lambda~\in~\Lambda_G^+$ and let
$V = \bigoplus_{\nu} V[\nu]$
be its weight decomposition.
Then the direct sum of weight spaces
$$V[\geq w\lambda] \ := \ \bigoplus_{\nu \, \geq \, w\lambda} V[\nu]$$
is a submodule of $V$ for the algebra of distributions $\Dist(B)$ of the Borel $B$
and is hence
a $B$-subrepresentation of $V$. We denote the quotient of $V[\geq w\lambda]$ by the analogously defined $B$-subrepresentation $V[>w\lambda]$ by $V[w \lambda]$, since $w \lambda$ is its unique weight.

\medskip

In addition we also consider the highest weight space $V[\lambda] \subset V$. Since the algebra of distributions $\Dist(U(B))$ of the unipotent radical $U(B)$ of~$B$ acts trivially on the subspace $V[\lambda]$ and the quotient space $V[w \lambda]$, the same holds true for $U(B)$ itself and thus both $B$-representations descend to representations of the torus $T$.

\medskip

Let now $F_B$ and $\tilde{F}_B$ be two reductions to the Borel~$B$ of the same $G$-bundle on a scheme $S$, and let $F_T$ and $\tilde{F}_T$ denote their corresponding $T$-bundles. Then the above quotient and inclusion maps of $B$-representations induce the following diagram on associated vector bundles:

$$\xymatrix{
V[\lambda]_{\tilde{F}_T} \ar@{^{ (}->}^{ \ \ \ \ \kappa^\lambda}[r]   &   V_{F_G}   &   \\
   &   V[\geq w\lambda]_{F_B} \ \ar@{^{ (}->}[u] \ar@{>>}^{ \ \pi^\lambda}[r]   &   \ V[w\lambda]_{F_T}   \\
}
$$

\medskip

With this notation, the result is the following:

\begin{proposition}
Let $F_B$ and $\tilde{F}_B$ be two reductions to the Borel $B$ of the same $G$-bundle $F_G$ on a scheme $S$, and assume that $F_B$ is in relative position~$w$ with respect to $\tilde{F}_B$. Then for any $G$-representation $V$ of highest weight $\lambda \in \Lambda_G^+$ the inclusion $\kappa^{\lambda}$ factors through the subbundle $V[\geq w\lambda]_{F_B}$, and the composition of $\kappa^\lambda$ with the surjection $\pi^\lambda$ is an isomorphism of vector bundles
$$\pi^\lambda \circ \kappa^{\lambda}: \ \ V[\lambda]_{\tilde{F}_T} \ \stackrel{\cong}{\longto} \ V[w\lambda]_{F_T} \, .$$

\medskip

Diagrammatically:
$$\xymatrix{
V[\lambda]_{\tilde{F}_T} \ar@{^{ (}->}^{ \ \ \ \ \kappa^\lambda}[r] \ar@{..>}[dr] \ar@{..>}@/_5pc/[drr]^{\cong}   &   V_{F_G}   &   \\
   &   V[\geq w\lambda]_{F_B} \ \ar@{^{ (}->}[u] \ar@{>>}^{ \ \pi^\lambda}[r]   &   \ V[w\lambda]_{F_T}   \\
}
$$
\end{proposition}

\bigskip\medskip

\sssec{The general case}

We now let $P_1$ and $P_2$ be any parabolic subgroups of $G$, and let $w \in W_{1,2}$. In this setting we have already defined the groups $Q_1$, $Q_2$, $L_1$, and $L_2$ in Section \ref{The parabolics $Q_1$ and $Q_2$} above. Let
$$\BZ R_{L_1} \ := \ \spn_{\BZ}(R_{L_1}) \ \subset \ \Lambda_G$$
denote the root lattice of $L_1$ and $\BZ R_{L_2}$ the root lattice of $L_2$. Given any $G$-representation $V$ of highest weight $\lambda \in \Lambda_G^+$ with weight decomposition
$V = \bigoplus_{\nu} V[\nu]$
as before, we define the following sums of weight spaces:

\begin{align*}
V[\lambda + \BZ R_{L_2}] \ \ & := \ \ \bigoplus_{\nu \, \in \, \lambda + \BZ R_{L_2}} V[\nu]   \\
V[\geq(w\lambda + \BZ R_{L_1})] \ \ & := \ \ \sum_{\nu' \in \, w\lambda + \BZ R_{L_1}} \ \bigoplus_{\nu \geq \nu'} V[\nu]   \\
V[>(w\lambda + \BZ R_{L_1})] \ \ & := \ \ \sum_{\nu' \in \, w\lambda + \BZ R_{L_1}} \ \bigoplus_{\substack{\nu > \nu' \\ \nu \notin \, w\lambda + \BZ R_{L_1}}} V[\nu]
\end{align*}

\medskip\medskip

We record some immediate facts in the following lemma:

\medskip

\begin{lemma}
\label{subrepresentations}
\begin{itemize}
\item[]
\item[]
\item[$(a)$] The subspace $V[\lambda + \BZ R_{L_2}]$ is a $Q_2$-subrepresentation of $V$. Furthermore the unipotent radical $U(Q_2)$ acts trivially, and thus the action of $Q_2$ descends to an action of $L_2$.
\medskip
\item[$(b)$] The subspaces $V[\geq(w\lambda + \BZ R_{L_1})]$ and $V[>(w\lambda + \BZ R_{L_1})]$ are $Q_1$-subrepresentations of $V$.
\medskip
\item[$(c)$] The unipotent radical $U(Q_1)$ of $Q_1$ acts trivially on the quotient of $V[\geq(w\lambda + \BZ R_{L_1})]$ by $V[>(w\lambda + \BZ R_{L_1})]$, so that this quotient is naturally an $L_1$-representation. It will be denoted by
$V[w\lambda + \BZ R_{L_1}]$
since its weight decomposition is
$$V[w\lambda + \BZ R_{L_1}] \ \ = \ \ \bigoplus_{\nu \in \, w\lambda + \BZ R_{L_1}} V[\nu].$$
\end{itemize}
\end{lemma}

\medskip

\begin{proof}
Using algebras of distributions, all three parts are verified similarly to the analogous statements in Section \ref{The case $P_1 = B = P_2$} or in Lemma \ref{the usual sum of weight spaces} above.
\end{proof}

\medskip

\sssec{The Tannakian interpretation in the general case}

Let now $F_{P_1}$ and $F_{P_2}$ be two reductions of the same $G$-bundle on a scheme $S$, and assume that $F_{P_1}$ is in relative position $w \in W_{1,2}$ with respect to $F_{P_2}$. Let then $F_{Q_1}$ and $F_{Q_2}$ be the reductions yielded by Corollary \ref{FQ1FQ2} in this situation.

\medskip

In this setting we find the following natural maps between associated vector bundles. First, the inclusion of $Q_2$-representations
$$V[\lambda + \BZ R_{L_2}] \ \longinto \ V$$

\noindent yields a subbundle map

$$\kappa^\lambda_{Q_2}: \ V[\lambda + \BZ R_{L_2}]_{F_{L_2}} \ \longinto \ V_{F_G} $$

\noindent as in Section \ref{Subbundles induced by reductions}. Furthermore, the inclusion and quotient maps of $Q_1$-representations
$$\xymatrix{
V   &   \\
V[\geq(w\lambda + \BZ R_{L_1})] \ \ar@{^{ (}->}[u] \ar@{>>}[r]   &   \ V[w\lambda + \BZ R_{L_1}] \\
}
$$

\medskip

\noindent yield maps of associated vector bundles
$$\xymatrix{
V_{F_G}   &   \\
V[\geq(w\lambda + \BZ R_{L_1})]_{F_{Q_1}} \ \ar@{^{ (}->}[u] \ar@{>>}^{ \ \ \pi^\lambda_{Q_1}}[r]   &   \ V[w\lambda + \BZ R_{L_1}]_{F_{L_1}}. \\
}
$$

\medskip
\medskip

With this notation, we can finally state the result in the general case:

\medskip

\begin{proposition}
\label{proposition-modular-bruhat}
Let $F_{P_1}$ and $F_{P_2}$ be two reductions of the same $G$-bundle $F_G$ on a scheme $S$, and assume that $F_{P_1}$ is in relative position $w$ with respect to $F_{P_2}$. Then for any $G$-representation $V$ of highest weight $\lambda \in \Lambda_G^+$ the inclusion $\kappa^{\lambda}_{Q_2}$ factors through the subbundle $V[\geq(w\lambda + \BZ R_{L_1})]_{F_{Q_1}}$, and the composition of $\kappa^\lambda_{Q_2}$ with the surjection $\pi^\lambda_{Q_1}$ is an isomorphism of vector bundles
$$\pi^\lambda_{Q_1} \circ \kappa^{\lambda}_{Q_2}: \ \ V[\lambda + \BZ R_{L_2}]_{F_{L_2}} \ \stackrel{\cong}{\longto} \ V[w\lambda + \BZ R_{L_1}]_{F_{L_1}} \, .$$

\medskip

Diagrammatically:
$$\xymatrix{
V[\lambda + \BZ R_{L_2}]_{F_{L_2}} \ar@{^{ (}->}^{ \ \ \ \ \ \ \kappa^\lambda_{Q_2}}[r] \ar@{..>}[dr] \ar@{..>}@/_5pc/[drr]^{\cong}   &   V_{F_G}   &   \\
   &   V[\geq(w\lambda + \BZ R_{L_1})]_{F_{Q_1}} \ar@{^{ (}->}[u] \ar@{>>}^{ \ \ \pi^\lambda_{Q_1}}[r]   &   \ V[w\lambda + \BZ R_{L_1}]_{F_{L_1}}   \\
}
$$
\end{proposition}

\medskip

\begin{proof}
By Corollary \ref{FQ1FQ2} it suffices to prove that the two assertions hold for the universal family $(F_{Q_1}, F_{Q_2})$ on $Q_1 \backslash G / Q_2$ over the locus $Q_1 \backslash (Q_1 w Q_2) / Q_2$. As both assertions can be checked locally in the smooth topology, it furthermore suffices to check them after pulling back the universal family along the smooth surjection
$$G / Q_2 \ \stackrel{p \, }{\longonto} \ Q_1 \backslash G / Q_2.$$

\noindent The pullback of $F_{Q_1}$ along $p$ possesses a canonical trivialization, which in turn also induces a trivialization of the pullback of $F_G$. Moreover, the pullback of $F_{Q_2}$ along $p$ is canonically isomorphic to the tautological $Q_2$-bundle $\CP_{Q_2}$ on $G / Q_2$, i.e., the $Q_2$-bundle corresponding to the quotient map $G \onto G / Q_2$.

\medskip

We thus obtain the diagram of vector bundles

$$\xymatrix{
V[\lambda + \BZ R_{L_2}]_{\CP_{Q_2}} \ar@{^{ (}->}^{ \ \ \ \ \ \ \kappa^\lambda_{Q_2}}[r] \ar@{..>}[dr] \ar@{..>}@/_5pc/[drr]^{\cong}   &   V \otimes \CO_{G/Q_2} & \\
 & V[\geq(w\lambda + \BZ R_{L_1})] \otimes \CO_{G/Q_2} \ar@{^{ (}->}[u] \ar@{>>}^{ \ \ \pi^\lambda_{Q_1}}[r]   &   \ V[w\lambda + \BZ R_{L_1}] \otimes \CO_{G/Q_2} \\
}
\medskip
$$

\medskip

\noindent on $G/Q_2$ and need to verify the assertions (see dotted arrows) over the locus $Q_1 w Q_2 / Q_2$. As $G/Q_2$ is a variety, it suffices to check the assertions on the fibers of the vector bundles under consideration. Now if $\bar{x}$ is a point of $G/Q_2$, then by construction of $\CP_{Q_2}$ the map $\kappa^\lambda_{Q_2}$ maps the fiber of $V[\lambda + \BZ R_{L_2}]_{\CP_{Q_2}}$ at ${\bar{x}}$ isomorphically onto the (well-defined) subspace
$$\bar{x} \cdot V[\lambda + \BZ R_{L_2}] \ \subset \ V.$$

\medskip

But if $\bar{x} =: \overline{q_1 w}$ lies in the locus $Q_1 w Q_2 / Q_2$, this subspace equals
$$q_1 \cdot V[w\lambda + \BZ R_{L_1}] \ \subset \ V$$
since $w(\BZ R_{L_2}) = \BZ R_{L_1}$. It is thus contained in the subspace $V[\geq(w\lambda + \BZ R_{L_1})]$ as the latter is a $Q_1$-representation, proving the first claim. Furthermore, in this case the quotient map of $Q_1$-representations
$$V[\geq(w\lambda + \BZ R_{L_1})] \ \longonto \ V[w\lambda + \BZ R_{L_1}]$$
of course maps the subspace $q_1 \cdot V[w\lambda + \BZ R_{L_1}]$ isomorphically onto the target space, establishing the second assertion.
\end{proof}

\bigskip

\ssec{Relative position in the curve case and a key inequality}
\label{Relative position in the curve case and a key inequality}

\mbox{} \\

We now apply the results of Section \ref{Tannakian-Bruhat} in the case where the base scheme is the curve $X$. In this situation we prove a technical inequality of certain slopes (Proposition \ref{w Q_1 Q_2 proposition}) which will form the key ingredient in the proof of the comparison theorem, Theorem \ref{comparison theorem} below.

\medskip

\sssec{The curve case}

For the remainder of this section we let $F_{P_1}$ and $F_{P_2}$ be two reductions of the same $G$-bundle on the curve $X$. By Lemma~\ref{generic-relative-position} there exists a unique $w \in W_{1,2}$ such that $F_{P_1}$ is generically in relative position~$w$ with respect to $F_{P_2}$, i.e., there exists an open dense subset $U \subset X$ such that the restriction $F_{P_1}|_U$ is in relative position $w$ with respect to $F_{P_2}|_U$. Thus Corollary \ref{FQ1FQ2} yields the reductions $F_{Q_1}|_U$ and $F_{Q_2}|_U$ only on the open subset $U$. It is however easy to see:

\medskip

\begin{proposition}
\label{factorization-diagram-on-X}
The reductions $F_{Q_1}|_U$ and $F_{Q_2}|_U$ on $U$ extend uniquely to reductions $F_{Q_1}$ and $F_{Q_2}$ of $F_{P_1}$ and $F_{P_2}$ on the entire curve $X$. Furthermore, in the situation of Proposition \ref{proposition-modular-bruhat}, the factorization of the inclusion~$\kappa^\lambda_{Q_2}$ of associated vector bundles
$$\xymatrix{
V[\lambda + \BZ R_{L_2}]_{F_{L_2}} \ar@{^{ (}->}^{ \ \ \ \ \ \ \kappa^\lambda_{Q_2}}[r] \ar@{..>}[dr] \ar@{^{(}..>}@/_5pc/[drr]   &   V_{F_G}   &   \\
   &   V[\geq(w\lambda + \BZ R_{L_1})]_{F_{Q_1}} \ar@{^{ (}->}[u] \ar@{>>}^{ \ \ \pi^\lambda_{Q_1}}[r]   &   \ V[w\lambda + \BZ R_{L_1}]_{F_{L_1}}. \\
}
$$

\medskip
\medskip
\medskip
\medskip
\medskip
\medskip

\noindent holds not only on $U$ but on the entire curve $X$, for all $\lambda \in \Lambda_G^{+}$. Finally, although the composition $\pi^\lambda_{Q_1} \circ \kappa^{\lambda}_{Q_2}$ is in general only an isomorphism on $U$, it is still an injection of locally free sheaves on $X$.
\end{proposition}

\medskip

\begin{proof}
The statement about uniquely extending $F_{Q_1}$ and $F_{Q_2}$ to all of $X$ follows via standard arguments from the facts that $X$ is a smooth curve and that the quotient varieties $P_1/Q_1$ and $P_2/Q_2$ are proper. By a similar argument, the relation of containment among subbundles of a vector bundle on a smooth curve can be checked generically, proving the factorization of the inclusion $\kappa^\lambda_{Q_2}$ on all of $X$. Finally, the kernel of the map $\pi^\lambda_{Q_1} \circ \kappa^{\lambda}_{Q_2}$ is locally free because $X$ is smooth, but since it vanishes generically, it must vanish entirely.
\end{proof}

\medskip

\sssec{The key inequality}

By Proposition \ref{factorization-diagram-on-X}, the reductions $F_{Q_1}$ of $F_{P_1}$ and $F_{Q_2}$ of $F_{P_2}$ are defined on the entire curve $X$. Let $\lambdach_{P_1}$, $\lambdach_{P_2}$, $\lambdach_{Q_1}$, and~$\lambdach_{Q_2}$ denote their respective degrees. Then we have the following technical inequality of slopes:

\medskip

\begin{proposition}
\label{w Q_1 Q_2 proposition}
$$w^{-1} \phi_{Q_1}(\lambdach_{Q_1}) \ \geq \ \phi_{Q_2}(\lambdach_{Q_2})$$
\end{proposition}

\medskip

\begin{proof}
By part $(a)$ of Lemma \ref{difference lemma for two parabolics} it suffices to prove that
$$\langle \phi_{Q_1}(\lambdach_{Q_1}) , w \lambda \rangle \ \geq \ \langle \phi_{Q_2}(\lambdach_{Q_2}), \lambda \rangle$$
for every dominant weight $\lambda \in \Lambda_G^+$.
To prove it, consider the map of associated vector bundles
$$\pi^\lambda_{Q_1} \circ \kappa^{\lambda}_{Q_2}: \ V[\lambda + \BZ R_{L_2}]_{F_{L_2}} \ \longto \ V[w\lambda + \BZ R_{L_1}]_{F_{L_1}}$$
in Proposition \ref{factorization-diagram-on-X} above. According to that proposition the map $\pi^\lambda_{Q_1} \circ \kappa^{\lambda}_{Q_2}$ is generically an isomorphism and thus injective. The same then also holds for the induced map of top exterior powers, and hence
$$\mu(V[w\lambda + \BZ R_{L_1}]_{F_{L_1}}) \ \geq \ \mu(V[\lambda + \BZ R_{L_2}]_{F_{L_2}}) \, .$$

\medskip

\noindent Now on the one hand we have
$$\mu(V[\lambda + \BZ R_{L_2}]_{F_{L_2}}) \ = \ \langle \phi_{Q_2}(\lambdach_{Q_2}), \lambda \rangle$$
by part $(c)$ of Proposition \ref{slope-of-associated-bundles}. On the other hand, using part $(a)$ of Proposition \ref{slope-of-associated-bundles} one computes that
$$\mu(V[w\lambda + \BZ R_{L_1}]_{F_{L_1}}) \ = \ \langle \phi_{Q_1}(\lambdach_{Q_1}), w \lambda \rangle \, ,$$
and the desired inequality follows.
\end{proof}

\medskip\medskip\medskip

\ssec{The comparison theorem}

\mbox{} \\

This section is devoted to the proof of the following \textit{comparison theorem}, which should be regarded as the primary application of sections \ref{Relative position of two reductions}--\ref{Tannakian-Bruhat}. It shows that the canonical reduction is not only unique, but that it moreover enjoys a certain extremal property:

\medskip

\begin{theorem}
\label{comparison theorem}
Let $F_{P_1} \in \Bun_{P_1,\lambdach_{P_1}}$ and $F_{P_2} \in \Bun_{P_2, \lambdach_{P_2}}$ be two reductions of the same $G$-bundle on the curve $X$. Assume furthermore that $F_{P_1}$ lies in fact in $\Bun_{P_1,\lambdach_{P_1}}^{ss}$ and that $\lambdach_{P_1}$ is dominant $P_1$-regular. Then we have
$$\phi_{P_1}(\lambdach_{P_1}) \ \geq \ \phi_{P_2}(\lambdach_{P_2}) \, .$$

\noindent If this inequality is in fact an equality, then $P_2$ is already contained in $P_1$ and~$F_{P_1}$ is obtained from $F_{P_2}$ by extension of structure group along the inclusion $P_2 \subset P_1$.
\end{theorem}

\medskip

\sssec{Strategy of proof}

To prove the comparison theorem, we study how the hypotheses of the theorem affect the relative position of the bundles $F_{P_1}$ and $F_{P_2}$, and relate their slopes via the deeper reductions $F_{Q_1}$ and $F_{Q_2}$ from Section~\ref{Relative position in the curve case and a key inequality} above. More precisely, to prove the asserted inequality, we first relate the slopes $\phi_{P_1}(\lambdach_{P_1})$ of $F_{P_1}$ and $\phi_{P_2}(\lambdach_{P_2})$ of $F_{P_2}$ to the slopes $\phi_{Q_1}(\lambdach_{Q_1})$ of $F_{Q_1}$ and $\phi_{Q_2}(\lambdach_{Q_2})$ of $F_{Q_2}$, respectively, and then use the key inequality~\ref{w Q_1 Q_2 proposition} to relate $\phi_{Q_1}(\lambdach_{Q_1})$ and $\phi_{Q_2}(\lambdach_{Q_2})$ to each other. If the asserted inequality is in fact an equality, we again use an analysis of the relative position to show that this forces $Q_1 = Q_2 = P_2$ and $F_{Q_1} = F_{Q_2} = F_{P_2}$, from which the second claim of the theorem follows.

\medskip

\sssec{Three easy lemmas}

We first establish the following three easy lemmas needed in the proof of the comparison theorem below. Let the hypotheses and notation be as in Section \ref{Relative position in the curve case and a key inequality} above. The first two lemmas simply spell out the implications of the hypotheses of the comparison theorem in this context:

\medskip

\begin{lemma}
\label{semistability difference lemma}
Assume that $F_{P_1}$ lies in $\Bun_{P_1}^{ss}$. Then we have not only
$$\phi_{P_1}(\lambdach_{P_1}) \ \geq \ \phi_{Q_1}(\lambdach_{Q_1})$$
but also
$$w^{-1} \, \phi_{P_1}(\lambdach_{P_1}) \ \geq \ w^{-1} \, \phi_{Q_1}(\lambdach_{Q_1}).$$
\end{lemma}

\medskip

\begin{proof}
For the first inequality see Lemma \ref{P-semistability} above. Combining this inequality with part $(b)$ of Lemma \ref{difference lemma for two parabolics} applied to $P_1$ and $Q_1$ we see that
$$\phi_{P_1}(\lambdach_{P_1}) \ - \ \phi_{Q_1}(\lambdach_{Q_1}) \ \ \ \in \ \sum_{i \in \CI_{M_1}} \BQ_{\geq 0} \, \alphach_i \, .$$
Applying $w^{-1}$ to this difference, Lemma \ref{W-positivity-lemma} shows that
$$w^{-1} \, \phi_{P_1}(\lambdach_{P_1}) \ - \ w^{-1} \, \phi_{Q_1}(\lambdach_{Q_1}) \ \ \ \in \ \sum_{i \in \CI} \BQ_{\geq 0} \, \alphach_i$$
as desired.
\end{proof}

\medskip

\begin{lemma}
\label{w=1 lemma}
Assume that $\lambdach_{P_1}$ is dominant $P_1$-regular. Then
$$\phi_{P_1}(\lambdach_{P_1}) \ \geq \ w^{-1} \, \phi_{P_1}(\lambdach_{P_1}) \, ,$$
and equality holds if and only if $w = 1$.
\end{lemma}

\medskip

\begin{proof}
The inequality holds since $\phi_{P_1}(\lambdach_{P_1})$ lies in $\Lambdach_G^{\BQ,+}$. Now assume that equality holds, and assume that $w \neq 1$. Then there exists a simple root $\alpha_i$ such that $w^{-1} \alpha_i$ lies in the negative part $-R_+$ of $R$, and since $w \in W_{1,2}$ we conclude that $i \notin \CI_{M_1}$. Pairing both sides of the equality
with $w^{-1} \alpha_i$ we obtain the desired contradiction
$$0 \ \geq \ \langle \phi_{P_1}(\lambdach_{P_1}), w^{-1} \alpha_i \rangle \ = \ \langle \phi_{P_1}(\lambdach_{P_1}), \alpha_i \rangle \ > \ 0.$$
\end{proof}

\medskip

In addition to the last two lemmas about slopes, we will also need the following combinatorial lemma, which generalizes the fact that for a semisimple group the dominant cone $\Lambda_G^{\BQ,+}$ is contained in the positive cone $\Lambda_G^{\BQ,pos}$.

\medskip

\begin{lemma}
\label{dominant positive lemma}
The dominant cone $\Lambda_G^{\BQ,+}$ is contained in the cone
$$\Lambda_G^{\BQ,+} \cap \Lambda_{G,P_2}^\BQ \ \ + \ \sum_{i \in \CI_{M_2}} \BQ_{\geq 0} \, \alpha_i \, .$$
\end{lemma}

\medskip

\begin{proof}
First note that the subspaces $\Lambda_{G,P_2}^\BQ$ and $\sum_{i \in \CI_{M_2}}\BQ \alpha_i$ together span~$\Lambda_G^\BQ$. Thus given an element $\lambda \in \Lambda_G^{\BQ,+}$ we write
$$\lambda \ = \ \mu + \sum_{i \in \CI_{M_2}} c_i \alpha_i$$
with $\mu \in \Lambda_{G,P_2}^\BQ$ and $c_i \in \BQ$, and prove that $\mu \in \Lambda_G^{\BQ,+}$ and that $c_i \geq 0$ for all~$i$.
To do so, consider the natural map
$$p: \ \Lambda_G^\BQ \ \longto \ \Lambda_{[M_2, M_2]}^\BQ \, .$$

\noindent Then $p$ maps the dominant cone $\Lambda_G^{\BQ,+}$ into the dominant cone $\Lambda_{[M_2, M_2]}^{\BQ,+}$,
and thus in particular into the positive cone $\Lambda_{[M_2,M_2]}^{\BQ,pos}$ as the group $[M_2, M_2]$ is semisimple. Furthermore, by construction the image of $\lambda$ under $p$ is equal to the sum
$$\sum_{i \in \CI_{M_2}} c_i \alpha_i \ \in \ \Lambda_{[M_2, M_2]}^\BQ \, .$$
Together these two facts yield that $c_i \geq 0$ for all $i$.

\medskip

To show that $\mu$ is dominant, we only need to evaluate it on simple coroots~$\alphach_j$ for $j \in \CI \setminus \CI_{M_2}$ since it by definition already vanishes on the remaining ones. But for $j \in \CI \setminus \CI_{M_2}$ we have
$$\langle \alphach_j , \sum_{i \in \CI_{M_2}} c_i \alpha_i \rangle \ \leq \ 0$$ and hence
$$\langle \alphach_j , \mu \rangle \ = \ \langle \alphach_j , \lambda \rangle - \langle \alphach_j , \sum_{i \in \CI_{M_2}} c_i \alpha_i \rangle \ \geq \ 0$$
as desired.
\end{proof}

\medskip\medskip

\sssec{The proof}

We can now finally prove the comparison theorem:

\medskip

\begin{proof}[Proof of Theorem \ref{comparison theorem}]
We prove the claimed inequality by showing that
$$\phi_{P_1}(\lambdach_{P_1}) \ \geq \ w^{-1} \phi_{P_1}(\lambdach_{P_1}) \ \geq \ \phi_{P_2}(\lambdach_{P_2}) \, .$$
The first inequality is immediate from the fact that $\phi_{P_1}(\lambdach_{P_1})$ is dominant, as was already recorded in Lemma \ref{w=1 lemma} above. In proving the second inequality, part $(a)$ of Lemma \ref{difference lemma for two parabolics} implies that it suffices to show that
$$\langle w^{-1} \phi_{P_1}(\lambdach_{P_1}) , \lambda \rangle \ \geq \ \langle \phi_{P_2}(\lambdach_{P_2}) , \lambda \rangle$$
for all $\lambda \in \Lambda_G^{\BQ,+}$. Using Lemma \ref{dominant positive lemma} we let $\lambda = \mu + \tau$ for $\mu \in \Lambda_G^{\BQ,+} \cap \Lambda_{G,P_2}^\BQ$ and $\tau \in \sum_{i \in \CI_{M_2}}\BQ_{\geq 0} \, \alpha_i$, and verify this last inequality for $\mu$ and $\tau$ separately.

\medskip

\noindent For $\tau$ Lemma \ref{W-positivity-lemma} implies that $w \tau \in \Lambda_G^{\BQ, pos}$ and thus
$$\langle w^{-1} \phi_{P_1}(\lambdach_{P_1}) , \tau \rangle \ = \ \langle \phi_{P_1}(\lambdach_{P_1}) , w \tau \rangle \ \geq \ 0 \ = \ \langle \phi_{P_2}(\lambdach_{P_2}) , \tau \rangle.$$

\medskip

\noindent For $\mu$ we find the sequence of inequalities
\begin{align*}
\langle w^{-1} \phi_{P_1}(\lambdach_{P_1}) , \mu \rangle \ & \stackrel{\ref{semistability difference lemma}}{\geq} \ \langle w^{-1} \phi_{Q_1}(\lambdach_{Q_1}) , \mu \rangle \\
 \ & \stackrel{\ref{w Q_1 Q_2 proposition}}{\geq} \ \ \langle \phi_{Q_2}(\lambdach_{Q_2}) , \mu \rangle \\
  \ & \stackrel{\ref{difference lemma for two parabolics}(b)}{= \ \ \, } \ \ \langle \phi_{P_2}(\lambdach_{P_2}) , \mu \rangle ,
\end{align*}

\noindent finishing the proof of the first assertion of the theorem.

\medskip

To prove the second assertion we now assume that $\phi_{P_1}(\lambdach_{P_1}) = \phi_{P_2}(\lambdach_{P_2})$. Then the sequence of inequalities
$$\phi_{P_1}(\lambdach_{P_1}) \ \geq \ w^{-1} \phi_{P_1}(\lambdach_{P_1}) \ \geq \ \phi_{P_2}(\lambdach_{P_2})$$
obtained in the first part of the proof forces that $\phi_{P_1}(\lambdach_{P_1}) = w^{-1} \phi_{P_1}(\lambdach_{P_1})$. This implies that $w=1$ by Lemma \ref{w=1 lemma} and hence $Q_1=Q_2$ by Definition~\ref{The parabolics $Q_1$ and $Q_2$}.

\medskip

Next, since for $w=1$ and $Q_1 = Q_2$ the substack
$$(Q_1 \backslash (Q_1 w Q_2) / Q_2)_{w=1} \ \longinto \ Q_1 \backslash G / Q_2$$
is in fact a closed substack, the bundle $F_{Q_1}$ has relative position $w=1$ with respect to $F_{Q_2}$ not only generically, but on the entire curve $X$, and since $w=1$ we therefore obtain that $F_{Q_1} = F_{Q_2}$ on all of $X$.

\medskip

We now claim that $\CI_{M_2} \subset \CI_{M_1}$. Indeed, if $i \in \CI_{M_2}$, then
$$\langle \phi_{P_1}(\lambdach_{P_1}) , \alpha_i \rangle \ = \ \langle \phi_{P_2}(\lambdach_{P_2}) , \alpha_i \rangle \ = \ 0,$$
and since $\lambdach_{P_1}$ is dominant $P_1$-regular this forces $i \in \CI_{M_1}$. But from the facts that $\CI_{M_2} \subset \CI_{M_1}$ and that $w=1$ we conclude that $\CI_{L_1} = \CI_{L_2} = \CI_{M_2}$. Combined with the above this in turn implies that $Q_1 = Q_2 = P_2$ and $F_{Q_1} = F_{Q_2} = F_{P_2}$, proving the second assertion of the theorem.
\end{proof}

\bigskip\bigskip\bigskip\bigskip

\section{The case of characteristic $0$ and the case $P=B$}
\label{The case of characteristic $0$}

\medskip

This section is independent from the previous Section \ref{Comparing two reductions}. Under the hypothesis that the characteristic of the ground field~$k$ is $0$, we give a Tannakian interpretation of the notion of canonical reduction, from which all other results of this section are deduced. This Tannakian interpretation is also valid in arbitrary characteristic if the canonical reduction has as its structure group the Borel~$B$ of $G$ (see Proposition \ref{HN}). Thus we make no restriction on the characteristic of $k$ unless otherwise stated.

\medskip

We then use this Tannakian interpretation to prove that under the same hypotheses as above, the projection map $\Bun_{P,\lambdach_P}^{ss} \to \Bun_G$ is a monomorphism, which will in turn immediately imply the assertion of part~$(b)$ of Theorem \ref{main theorem} in Section \ref{Proof of the remaining parts of the main theorem} below. Finally, we give a quick proof of the comparison theorem in characteristic $0$, so that a reader only interested in the case of characteristic~$0$ can skip Section \ref{Comparing two reductions} altogether.

\medskip\medskip

\ssec{Tannakian interpretation of the canonical reduction}

\mbox{} \\

In this section we give an interpretation of the canonical reduction on the level of the associated subbundles $\kappa^\lambda$ from Section \ref{Subbundles induced by reductions} above. To do so we first use the slope maps $\phi_P$ from Section \ref{phiP} above to construct the following filtrations:

\sssec{Filtrations on representations of highest weight $\lambda$}
\label{Filtrations on representations of highest weight}

Let $V$ be any $G$-representation of highest weight $\lambda \in \Lambda_G^+$ and let $V = \bigoplus_{\nu} V[\nu]$ be its weight decomposition. Furthermore fix a parabolic $P$ and an element $\lambdach_P \in \Lambdach_{G,P}$. We define a filtration $V_\bullet$ on the vector space $V$ which depends on the element~$\lambdach_P$ as follows.

\medskip

For any rational number $q \in \BQ$ we define the subspace $V_q$ as the sum of weight spaces
$$V_q \ := \ \bigoplus_{\langle \phi_P(\lambdach_P), \nu \rangle \, \geq \, q} V[\nu].$$

\medskip

\noindent Clearly $V_{q'} \subseteq V_{q}$ whenever $q' \geq q$. We will consider the subspaces $V_q$ only for the finitely many $q \in \BQ$ where a jump occurs, i.e., only for those $q$ such that $V_{q'} \subsetneq V_{q}$ for all $q' > q$. Let $q_0$ be the smallest and $q_1$ the largest rational number occurring among such~$q$. Then $V_{q_1}$ is the smallest non-zero filtration step, and $V_{q_0}$ equals $V$.

\medskip

We will in fact consider the filtration $V_\bullet$ only in the following case:

\medskip

\begin{lemma}
\label{P-filtration}
Suppose $\lambdach_P \in \Lambdach_{G,P}$ is dominant $P$-regular. Then the filtration $V_\bullet$ is a filtration of $V$ by $P$-subrepresentations. Furthermore, the unipotent radical $U(P)$ acts trivially on each successive quotient
$$\gr_q V_\bullet = \bigoplus_{\langle \phi_P(\lambdach_P), \nu \rangle \, = \, q} V[\nu] \, ,$$
and hence the $P$-action on each such quotient descends to an action of the Levi $M$. Finally, the smallest step of the filtration $V_{q_1}$ is equal to the subspace $V[\lambda + \BZ R_M]$ of $V$.
\end{lemma}

\medskip

\begin{proof}
It suffices to check the first two assertions on the level of the algebras of distributions $\Dist(P)$ and $\Dist(U(P))$, respectively, in which case they are easily verified using the fact that $\lambdach_P$ is dominant $P$-regular. As $\lambda \geq \nu$ for all weights $\nu$ of $V$, the last assertion of the lemma follows from $\lambdach_P$ being dominant $P$-regular as well.
\end{proof}

\medskip

\sssec{Filtrations on associated bundles}

The filtrations of the previous section induce filtrations on associated bundles in the following setting.
Let~$S$ be any scheme over $k$ and let $F_G \in \Bun_G(S)$ be a $G$-bundle on $X \times S$. Let furthermore $F_P \in \Bun_{P,\lambdach_P}(S)$ be a reduction of $F_G$ to $P$ on $X \times S$ such that~$\lambdach_P$ is dominant $P$-regular, and as before let $F_M$ denote its corresponding Levi bundle. Then by twisting the $P$-subrepresentations $V_q$ above by~$F_P$, we obtain a filtration~$V_{\bullet, F_P}$ of the vector bundle $V_{F_G}$ by subbundles
$$0 \neq V[\lambda + \BZ R_M]_{F_M} = V_{q_1, F_P} \ \subsetneq \ldots \subsetneq \ V_{q, F_P} \ \subsetneq \ldots \subsetneq \ V_{q_0, F_P} = V_{F_G} \, .$$

\medskip

We will use this filtration in the case of an arbitrary base scheme $S$ over $k$ in Section \ref{in families} below. Before doing so, we specialize to the case $S = \Spec (k)$ and prove the above-mentioned Tannakian interpretation of the canonical reduction:

\medskip

\begin{proposition}
\label{HN}
Assume that the characteristic of the ground field $k$ is~$0$, or assume that P=B. Let $F_P$ be a reduction to $P$ of a $G$-bundle $F_G$ on the curve $X$ and suppose that $F_P$ lies in $\Bun_{P,\lambdach_P}^{ss}$ for $\lambdach_P$ dominant $P$-regular.
Then the filtration $V_{\bullet, F_P}$ of $V_{F_G}$ is the Harder-Narasimhan filtration. In particular, the subbundle
$$\kappa^\lambda: \ V[\lambda + \BZ R_M]_{F_M} \ \longinto \ V_{F_G}$$
from Section \ref{Subbundles induced by reductions} is precisely the maximal destabilizing subsheaf of $V_{F_G}$.
\end{proposition}

\medskip

\begin{proof}
Since twisting by a bundle is an exact operation and since each $\gr_q V_\bullet$ is a representation of the Levi quotient $M$ by Lemma \ref{P-filtration}, we see that
$$\gr_q (V_{\bullet, F_P}) \ = \ (\gr_q V_\bullet)_{F_M} \, .$$
To determine its slope we apply Proposition \ref{slope-of-associated-bundles} $(a)$ to the group $M$, the $M$-representation
$$\gr_q V_\bullet \ = \bigoplus_{\langle \phi_P(\lambdach_P), \nu \rangle \, = \, q} V[\nu] \, ,$$
and the $M$-bundle $F_M$. 
Namely, we compute
$$\deg((\gr_q V_\bullet)_{F_M}) \ = \ \langle \phi_P(\lambdach_P), \sum_\nu m_\nu \nu \rangle \ = \  q \cdot \sum_\nu m_\nu \nu \, .$$
It follows that 
$$\mu((\gr_q V_\bullet)_{F_M}) = q$$
and hence in particular
$$\mu((\gr_q V_\bullet)_{F_M}) \ > \ \mu((\gr_{q'} V_\bullet)_{F_M})$$
whenever $q > q'$. This proves the first property of the Harder-Narasimhan filtration (see Section \ref{The Harder-Narasimhan filtration}).

\medskip

To verify the second property, we need to show that the subquotients $(\gr_q V_\bullet)_{F_M}$ of the filtration are semistable vector bundles.
To prove this, we will use the assumption that either the characteristic of the base field $k$ is~$0$, or that $P=B$. First, in both cases the category of finite-dimensional representations of the Levi $M$ is semisimple, and thus we can decompose the $M$-representation $\gr_q V_\bullet$ into irreducible components $(\gr_q V_\bullet)_i$. Then the exact same computation as above shows that each individual associated vector bundle $((\gr_q V_\bullet)_i)_{F_M}$ has slope $q$.

\medskip

Since a direct sum of semistable vector bundles of the same slope is again semistable, it now suffices to prove that each summand $((\gr_q V_\bullet)_i)_{F_M}$ is semistable.
In the case $P=B$ every irreducible representation of $M=T$ is $1$-dimensional. Thus the summands are line bundles and hence automatically semistable. In the case that the characteristic of $k$ is $0$ the semistability of the summands follows by applying the following well-known proposition to the groups $M$ and $\GL ((\gr_q V_\bullet)_i)$.
\end{proof}

\medskip

\begin{proposition}
Let $H_1$ and $H_2$ be reductive groups over an algebraically closed field $k$ of characteristic $0$ and let $H_1 \to H_2$ be a homomorphism of algebraic groups which maps $Z_0(H_1)$ to~$Z_0(H_2)$. Let $F_{H_1}$ be a semistable $H_1$-bundle on a smooth complete curve $X$ over $k$. Then the $H_2$-bundle $F_{H_2}$ obtained from $F_{H_1}$ by extension of structure group is again semistable.
\end{proposition}

\begin{proof}
See for example \cite{RamananRamanathan}, Theorem 3.18.
\end{proof}

\medskip

\sssec{Remark}

In view of Lemma \ref{the usual sum of weight spaces} $(b)$ and Proposition \ref{Plucker redundancy}, taking the representations $V$ to be the Weyl modules $V^\lambda$ in Proposition \ref{HN} immediately implies the uniqueness of the canonical reduction in characteristic $0$. In fact, in Section \ref{Proof of the comparison theorem in characteristic 0} below we deduce from Proposition \ref{HN} the comparison theorem, a stronger result.

\medskip\medskip

\ssec{The canonical reduction in families}
\label{in families}

\mbox{} \\

Under the assumption that either $P=B$ or that the characteristic of $k$ is $0$ as above, we now prove that the canonical reduction is also unique ``in families''. Recall first that a morphism of algebraic stacks is a \textit{monomorphism} if it is representable by algebraic spaces and becomes a monomorphism of algebraic spaces after any base change to an algebraic space. One can check
that a morphism of algebraic stacks $\CX \to \CY$ is a monomorphism if and only if for every scheme $S$ the functor $\CX(S) \to \CY(S)$ is fully faithful. Then the result is:

\medskip

\begin{proposition}
\label{monoprop}
Assume that the characteristic of $k$ is $0$ or that $P=B$, and let $\lambdach_P \in \Lambdach_{G,P}$ be dominant $P$-regular. Then the projection map $$\Fp_P:  \ \Bun_{P,\lambdach_P}^{ss} \ \longto \ \Bun_G$$ is a monomorphism.
\end{proposition}

\medskip

\begin{proof}
Let $F_P \in \Bun_{P,\lambdach_P}^{ss}(S)$ and $\tilde{F}_P \in \Bun_{P,\lambdach_P}^{ss}(S)$ be two reductions to $P$ of the same $G$-bundle $F_G$ on $X \times S$. We need to prove that $F_P = \tilde{F}_P$.
We first show that the induced subbundles
$$V[\lambda + \BZ R_M]_{F_M} \ \stackrel{\kappa^{\lambda}}{\longinto} \ V_{F_G} \ \stackrel{\tilde{\kappa}_{\lambda}}\longotni \ V[\lambda + \BZ R_M]_{\tilde{F}_M}$$
are equal for all $G$-representations $V$ of highest weight $\lambda \in \Lambda_G^+$, where the notation follows Section \ref{Subbundles induced by reductions} above.

\medskip

To prove the equality of the two subbundles, we prove the stronger assertion that the vector space of homomorphisms from $V[\lambda + \BZ R_M]_{\tilde F_M}$ to the quotient bundle $V_{F_G} / V[\lambda + \BZ R_M]_{F_M}$ is trivial:
\begin{equation} 
\Hom_{\CO_{X \times S}} \bigl(V[\lambda + \BZ R_M]_{\tilde F_M}, V_{F_G} / V[\lambda + \BZ R_M]_{F_M} \bigr) \ = \ 0 \, . \tag{$\star$}
\end{equation}

\medskip

\noindent As this vector space is the space of global sections $H^0(X \times S, E)$ of the vector bundle
$$E \ := \ \bigl( V_{F_G} / V[\lambda + \BZ R_M]_{F_M} \bigr) \otimes \bigl( V[\lambda + \BZ R_M]_{\tilde F_M} \bigr)^*$$

\medskip

\noindent on $X \times S$, we can prove $(\star)$ by showing that the pushforward $p_*(E)$ of $E$ along the projection $p: X \times S \to S$ is equal to $0$. By the theorem on cohomology and base change,
the latter can be checked on the fibers of the projection~$p$, i.e., by showing that $H^0(X \times \bar s, E_{\bar s}) = 0$ for every geometric point $\bar s$ of~$S$. We thus only need to prove $(\star)$ for $S = \Spec (k)$.

\medskip

In the case $S = \Spec (k)$, Proposition \ref{HN} shows that the subbundles $V[\lambda + \BZ R_M]_{F_M}$ and $V[\lambda + \BZ R_M]_{\tilde F_M}$ agree as they are both equal to the maximal destabilizing subsheaf of $V_{F_G}$. Now $(\star)$ follows from the next lemma, Lemma \ref{MDS lemma} below.

\medskip

Having established the equality of the two subbundles, the assertion of the proposition follows by taking the representations $V$ to be the Weyl modules~$V^\lambda$ and using Proposition \ref{Plucker redundancy} above.
\end{proof}

\medskip

\begin{lemma}
\label{MDS lemma}
Let $F$ be a vector bundle on $X$ and let $D \subset F$ denote the maximal destabilizing subsheaf of $F$. Then the vector space of homomorphisms from $D$ to the quotient $F/D$ is trivial:
$$\Hom_{\CO_X} (D, F/D) \ = \ 0 \, .$$
\end{lemma}

\begin{proof}
Assume there exists a non-zero map $D \to F/D$ and let $H$ denote its image. Then the maximal destabilizing subsheaf $K$ of the quotient $F/D$ must have strictly smaller slope than $D$, and hence
$$\mu (H) \ \leq \ \mu (K) \ < \ \mu (D) \, .$$
But since $H$ is also a quotient of $D$ this contradicts the semistability of $D$ and finishes the proof.
\end{proof}

\medskip\medskip\medskip

\ssec{Proof of the comparison theorem in characteristic $0$}
\label{Proof of the comparison theorem in characteristic 0}
\mbox{} \\

We conclude this section by demonstrating how in characteristic $0$ the comparison theorem, Theorem \ref{comparison theorem}, follows directly from Proposition \ref{HN} above.

\medskip

\begin{proof}[Proof of Theorem \ref{comparison theorem} in characteristic $0$]
As in the proof of Proposition \ref{monoprop} above, consider the inclusions of subbundles $\kappa^{\lambda,1}$, $\kappa^{\lambda,2}$ induced by $F_{P_1}$ and $F_{P_2}$, for any $G$-representation $V$ of highest weight $\lambda \in \Lambda_G^+$. Then since the maximal destabilizing subsheaf of a vector bundle has maximal slope among all subbundles, the inequality
$$\phi_{P_1}(\lambdach_{P_1}) \ \geq \ \phi_{P_2}(\lambdach_{P_2})$$
follows immediately from Proposition \ref{HN} together with Proposition \ref{slope-of-associated-bundles}~$(c)$ and Lemma \ref{difference lemma for two parabolics} $(a)$.

\medskip

Next assume that $\phi_{P_1}(\lambdach_{P_1}) = \phi_{P_2}(\lambdach_{P_2})$. As $\lambdach_{P_1}$ is dominant $P_1$-regular this equality forces $\CI_{M_2} \subset \CI_{M_1}$ and thus $P_2 \subset P_1$. We need to show that the $P_1$-bundle $\tilde F_{P_1}$ obtained from $F_{P_2}$ by extension of structure group along $P_2 \subset P_1$ agrees with the reduction $F_{P_1}$. As before it suffices to show that the subbundles
$$V[\lambda + \BZ R_{M_1}]_{\tilde F_{M_1}} \ \longinto \ V_{F_G} \ \longotni \ V[\lambda + \BZ R_{M_1}]_{F_{M_1}}$$
agree for any $G$-representation $V$ of highest weight $\lambda \in \Lambda_G^+$, since we can then take the representations $V$ to be the Weyl modules $V^\lambda$ and invoke Proposition~\ref{Plucker redundancy} above.

\medskip

To prove this, note first that the slope of $\tilde F_{P_1}$ is again equal to $\phi_{P_1}(\lambdach_{P_1})$, and let $V_\bullet$ denote the filtration of $V$ from Section \ref{Filtrations on representations of highest weight} corresponding to the element $\phi_{P_1}(\lambdach_{P_1})$. By Lemma \ref{P-filtration}, the equality of the above subbundles will follow once we show that in fact all terms of the filtrations $(V_\bullet)_{\tilde F_{P_1}}$ and~$(V_\bullet)_{F_{P_1}}$ of $V_{F_G}$ agree. To see the latter, recall from the proof of Proposition \ref{HN} that
$$\mu((\gr_q V_\bullet)_{\tilde F_{M_1}}) \ = \ q \ = \ \mu((\gr_q V_\bullet)_{F_{M_1}})$$
for all $q$. This and Proposition \ref{HN} together show that the two filtrations satisfy the hypotheses of the next lemma and are thus equal.
\end{proof}

\medskip

\begin{lemma}
Let $E$ be a vector bundle on $X$ and let
$$
E_\bullet = ( 0 \neq E_1 \subsetneq E_2 \subsetneq \ldots \subsetneq E_m = E )
$$
be its Harder-Narasimhan filtration. Let 
$$
F_\bullet = ( 0 \neq F_1 \subsetneq F_2 \subsetneq \ldots \subsetneq F_m = E )
$$
be another filtration of $E$ by subbundles, with the same numerical data as $E^\bullet$, i.e., of the same length as $E^\bullet$ and such that $\rk(F_i) = \rk(E_i)$ and $\mu(F_i/F_{i+1}) = \mu(E_i/E_{i+1})$ for all $i$. Then $F^\bullet = E^\bullet$. 
\end{lemma}

\begin{proof}
Follows easily by repeatedly using the properties of the maximal destabilizing subsheaf stated in Section \ref{The Harder-Narasimhan filtration}.
\end{proof}

\bigskip\bigskip\bigskip\bigskip\bigskip\bigskip

\section{Construction of the strata and conclusion of proof}
\label{Proof of the remaining parts of the main theorem}

\medskip

In this section we first define the strata $\Bun_G^{P, \lambdach_P}$ using Drinfeld's compactifications $\bBun_P$ and then prove the main theorem, Theorem \ref{main theorem}, using the results of the previous sections. We begin by reviewing:

\medskip

\ssec{Drinfeld's compactification $\bBun_P$}
\label{Drinfeld overview}

\sssec{Overview}

In \cite{GeometricEisenstein}, D. Gaitsgory and A. Braverman construct, following V. Drinfeld, a relative compactification $\bBun_P$ of $\Bun_P$ along the fibers of the projection map
$$\Fp_P: \ \Bun_P \ \longto \ \Bun_G$$
which will be used in the present article to construct the strata $\Bun_G^{P,\lambdach_P}$.

\medskip

Below we summarize the properties of $\bBun_P$ relevant to the proof of Theorem \ref{main theorem}. A more extensive discussion can be found in \cite{GeometricEisenstein}, \cite{ICofDrinfeldCompactifications}, and also in Section \ref{drinfeld} of the present article. There we also show how $\bBun_P$ can be constructed for an arbitrary reductive group $G$, whereas the definition in~\cite{GeometricEisenstein} yields the ``correct'' object only in the case when the derived group~$[G,G]$ of $G$ is simply connected.

\medskip

\sssec{First properties}

Drinfeld's compactification $\bBun_P$ is an algebraic stack containing $\Bun_P$ as an open dense substack, and is equipped with a schematic map
$$\Fpb: \bBun_P \ \longto \ \Bun_G$$
which extends the projection $\Fp_P$ and which is proper after restriction to any connected component of $\bBun_P$. Furthermore the inclusion of $\Bun_P$ into $\bBun_P$ induces a bijection on the level of connected components:
$$\pi_0 (\bBun_P) \ = \ \pi_0 (\Bun_P) \ = \ \Lambdach_{G,P}$$

\medskip

\sssec{The maps $j_{\check\theta}$}
\label{the maps j theta quick}

The stack $\bBun_P$ possesses a natural stratification which we now describe, for a fixed component $\bBun_{P, \lambdach_P}$.
Given an element
$$\check\theta \ \ = \ \sum_{i \in \CI \setminus \CI_M} n_i \alphach_i$$
of $\Lambdach_{G,P}^{pos}$, we define $X^{\check\theta}$ to be the partially symmetrized power of the curve
$$X^{\check\theta} \ := \ \prod_{i \in \CI \setminus \CI_M} X^{(n_i)} \, .$$

\noindent Then for each $\check\theta$ there exists a naturally defined locally closed immersion
$$j_{\check\theta}: \ X^{\check\theta} \times \Bun_{P, \lambdach_P + \check\theta} \ \longinto \ \bBun_{P, \lambdach_P}$$
which renders the diagram
\begin{equation}
\begin{aligned}
\xymatrix@+10pt{
  X^{\check\theta} \times \Bun_{P, \lambdach_P + \check\theta}    \ar[d]^{pr_2}  \ar@{^{ (}->}[r]^{ \ \ \ j_{\check\theta}}  &  \bBun_{{P, \lambdach_P}} \ar[d]^{\Fpb} \\ 
  \Bun_{P, \lambdach_P + \check\theta}      \ar[r]^{\Fp_P}                  &    \Bun_G \\
}
\end{aligned}
\tag{\ref{the maps j theta quick}}
\end{equation}
commutative. If $\check\theta = 0$ the map $j_0$ is just the inclusion $\Bun_{P, \lambdach_P} \into \bBun_{P, \lambdach_P}$.

\medskip

\sssec{The stratification}
\label{bBunP stratification quick}

The collection of locally closed substacks $\bBun_{P, \lambdach_P}^{\check\theta}$ corresponding to the immersions $j_{\check{\theta}}$ defines a stratification of $\bBun_{P, \lambdach_P}$ in the sense that on the level of $k$-points $\bBun_{P, \lambdach_P}$ is equal to the disjoint union
$$\bBun_{P, \lambdach_P} \ = \ \bigcup_{\check\theta \in \Lambdach_{G,P}^{pos}} \bBun_{P, \lambdach_P}^{\check\theta} \, .$$

\medskip

\sssec{The semistable locus}

In the next proposition we demonstrate that Drinfeld's compactifications are well-suited to the study of semistability by giving a quick proof of the well-known fact that the semistable locus in~$\Bun_G$ is open. This yields precisely the open strata of the stratification in Theorem~\ref{main theorem}, and will also be used in the construction of all other strata via Drinfeld's compactifications in the next section.

\medskip

\begin{proposition}
\label{semistable locus}
The semistable locus $\Bun_G^{ss}$ is open in $\Bun_G$.
\end{proposition}

\medskip

\begin{proof}
We show that $\Bun_{G,\lambdach_G}^{ss}$ is open in $\Bun_{G,\lambdach_G}$ for a fixed $\lambdach_G \in \Lambdach_{G,G}$. For every $i \in \CI$ let $\Upsilon_i$ denote the subset of $\Lambdach_{G,P_i}$ consisting of those elements~$\lambdach_{P_i}$ that map to $\lambdach_G$ under the natural projection and which satisfy that
$$\phi_{P_i}(\lambdach_{P_i}) \ \nleq \ \phi_G(\lambdach_G) \, .$$
Then by the definition of semistability in \ref{definition of semistability} $(c)$, a $G$-bundle in $\Bun_{G,\lambdach_G}$ is semistable if and only if it does not lie in the infinite union of $k$-points
$$\bigcup_{i \in \CI} \bigcup_{\lambdach_{P_i} \in \Upsilon_i} \Fp_{P_i}(\Bun_{P_i, \lambdach_{P_i}}) \, .$$
We prove the proposition by showing that this union is the collection of $k$-points of a closed substack of $\Bun_G$.

\medskip

We first describe the sets $\Upsilon_i$ in a different way, taking into account the fact that the parabolics $P_i$ are maximal. Namely, we claim that an element $\lambdach_{P_i} \in \Lambdach_{G, P_i}$ which maps to $\lambdach_G$ under the natural projection satisfies the above condition
$$\phi_{P_i}(\lambdach_{P_i}) \ \nleq \ \phi_G(\lambdach_G)$$
in $\Lambdach_G^\BQ$ if and only if the inequality
$$\lambdach_{P_i} \ > \ \phi_G(\lambdach_G)$$
holds in $\Lambdach_{G,P_i}^\BQ$, i.e., if and only if
$$\lambdach_{P_i} \ \in \ \ \bigl( \phi_G(\lambdach_G) + \Lambdach_{G,P_i}^{\BQ, pos} \bigr) \, .$$
This assertion follows easily from part $(a)$ of Proposition \ref{phialpha} and the fact that the positive cone $\Lambdach_{G,P_i}^{\BQ,pos}$ is generated by the single coroot $\alphach_i$.

\medskip

The assertion implies that each subset $\Upsilon_i \subset \Lambdach_{G,P_i}$ contains a unique minimal element $\check\mu_{P_i}$ with respect to the partial order $\leq$ on $\Lambdach_{G,P_i}$ induced by $\Lambdach_{G,P_i}^{pos}$. We can then write the set $\Upsilon_i$ as
$$\Upsilon_i \ = \ \check\mu_{P_i} + \Lambdach_{G,P_i}^{pos} \, .$$

\medskip

We now claim that the finite union
$$\bigcup_{i \in \CI} \bar\Fp_{P_i}(\bBun_{P_i, \check\mu_{P_i}})$$
naturally forms a closed substack of $\Bun_G$ which exhibits the desired collection of $k$-points. Indeed, since the projection map $\bar\Fp_{P_i}$ is proper when restricted to any connected component of $\bBun_{P_i}$, the above union naturally carries a closed substack structure. To see that it possesses the desired collection of $k$-points, observe first that the stratification result \ref{bBunP stratification quick} and the commutativity of diagram \ref{the maps j theta quick} together imply that
$$\bar\Fp_{P_i}(\bBun_{P_i, \check\mu_{P_i}}) \ = \ \bigcup_{\check\theta \in \Lambdach_{G,P_i}^{pos}} \Fp_{P_i}(\Bun_{P_i, \check\mu_{P_i} + \check\theta})$$
on the level of $k$-points. Combined with the above fact that $\Upsilon_i = \check\mu_{P_i} + \Lambdach_{G,P_i}^{pos}$ we conclude that
$$\bigcup_{i \in \CI} \bar\Fp_{P_i}(\bBun_{P_i, \check\mu_{P_i}}) \ = \ \bigcup_{i \in \CI} \bigcup_{\lambdach_{P_i} \in \Upsilon_i} \Fp_{P_i}(\Bun_{P_i, \lambdach_{P_i}})$$
on $k$-points, finishing the proof.
\end{proof}

\medskip\medskip\medskip\medskip

\ssec{Construction of the strata and proof of $(a)$, $(b)$}
\label{Construction of the strata and proof of $(a)$, $(b)$}

\mbox{} \\

In this section we define the substacks $\Bun_G^{P,\lambdach_P}$ which form the strata of the stratification in Theorem \ref{main theorem}, and then complete the proofs of parts~$(a)$ and $(b)$ of the theorem.

\medskip

\sssec{Definition of the substacks $\Bun_G^{P,\lambdach_P}$}

Let $\lambdach_P \in \Lambdach_{G,P}$ be dominant $P$-regular, and consider Drinfeld's relative compactification
$$\Fpb: \ \bBun_P \ \longto \ \Bun_G$$
from Section \ref{Drinfeld overview}. Since the map $\Fpb$ is proper when restricted to any connected component and since $\Bun_P^{ss}$ is open in $\bBun_P$ by Proposition \ref{semistable locus} applied to the Levi $M$, the images of $\bBun_{P, \lambdach_P}$ and $\bBun_P \setminus \Bun_{P, \lambdach_P}^{ss}$ under~$\Fpb$ are closed substacks of $\Bun_G$. We then define $\Bun_G^{P,\lambdach_P}$ to be the locally closed substack
$$\Bun_G^{P,\lambdach_P} \ := \ \Fpb (\bBun_{P, \lambdach_P}) \ \setminus \ \Fpb (\bBun_P \setminus \Bun_{P, \lambdach_P}^{ss}) \, .$$

\medskip\medskip

\sssec{Proof of parts $(a)$ and $(b)$ of the main theorem}

We now use the properties of Drinfeld's compactifications from Section \ref{Drinfeld overview} above and the comparison theorem, Theorem \ref{comparison theorem}, to prove part~$(a)$ of the main theorem.

\medskip

\begin{proof}[Proof of Theorem \ref{main theorem} $(a)$]

Let $\CU$ be the open substack of $\bBun_{P, \lambdach_P}$ defined as the following fiber product:
$$\xymatrix@+10pt{
\CU \  \ar[d]  \ar@{^{ (}->}[r] & \bBun_{P, \lambdach_P} \ar[d]^{\Fpb} \\
\Bun_G^{P,\lambdach_P} \ar@{^{ (}->}[r]  &   \Fpb (\bBun_{P, \lambdach_P})   \\
}$$

\medskip\medskip

\noindent We claim that in fact $\CU = \Bun_{P, \lambdach_P}^{ss}$. To prove this, note first that both $\Bun_{P, \lambdach_P}^{ss}$ and $\CU$ are open substacks of $\overline{\Bun}_{P, \lambdach_P}$. Indeed, for $\Bun_{P, \lambdach_P}^{ss}$ this follows from Proposition \ref{semistable locus} applied to the Levi $M$, and for $\CU$ this holds by construction. Hence to prove the above claim it suffices to show that $\CU$ and $\Bun_{P, \lambdach_P}^{ss}$ coincide on the level of $k$-points.

\medskip

The inclusion $\CU \subset \Bun_{P,\lambdach_P}^{ss}$ is immediate from the definition of $\Bun_G^{P,\lambdach_P}$. We prove the converse by showing that given a $k$-point of $\Bun_{P, \lambdach_P}^{ss}$, there exists no other $k$-point of $\bBun_{P, \lambdach_P}$ with the same image under $\Fpb$. By the stratification of $\bBun_{P,\lambdach_P}$ in Section \ref{bBunP stratification quick} and the commutative diagram~\ref{the maps j theta quick}, the last assertion is equivalent to the following claim: If a $G$-bundle on $X$ admits a reduction to $P$ lying in $\Bun_{P,\lambdach_P}^{ss}$, then this reduction is in fact the only reduction to $P$ which lies in any of the connected components $\Bun_{P, \lambdach_P + \check\theta}$ for $\check\theta \in \Lambdach_{G,P}^{pos}$.

\medskip

To prove this claim, note first that for any non-zero $\check\theta \in \Lambdach_{G,P}^{pos}$ we have
$$\phi_P (\lambdach_P + \check\theta) \ > \ \phi_P (\lambdach_P)$$
by part $(a)$ of Proposition \ref{phialpha}. Thus the claim follows from the comparison theorem, Theorem \ref{comparison theorem}, completing the proof that $\CU = \Bun_{P, \lambdach_P}^{ss}$.

\medskip

Next, the fact that $\CU = \Bun_{P, \lambdach_P}^{ss}$ implies that the map
$$\Bun_{P,\lambdach_P}^{ss} \ \longto \ \Bun_G^{P,\lambdach_P}$$
is schematic and proper as it is the base change of a schematic and proper map. Furthermore, Theorem \ref{comparison theorem} shows that the canonical reduction is unique, and hence every geometric fiber of this map must consist of precisely one point as a topological space. But since the map is already proper, it has to be finite, and is hence an almost-isomorphism.

\medskip

Again since $\CU = \Bun_{P, \lambdach_P}^{ss}$ in the fiber square above, the substack $\Bun_G^{P,\lambdach_P}$ is the stack-theoretic image of the stack $\Bun_{P, \lambdach_P}^{ss}$. As the latter is reduced and quasicompact, the same holds true for the former.
Finally, the uniqueness of the substacks $\Bun_G^{P, \lambdach_P}$ is immediate from the required properties.
\end{proof}

\medskip\medskip

Part $(b)$ of the main theorem now follows immediately from part $(a)$ and the results of Section \ref{The case of characteristic $0$} above:

\medskip

\begin{proof}[Proof of Theorem \ref{main theorem} $(b)$]
The map is proper by part $(a)$ and a monomorphism by Proposition \ref{monoprop}, and hence a closed immersion. As both stacks are reduced, this and part $(a)$ together imply that the map is an isomorphism.
\end{proof}

\medskip

\ssec{Existence of the canonical reduction and proof of $(c)$}
\label{Proof of part $(c)$}

\mbox{} \\

By Theorem \ref{comparison theorem} the canonical reduction is unique. Next we prove its existence, and thus establish part~$(c)$ of the main theorem. In fact, the extremal property of the canonical reduction established in Theorem \ref{comparison theorem} suggests a strategy for its construction, which we now carry out.

\medskip

\sssec{Notation}
\label{omega notation}

As usual let $(\omega_i)_{i \in \CI}$ be the basis of $\Lambda_{[G,G]}^\BQ$ which is dual to the basis $(\alphach_i)_{i \in \CI}$ of $\Lambdach_{[G,G]}^\BQ$, and define $(\check\omega_i)_{i \in \CI}$ analogously. Furthermore for each $i \in \CI$ let $\lambda_i \in \Lambda_G^+$ be some fixed integral multiple of $\omega_i$.

\medskip\medskip

\begin{proposition}
\label{existence and uniqueness of the canonical reduction}
Every $G$-bundle on $X$ possesses a unique canonical reduction to a unique parabolic subgroup of $G$. 
\end{proposition}

\medskip

\begin{proof}
We only need to prove existence, as uniqueness is immediate from Theorem \ref{comparison theorem}. Given a $G$-bundle $F_G$ on $X$ we associate to every reduction $F_P \in \Bun_{P, \lambdach_P}$ of~$F_G$ its slope, i.e., the element $\phi_P(\lambdach_P) \in \Lambdach_G^\BQ$. Ranging over all possible reductions of $F_G$ to all parabolics of $G$ we obtain a subset of $\Lambdach_G^\BQ$ which we denote by~$\Phi$.

\medskip

We claim that $\Phi$ contains maximal elements with respect to the usual partial ordering on $\Lambdach_G^\BQ$ given by~$\leq$. To prove this, it suffices to show that for each weight $\lambda_i$ as defined above the numbers $\langle \phi_P(\lambdach_P) , \lambda_i \rangle$ remain bounded from above as $\phi_P(\lambdach_P)$ ranges over $\Phi$. To find such a bound, first recall that for any vector bundle on the curve $X$ the set of all possible slopes of subbundles is bounded from above. Now since every reduction $F_P$ of $F_G$ induces a subbundle
$$V^{\lambda_i}[\lambda_i + \BZ R_M]_{F_M} \ \longinto \ V^{\lambda_i}_{F_G}$$
whose slope is
$$\mu \bigl( V^{\lambda_i}[\lambda_i + \BZ R_M]_{F_M} \bigr) \ = \ \langle \phi_P(\lambdach_P) , \lambda_i \rangle$$
by Proposition \ref{slope-of-associated-bundles} $(c)$, applying the previous remark to the vector bundle~$V^{\lambda_i}_{F_G}$ yields an upper bound on the values of all $\langle \phi_P(\lambdach_P) , \lambda_i \rangle$.

\medskip

Let now $\check\tau$ be some maximal element of $\Phi$. Among all reductions $(P, F_{P})$ of $F_G$ such that $\phi_{P}(\lambdach_{P}) = \check\tau$ we choose one with $P$ of maximal dimension. We then claim that this pair $(P,F_P)$ has the desired properties.

\medskip

We first show that $F_P$ lies in $\Bun_P^{ss}$. By Lemma \ref{P-semistability} it suffices to show that if $F_P$ admits a reduction $F_{P'} \in \Bun_{P', \lambdach_{P'}}$ to any maximal proper sub-parabolic $P' \subset P$, then $\phi_{P'}(\lambdach_{P'}) \leq \phi_P(\lambdach_P)$.
To prove this, let $\CI_M \setminus \{i\}$ be the subset of $\CI_M$ corresponding to $P'$. Using Lemma \ref{compatibility of slope maps} we find that
$$\phi_{P'}(\lambdach_{P'}) - \phi_P(\lambdach_P) \ = \ b \cdot \phi_{P'}(\alphach_i)$$
for some rational number $b \in \BQ$. But since the element $\phi_{P'}(\alphach_i)$ lies in $\Lambdach_G^{\BQ,pos}$ by part $(a)$ of Proposition \ref{phialpha}, the choice of $\phi_P(\lambdach_P)$ as maximal implies that $b \leq 0$ and thus $\phi_{P'}(\lambdach_{P'}) \ \leq \ \phi_P(\lambdach_P)$ as desired.

\medskip

Next we show that $\lambdach_P$ is dominant $P$-regular, i.e., given any $j \in \CI \setminus \CI_M$ we need to show that $\langle \phi_P(\lambdach_P) , \alpha_j \rangle > 0$. To do so, let $P''$ be the parabolic corresponding to the subset $\CI_M \cup \{ j \}$ and let $F_{P''} \in \Bun_{P'', \lambdach_{P''}}$ be obtained from $F_P$ by extension of structure group along the inclusion $P \subset P''$. Then similarly to above we find that
$$\phi_P(\lambdach_P) - \phi_{P''}(\lambdach_{P''}) \ = \ c \cdot \phi_P(\alphach_j)$$
for some $c \in \BQ$. We claim that $c>0$. Indeed, since $\phi_P(\alphach_j)$ lies in $\Lambdach_G^{\BQ, pos}$ by part $(a)$ of Proposition \ref{phialpha}, the choice of $\phi_P(\lambdach_P)$ as maximal forces $c \geq 0$, and since $P \subsetneq P''$ the choice of $P$ as of maximal dimension forces $c>0$.

\medskip

We now pair both sides of the last equation with the root $\alpha_j$. Since $\langle \phi_{P''}(\lambdach_{P''}), \alpha_j \rangle = 0$ by construction and since $\langle \phi_P(\alphach_j), \alpha_j \rangle > 0$ by part $(b)$ of Proposition~\ref{phialpha} we conclude that $\langle \phi_P (\lambdach_P) , \alpha_j \rangle > 0$, establishing that~$\lambdach_P$ is dominant $P$-regular.
\end{proof}

\medskip\medskip

Tautologically this yields:

\medskip

\begin{proof}[Proof of Theorem \ref{main theorem} $(c)$]
Combine part $(a)$ of Theorem \ref{main theorem} with Proposition \ref{existence and uniqueness of the canonical reduction} above.
\end{proof}

\bigskip\bigskip\bigskip\bigskip\bigskip

\ssec{Proof of $(d)$, $(e)$, $(f)$, $(g)$}

\mbox{} \\

We now establish the remaining parts of the main theorem.

\medskip

\begin{proof}[Proof of Theorem \ref{main theorem} $(d)$]
Immediate from part $(a)$ and Theorem \ref{comparison theorem}.
\end{proof}

\medskip

\begin{proof}[Proof of Theorem \ref{main theorem} $(e)$]

Using the notation from Section \ref{omega notation} above, we construct an open cover $(\CU_n)_{n \in \BZ_{\geq 0}}$ of a fixed connected component $\Bun_{G, \lambdach_G}$ as follows. Choose an ample line bundle $\CO(1)$ on the curve $X$, and fix for each~$i \in \CI$ a $G$-representation of highest weight $\lambda_i$. For notational concreteness only, let us choose the Weyl modules $V^{\lambda_i}$.

\medskip

We then define the open substack $\CU_n$ of $\Bun_{G, \lambdach_G}$ as the locus where, for every $i \in \CI$, the $n$-th twist $V^{\lambda_i}_{F_G}(n)$ of the associated vector bundle $V^{\lambda_i}_{F_G}$ is generated by global sections. We now verify that this cover satisfies the desired properties.

\medskip

First, it is clear that the $\CU_n$ cover all of $\Bun_{G, \lambdach_G}$. Next, fix $n$ and let $\Psi$ be the subset of $\Lambdach_G^{\BQ,+}$ consisting of all elements $\phi_P(\lambdach_P)$ such that $\Bun_G^{P,\lambdach_P}$ meets $\CU_n$. Then we need to show that the subset $\Psi$ is finite, or equivalently, bounded. Since by part $(a)$ of Lemma \ref{difference lemma for two parabolics} every element in $\Psi$ is of the form
$$\phi_P(\lambdach_P) \ = \ \sum_{i \in \CI} c_i \check\omega_i \ + \ \phi_G (\lambdach_G)$$
for certain $c_i \in \BQ_{\geq 0}$, it is enough to show that the numbers $c_i$ arising in this way remain bounded from above. Equivalently,
it suffices to prove that the numbers $\langle \phi_P(\lambdach_P) , \lambda_i \rangle$ remain bounded from above as $\phi_P(\lambdach_P)$ ranges over~$\Psi$.

\medskip

To prove the latter for a fixed $i \in \CI$, we first claim that there exists a uniform upper bound on all possible slopes of subbundles of the associated bundles $V^{\lambda_i}_{F_G}$ as $F_G$ varies in $\CU_n$. Indeed, if we denote by $H^0(X,V^{\lambda_i}_{F_G}(n))$ the space of global sections of the twisted bundle $V^{\lambda_i}_{F_G}(n)$, then by definition of~$\CU_n$ every $V^{\lambda_i}_{F_G}$ comes equipped with a surjection
$$H^0(X,V^{\lambda_i}_{F_G}(n)) \otimes \CO(-n) \ \longonto \ V^{\lambda_i}_{F_G} \, ,$$
and thus every quotient bundle of $V^{\lambda_i}_{F_G}$ is also a quotient of $H^0(X,V^{\lambda_i}_{F_G}(n)) \otimes \CO(-n)$. Since the latter vector bundle is semistable of slope $-n$ for any $F_G$, we find that $-n$ is a uniform lower bound on all possible slopes of quotient bundles of the $V^{\lambda_i}_{F_G}$. But since the slopes of the bundles $V^{\lambda_i}_{F_G}$ are computed as
$$\mu (V^{\lambda_i}_{F_G}) \ = \ \langle \phi_G (\lambdach_G) , \lambda_i \rangle$$
by Proposition \ref{slope-of-associated-bundles} $(b)$ and are thus independent of $F_G$, the existence of a uniform lower bound on the slopes of quotients is equivalent to the existence of a uniform upper bound on the slopes of subbundles of the $V^{\lambda_i}_{F_G}$, as desired.

\medskip

Let now $N \in \BZ$ be such an upper bound. We claim that then
$$\langle \phi_P(\lambdach_P) , \lambda_i \rangle \ \leq \ N$$
for any given $\phi_P(\lambdach_P) \in \Psi$. To see this, let $F_G$ be a $G$-bundle in the intersection of $\Bun_G^{P,\lambdach_P}$ and $\CU_n$. Then by construction the canonical reduction~$F_P$ of $F_G$ lies in $\Bun_{P,\lambdach_P}$, and induces the subbundle
$$V^{\lambda_i}[\lambda_i + \BZ R_M]_{F_M} \ \longinto \ V^{\lambda_i}_{F_G} \, .$$
Since the slope of this subbundle is
$$\mu \bigl( V^{\lambda_i}[\lambda_i + \BZ R_M]_{F_M} \bigr) \ = \ \langle \phi_P(\lambdach_P) , \lambda_i \rangle$$
by Proposition \ref{slope-of-associated-bundles} $(c)$, the previous discussion shows that
$$\langle \phi_P(\lambdach_P) , \lambda_i \rangle \ \leq \ N \, ,$$
finishing the proof.
\end{proof}

\medskip

\begin{proof}[Proof of Theorem \ref{main theorem} $(f)$]
By construction the closure of $\Bun_G^{P,\lambdach_P}$ in $\Bun_G$ is equal to the image $\Fpb (\bBun_{P, \lambdach_P})$. But in light of the commutative diagram~\ref{the maps j theta quick}, the stratification \ref{bBunP stratification quick} of $\bBun_{P, \lambdach_P}$ implies that
$$\Fpb \bigl( \bBun_{P, \lambdach_P} \bigr) \ = \ \bigcup_{\check\theta \in \Lambdach_{G,P}^{pos}} \Fp_P \bigl( \Bun_{P, \lambdach_P + \check\theta} \bigr)$$
on the level of $k$-points.
\end{proof}

\medskip

\begin{proof}[Proof of Theorem \ref{main theorem} $(g)$]
A counterexample demonstrating the first assertion is given in Section \ref{An example of strata closure} below. Regarding the second assertion, the formula for the closure of the stratum $\Bun_G^{P,\lambdach_P}$ in part $(f)$ of the theorem shows that the stratum $\Bun_G^{P,\lambdach_P'}$ meets the image $\Fp_P(\Bun_{P, \lambdach_P + \check\theta})$ for some element $\check\theta \in \Lambdach_{G,P}^{pos}$. Therefore part~$(d)$ of the theorem implies that
$$\phi_P(\lambdach_P') \ \geq \ \phi_P(\lambdach_P + \check\theta)$$
and hence also
$$\phi_P(\lambdach_P') \ \geq \ \phi_P(\lambdach_P)$$
by part $(a)$ of Proposition \ref{phialpha}. As the elements $\lambdach_P$ and $\lambdach_P'$ map to the same element of $\Lambdach_{G,G}$, the last inequality in turn forces that $\lambdach_P' \geq \lambdach_P$ in~$\Lambdach_{G,P}$, i.e., that $\lambdach_P'$ lies in the set $\lambdach_P + \Lambdach_{G,P}^{pos}$. In view of the formula for the closure of the stratum $\Bun_G^{P,\lambdach_P}$ in part $(f)$ of the theorem, this yields the claim.
\end{proof}

\bigskip\bigskip\bigskip\bigskip

\ssec{An example of strata closure}
\label{An example of strata closure}

\mbox{} \\

We conclude this section with an example showing that the closure of a stratum in Theorem \ref{main theorem} need not be a union of strata. We continue to use the notation and conventions from Section \ref{The main theorem}.

\medskip

Assume that the genus of the curve is at least $2$. Let $G = \GL_3$ and let $B$ denote the standard Borel subgroup of $\GL_3$. Consider the stratum $\Bun_{\GL_3}^{B,(2,1,0)}$ of the moduli stack $\Bun_{\GL_3}$ of vector bundles of rank $3$ on $X$. This stratum consists precisely of those vector bundles $E$ whose Harder-Narasimhan flag
$$0 \neq L \subsetneq F \subsetneq E$$
is complete and satisfies
$$\deg(L) = 2, \ \ \deg(F/L) = 1, \ \ \text{and} \ \deg(E/F) = 0.$$

\medskip

We claim that the closure $\overline{\Bun_{\GL_3}^{B,(2,1,0)}}$ of this stratum is not a union of strata. To prove this, we show that there exists another stratum which meets this closure but is not contained in it. Namely, we claim that this holds true for the stratum $\Bun_{\GL_3}^{P, (3,0)}$ consisting of those vector bundles $E$ whose Harder-Narasimhan flag is of the form
$$0 \neq L \subsetneq E$$
for $L$ a line bundle with
$$\deg(L)=3 \ \ \text{and} \ \deg(E/L)=0.$$

\medskip

To see this, first observe that by the formula for the strata closure in part $(f)$ of Theorem \ref{main theorem} the closure $\overline{\Bun_{\GL_3}^{B,(2,1,0)}}$ consists of precisely those vector bundles $E$ of degree $3$ which admit a complete flag
$$0 \neq L \subsetneq F \subsetneq E$$
such that
$$\deg(L) \geq 2 \ \ \text{and} \ \deg(E/F) \leq 0.$$

\medskip

This description of the closure shows that, given a line bundle $L_0$ of degree~$3$, the closure contains the vector bundle $L_0 \oplus \CO \oplus \CO$. Since the latter bundle also lies in the stratum $\Bun_{\GL_3}^{P, (3,0)}$, we see that $\overline{\Bun_{\GL_3}^{B,(2,1,0)}}$ and $\Bun_{\GL_3}^{P, (3,0)}$ indeed intersect non-trivially.

\medskip

To complete the argument we now construct a vector bundle which lies in the stratum $\Bun_{\GL_3}^{P, (3,0)}$ but not in the closure $\overline{\Bun_{\GL_3}^{B,(2,1,0)}}$. Recall first that a vector bundle is called \textit{stable} if every proper subbundle has strictly smaller slope than the bundle itself.

\medskip

By our assumption on the genus of the curve $X$ we can choose a stable vector bundle $F_0$ of rank $2$ and of degree $0$.
Let $L_0$ again denote a line bundle of degree $3$, and consider the direct sum
$$E_0 := L_0 \oplus F_0.$$
Then $E_0$ is clearly contained in the stratum $\Bun_{\GL_3}^{P, (3,0)}$. To prove that $E_0$ does not lie in the closure $\overline{\Bun_{\GL_3}^{B,(2,1,0)}}$, we need to show that it does not admit a complete flag
$$0 \neq L \subsetneq F \subsetneq E_0$$
such that
$$\deg(L) \geq 2 \ \ \text{and} \ \deg(E_0/F) \leq 0.$$

\medskip

First, given any complete flag with these properties, one sees easily that the line subbundle $L \subsetneq E_0$ must already be equal to the first summand $L_0$ of $E_0$. We thus obtain the short exact sequence
$$0 \longto F/L_0 \longto F_0 \longto E_0/F \longto 0.$$
But then the stability of $F_0$ forces the quotient $E_0/F$ to be of strictly positive slope, in contradiction to the requirement that $\deg(E_0/F) \leq 0$.

\bigskip\bigskip\bigskip\bigskip\bigskip

\section{Drinfeld's compactifications for an arbitrary reductive group}
\label{drinfeld}

\medskip

This section can be read independently from the rest of the article. Its only purpose towards the proof of the main theorem, Theorem \ref{main theorem}, is to provide justification for not requiring that the derived group $[G,G]$ of $G$ is simply connected in the brief discussion of Drinfeld's compactification $\bBun_P$ in Section \ref{Drinfeld overview}. We continue to use the notation and conventions from Section~\ref{The setting}.

\medskip

\ssec{The original definition}
\label{drinfeld-old}

\mbox{} \\

In this section we recall the definition of Drinfeld's compactification $\bBun_P^{or}$ from \cite{GeometricEisenstein}, the superscript ``or'' standing for ``original''. Throughout the article~\cite{GeometricEisenstein} the derived group $[G,G]$ of $G$ is assumed to be simply connected, since otherwise the definition of $\bBun_P^{or}$ given there does not guarantee that $\Bun_P$ is a dense substack of $\bBun_P^{or}$. This will be remedied by the ``new'' definition of $\bBun_P$ in Section \ref{drinfeld-new} below. However, in that section we will also make use of the stack $\bBun_P^{or}$ defined in the present section, but in the case of an arbitrary reductive group $G$. For this reason we do not assume that the derived group $[G,G]$ is simply connected in this section either, necessitating the superscript ``or''.

\medskip

\sssec{Mapping stacks}

Let $G$ be a reductive group and let $P$ be a parabolic subgroup. To define $\bBun_P^{or}$ in a form which will be convenient for the generalization in Section \ref{drinfeld-new} below, we introduce the following bit of terminology. Recall that given an algebraic stack $\CY$, the sheaf of groupoids $\Maps (X, \CY)$ of maps from the curve $X$ to $\CY$ is defined to have $S$-points
$$\Maps (X, \CY)(S) \ := \ \CY (X \times S).$$

\noindent Taking $\CY$ to be the classifying stack $\cdot / P$ we recover $\Bun_P$ as
$$\Bun_P \ = \ \Maps (X, \cdot / P).$$

\medskip

Let now $\overset{\circ}{\CY} \subset \CY$ be an open substack. We then define another sheaf of groupoids 
$$\Maps_{gen}(X, \CY \supset \overset{\circ}{\CY})$$
by associating to a scheme $S$ the full sub-groupoid of $\Maps (X, \CY)(S)$ consisting of those maps $X \times S \to \CY$ satisfying the following property: We require that for every geometric point $\bar{s} \to S$ there exists an open dense subset $U \subset X \times \bar{s}$ such that the restriction of the map $X \times S \to \CY$ to $U$ factors through the open substack $\overset{\circ}{\CY} \subset \CY$. 

\medskip

It is immediate from the definition that $\Maps_{gen}(X, \CY \supset \overset{\circ}{\CY})$ is indeed a subsheaf of $\Maps (X, \CY)$. Furthermore, there is a natural inclusion of sheaves
$$\Maps (X, \overset{\circ}{\CY}) \ \longinto \ \Maps_{gen}(X, \CY \supset \overset{\circ}{\CY}).$$

\bigskip

\sssec{The target stack $\CY^{or}$}

We now construct an algebraic stack $\CY^{or}$ containing $\cdot / P$ as an open dense substack, and will then define

$$\bBun_P^{or} \ := \ \Maps_{gen}(X, \CY^{or} \supset \cdot / P).$$

\medskip

\noindent To define $\CY^{or}$, recall that a scheme Z over $k$ is called {\it strongly quasi-affine} if its ring of global functions $\Gamma(Z, \CO_Z)$ is a finitely generated $k$-algebra and if the natural map
$$Z \ \longto \ \overline{Z} \ := \ \Spec (\Gamma(Z, \CO_Z))$$
is an open immersion. If $Z$ is strongly quasi-affine we will call $\overline{Z}$ its {\it affine closure}.

\medskip

It is shown in \cite[Thm. 1.1.2]{GeometricEisenstein} that the quotient $G/[P,P]$ of $G$ by the derived group $[P,P]$ of $P$ is strongly quasi-affine. By definition of the affine closure, the left action of $G$ and the right action of $T_M := P/[P,P]$ on $G/[P,P]$ extend to actions on $\overline{G/[P,P]}$. We then define $\CY^{or}$ as the double stack quotient
$$\CY^{or} \ := \ G \backslash (\overline{G/[P,P]}) /T_M.$$
It naturally contains
$$\overset{\circ}{\CY^{or}} \ := \ \cdot / P \ = \ G \backslash (G/[P,P]) /T_M$$
as a dense open substack.

\medskip

\sssec{Definition and first properties of $\bBun_P^{or}$}
\label{Definition and first properties of BunP bar original}

We can now define
$$\bBun_P^{or} \ := \ \Maps_{gen}(X, \CY^{or} \supset \cdot / P ).$$

\medskip

\noindent It is proven in \cite[Sec. 1.3.2]{GeometricEisenstein} that $\bBun_P^{or}$ is indeed an algebraic stack, and that the natural map $\Bun_P \into \bBun_P^{or}$ realizes $\Bun_P$ as an open substack.

\medskip

Next consider the forgetful map
$$\CY^{or} \ = \ G\backslash (\overline{G/[P,P]}) /T_M \ \longto \ G\backslash \cdot \ .$$
It induces a morphism 
$$\bar \Fp_P^{or}: \bBun_P^{or} \ \longto \ \Bun_G$$
whose composition with the natural map $\Bun_P \into \bBun_P^{or}$ equals the projection $\Fp_P$. It is shown in {\it loc. cit.} that the map $\bar \Fp_P^{or}$ is schematic, and is proper when restricted to any connected component of $\bBun_P^{or}$.

\medskip

\sssec{The case of simply connected derived group}

Assume now that the derived group $[G,G]$ of $G$ is simply connected. Under this hypothesis it is shown in \cite[Prop. 1.3.8]{GeometricEisenstein} that $\Bun_P$ is in fact dense in $\bBun_P^{or}$, and that the inclusion of $\Bun_P$ into $\bBun_P^{or}$ induces a bijection on the level of connected components:
$$\pi_0 (\bBun_P^{or}) \ = \ \pi_0 (\Bun_P) \ = \ \Lambdach_{G,P} \ .$$

\noindent Furthermore, it is proven in \cite[Sec. 1.3.3]{GeometricEisenstein} that under the above hypothesis the stack $\bBun_P^{or}$ possesses the stratification which was already described in sections \ref{the maps j theta quick} and \ref{bBunP stratification quick} above.

\medskip\medskip

\ssec{The case of an arbitrary reductive group}
\label{drinfeld-new}

\mbox{} \\

In this section we define Drinfeld's compactification $\bBun_P$ for an arbitrary reductive group. As before let $G$ be any reductive group, and let $P$ be a parabolic subgroup. To define $\bBun_P$ we will proceed exactly as above, the only difference being that we will use a new target stack $\CY$, which agrees with the original target stack $\CY^{or}$ if $[G,G]$ is simply connected but differs in general.

\medskip

\sssec{Two lemmas}

We will need the following two facts from the theory of reductive groups. Both can be easily proven on the level of root data.

\medskip

\begin{lemma}
\label{grouplemma1}
Let $G$ be a reductive group. Then there exists a short exact sequence
$$ 1 \longto Z \longto \tilde{G} \longto G \longto 1$$
where $Z$ is a connected torus which is central in $\tilde{G}$, and $\tilde{G}$ is a reductive group whose derived group $[\tilde{G}, \tilde{G}]$ is simply connected.
\end{lemma}

\medskip

\begin{lemma}
\label{grouplemma2}
Let $G$ be a reductive group. Then the derived group $[G,G]$ of~$G$ is simply connected if and only if every short exact sequence of the form
$$ 1 \longto Z \longto \tilde{G} \longto G \longto 1$$
splits, where $Z$ is a connected torus which is central in $\tilde{G}$.
\end{lemma}

\medskip

\sssec{Definition of the ``correct'' target stack $\CY$}

\medskip

To define the new target stack $\CY$, choose a short exact sequence
$$ 1 \longto Z \longto \tilde{G} \longto G \longto 1$$
as in Lemma \ref{grouplemma1}. Let $\tilde{P}$ denote the inverse image of $P$ in $\tilde{G}$, and let $T_{\tilde M}$ denote the torus
$$T_{\tilde{M}} \ := \ \tilde{P}/[\tilde{P},\tilde{P}].$$

\medskip

\noindent Then as before the quotient $\tilde{G}/[\tilde{P},\tilde{P}]$ is a strongly quasi-affine variety, and the left action of $\tilde{G}$ and the right action of $T_{\tilde{M}}$ on $\tilde{G}/[\tilde{P},\tilde{P}]$ naturally extend to actions on the affine closure $\overline{\tilde{G}/[\tilde{P},\tilde{P}]}$. 
Furthermore, since the left action of the central torus $Z$ on $\tilde{G}/[\tilde{P},\tilde{P}]$ agrees with its right action via the map $Z \to T_{\tilde{M}}$, the same holds true for the induced left and right actions of $Z$ on the affine closure $\overline{\tilde{G}/[\tilde{P},\tilde{P}]}$.
We can thus define the new candidate for the target stack as
$$\CY \ := \ G \backslash \overline{\tilde{G}/[\tilde{P},\tilde{P}]} / T_{\tilde{M}} \, .$$

\medskip

\noindent As before, the stack $\CY$ naturally contains the classifying stack
$$\cdot / P \ = \ G \backslash \bigl( \tilde{G}/[\tilde{P},\tilde{P}] \bigr) / T_{\tilde{M}}$$
as a dense open substack. We furthermore have:

\medskip

\begin{proposition}
\label{independent}
The stack $\CY$ and the inclusion $\cdot / P \into \CY$ are canonically independent of the choice of $\tilde{G}$.
\end{proposition}

\begin{proof}
Given two short exact sequences
$$ 1 \longto Z_1 \longto G_1 \longto G \longto 1$$
and
$$ 1 \longto Z_2 \longto G_2 \longto G \longto 1$$
as in Lemma \ref{grouplemma1}, we define $\CY_1$ and $\CY_2$ as above and construct a canonical isomorphism $\CY_1 \cong \CY_2$ which restricts to the identity on the open substack~$\cdot / P$.

\medskip

First, let $G_3$ be the fiber product of groups $G_3 := G_1 \times_G G_2$. Then we obtain a third short exact sequence
$$ 1 \longto Z_1\times Z_2 \longto G_3 \longto G \longto 1$$
as in Lemma \ref{grouplemma1}, and we define $\CY_3$ as above. We then claim that the projection maps
$$ G_1 \longleftarrow G_3 \longto G_2$$
induce isomorphisms
$$ \CY_1 \ \stackrel{\cong}{\longleftarrow} \ \CY_3 \ \stackrel{\cong}{\longto} \ \CY_2 \, .$$

\medskip

\noindent To see this, note first that the short exact sequence
$$ 1 \longto Z_1 \longto G_3 \longto G_2 \longto 1$$
splits by Lemma \ref{grouplemma2}, and hence the map
$$G_3/[P_3, P_3] \ \longto \ G_2/[P_2, P_2]$$
is a trivial $Z_1$-bundle. But since the ring of global functions on a product of two varieties is equal to the tensor product of the rings of global functions on each factor,
the induced map between the affine closures
$$\overline{G_3/[P_3, P_3]} \ \longto \ \overline{G_2/[P_2, P_2]}$$
is again a $Z_1$-bundle. It hence induces an isomorphism
$$ G \backslash \overline{G_3/[P_3, P_3]} / T_{M_3} \ = \ \CY_3 \ \stackrel{\cong}{\longto} \ \CY_2 \ = \ G \backslash \overline{G_2/[P_2, P_2]} / T_{M_2}$$
which restricts to the identity on $\cdot / P$, as desired. For $\CY_1$ one proceeds identically.
\end{proof}

\medskip

\sssec{Example}

To see that the natural map $\CY \to \CY^{or}$ can fail to be an isomorphism if the derived group $[G,G]$ is not simply connected, consider the simplest example $G = \PGL_2$ with $P = B$ and $N:=[B,B]$, and let $\tilde G = \GL_2$ with $Z = Z_0(GL_2) = \BG_m$. Then the natural map
$$\bigl( \overline{\tilde G / \tilde N} \bigr)/Z \ \longto \ \overline{G/N}$$
cannot be an isomorphism since the torus $Z$ does not act freely on the boundary of $\overline{\tilde G / \tilde N}$. Namely, the quotient $\tilde G / \tilde N$ is isomorphic to the open subset of affine $3$-space~$\BA^3$ obtained by removing a plane and a line not contained in the plane.
The boundary of the affine closure is then equal to the missing line with the origin removed, and the torus $Z$ acts on it by a quadratic character.

\medskip

\sssec{Definition of $\bBun_P$}

Using the ``corrected'' version $\CY$ of the target stack, we define Drinfeld's compactification $\bBun_P$ for an arbitrary reductive group $G$ as the sheaf of groupoids
$$\bBun_P \ := \ \Maps_{gen}(X, \CY \supset \cdot / P) \, .$$

\medskip

\noindent As before the forgetful map $\CY \to G \backslash \cdot$ induces a morphism
$$\Fpb: \ \bBun_P \ \longto \ \Bun_G$$
whose composition with the natural inclusion $\Bun_P \into \bBun_P$ is equal to the projection $\Fp_P$.

\medskip

Below we will show that $\bBun_P$ is indeed an algebraic stack, with the properties listed in Section \ref{Drinfeld overview}. Before doing so, we can already observe:

\medskip

\begin{corollary}
If the derived group $[G,G]$ of $G$ is simply connected, then the new definition of Drinfeld's compactification agrees with the previous definition in Section \ref{drinfeld-old}:
$$\bBun_P \ = \ \bBun_P^{or} \, .$$
\end{corollary}

\begin{proof}
Immediate from Proposition \ref{independent}.
\end{proof}

\medskip\medskip\medskip

\ssec{Verification of the main properties}

\medskip

\sssec{The setup}

The basic idea in proving that $\bBun_P$ is an algebraic stack and satisfies the desired properties is as follows. Choose a short exact sequence
$$1 \longto Z \longto \tilde{G} \longto G \longto 1$$
as in Lemma \ref{grouplemma1}, and as before let $\tilde P$ denote the inverse image of $P$ in $\tilde G$. Then we will descend the desired properties from $\bBun_{\tilde P}^{or}$, in which case they are already established, to $\bBun_P$. So in addition to the target stack
$$\CY \ = \ G \backslash \overline{\tilde{G}/[\tilde{P},\tilde{P}]} / T_{\tilde{M}}$$
we also consider the double quotient stack
$$\tilde{\CY}^{or} \ := \ \tilde{G} \backslash \overline{\tilde{G}/[\tilde{P},\tilde{P}]} / T_{\tilde{M}} \, .$$
By definition we have
$$\bBun_{\tilde{P}}^{or} \ = \ \Maps_{gen}(X, \tilde{\CY}^{or} \supset \cdot / \tilde{P}) \, ,$$
and the natural map $\tilde \CY \to \CY$ induces a map
$$\bBun_{\tilde{P}}^{or} \ \longto \ \bBun_P.$$
Moreover, the smooth group stack $\cdot / Z$ naturally acts on the classifying stacks~$\cdot / \tilde G$ and $\cdot / \tilde P$ and on the stack $\tilde{\CY}^{or}$. This in turn induces actions of the smooth group stack
$$\Maps (X, \cdot / Z) \ = \ \Bun_Z$$

\medskip

\noindent on the moduli stacks $\Bun_{\tilde G}$ and $\Bun_{\tilde P}$ and the compactification $\bBun_{\tilde{P}}^{or}$. Finally, each of the natural maps
$\Bun_{\tilde G} \to \Bun_G$ and $\Bun_{\tilde P} \to \Bun_P$ and $\bBun_{\tilde{P}}^{or} \to \bBun_P$
is invariant under this action.

\medskip

In this setup, the desired properties of $\bBun_P$ will easily follow from the following proposition.

\medskip

\begin{proposition}
\label{bun-torsor}
All three vertical arrows in the commutative diagram
$$\xymatrix@+10pt{
\Bun_{\tilde{P}} \ar[d]  \ar@{^{ (}->}[r]   &   \bBun_{\tilde{P}}^{or}   \ar[d] \ar[r] &  \Bun_{\tilde{G}}  \ar[d]    \\ 
\Bun_P \ar@{^{ (}->}[r]                            &           \bBun_P   \ar[r]                 &  \Bun_G                             \\
}$$
are torsors in the etale topology for the smooth group stack $\Bun_Z$. Furthermore, both squares are cartesian.
\end{proposition}

\medskip

\begin{proof}
We first show that both squares are cartesian by showing that in fact all three squares of the following extended diagram of sheaves of groupoids are cartesian:

$$\xymatrix@+10pt{
\Bun_{\tilde{P}} \ar[d]  \ar@{^{ (}->}[r]   &   \bBun_{\tilde{P}}^{or}   \ar[d]  \ar@{^{ (}->}[r]  &  \Maps (X, \tilde{\CY}^{or})  \ar[d] \ar[r] &  \Bun_{\tilde{G}}  \ar[d]    \\ 
\Bun_P \ar@{^{ (}->}[r]                            &           \bBun_P   \ar@{^{ (}->}[r]                   &   \Maps (X, \CY)   \ar[r]                    &   \Bun_G                             \\
}$$

\medskip

\noindent To see that the right square is cartesian, consider first the cartesian square
$$\xymatrix@+10pt{
\tilde{\CY}^{or} \ar[r] \ar[d]   &   \tilde{G} \backslash \cdot \ar[d]   \\
\CY \ar[r]                        &   G \backslash \cdot                         \\
}$$
Applying the functor
$$\Maps (X, -): \ \text{(sheaves of groupoids)} \ \longto \text{(sheaves of groupoids)}$$
$$ \CZ \ \longmapsto \ \Maps (X, \CZ )$$
to this cartesian square yields the right square in the above diagram, and since the functor $\Maps (X, -)$ commutes with all homotopy limits, the right square is again cartesian, as desired.

\medskip

Similarly, the fact that the left square in the diagram is cartesian follows easily after applying $\Maps (X, -)$ to the cartesian square
$$\xymatrix@+10pt{
\cdot / \tilde{P} \ \ar@{^{ (}->}[r] \ar[d]  &   \ \tilde{\CY}^{or} \ar[d]  \\
\cdot / P \ \ar@{^{ (}->}[r]        &     \ \CY         \\
}$$

\medskip

Finally, combining the fact that this last square is cartesian with the definitions of $\bBun_P$ and $\bBun_{\tilde{P}}^{or}$ yields that the middle square of the above diagram is cartesian as well.

\medskip

Knowing that both squares in the proposition are cartesian, it now suffices to show that the right vertical arrow $\Bun_{\tilde{G}} \to \Bun_G$ is a torsor for $\Bun_Z$. To see this, recall first that the obstruction to the existence of a reduction of a given $G$-bundle on $X$ to $\tilde{G}$ lies in the second etale cohomology group $H^2_{et}(X, Z)$ with values in $Z$, which vanishes since $Z$ is a torus. This shows that the map under consideration is surjective.

\medskip

Next, the fact that $\tilde{G}$ surjects onto $G$ implies that the map is smooth.
It implies furthermore that the diagram of classifying stacks
$$\xymatrix@+10pt{
\cdot / Z \ \times \ \cdot / \tilde{G} \ \ar[r]^{\ \ \ act} \ar[d]^{pr_2}  &   \ \cdot / \tilde{G} \ar[d]  \\
\cdot / \tilde{G} \ \ar[r]        &     \ \cdot / G         \\
}$$
is cartesian. Now the torsor property of the map $\Bun_{\tilde G} \to \Bun_G$ follows by applying the functor $\Maps(X,-)$ to this cartesian square, again using that $\Maps(X,-)$ commutes with all homotopy limits.
\end{proof}

\medskip\medskip

\sssec{Algebraicity}

From Proposition \ref{bun-torsor} we immediately obtain:

\medskip

\begin{corollary}
The sheaf $\bBun_P$ is an algebraic stack.
\end{corollary}

\begin{proof}
Proposition \ref{bun-torsor} implies that the map $\Fpb: \bBun_P \to \Bun_G$ is representable by algebraic spaces, since this can be checked smooth-locally on the target and since the map $\bBun_{\tilde P}^{or} \to \Bun_{\tilde G}$ is already known to be schematic. But mapping representably to the algebraic stack $\Bun_G$, the sheaf $\bBun_P$ must be an algebraic stack as well.
\end{proof}

\medskip

In fact we have:

\medskip

\begin{proposition}
\label{schematic}
The map $\Fpb: \bBun_P \to \Bun_G$ is schematic.
\end{proposition}

\begin{proof}
We use the original stack $\bBun_P^{or}$ from Section \ref{drinfeld-old} as an intermediate step. Since the map $\bBun_P^{or} \to \Bun_G$ is already known to be schematic, it suffices to show that the natural map $\bBun_P \to \bBun_P^{or}$ is schematic as well. The latter can be deduced from the definitions using standard arguments.
\end{proof}

\sssec{Remark} One might guess that the map $\bBun_P \to \bBun_P^{or}$ from the ``corrected'' version to the original version of Drinfeld's compactification is a closed immersion. This is however not the case in general. While one can show that the map is always radicial, it is not hard to construct examples in positive characteristic showing that the map need not be a monomorphism. However, in characteristic~$0$ the map is always a closed immersion.

\medskip
\bigskip

\sssec{Main properties}

We now record the main properties of $\bBun_P$ in the following corollaries to Proposition \ref{bun-torsor}:

\medskip

\begin{corollary}
The natural inclusion $\Bun_P \into \bBun_P$ realizes $\Bun_P$ as a dense open substack of $\bBun_P$. On the level of connected components the inclusion induces a bijection
$$\pi_0 (\bBun_P) \ = \ \pi_0 (\Bun_P) \ = \ \Lambdach_{G,P} \, .$$
\end{corollary}

\medskip

\begin{proof}
Since being an open immersion can be checked smooth-locally on the target, the cartesian square of Proposition \ref{bun-torsor} reduces the question to the case of $\bBun_{\tilde{P}}^{or}$, in which case the assertion is already established. The statement about being dense is immediate from the surjectivity of the vertical maps in Proposition \ref{bun-torsor}. The assertion about the connected components follows from the analogous assertion for $\bBun_{\tilde P}^{or}$ and the torsor property in Proposition \ref{bun-torsor}.
\end{proof}

\medskip

\begin{corollary}
\label{properness-new}

The map $\bBun_P \to \Bun_G$
is proper when restricted to any connected component of $\bBun_P$.
\end{corollary}

\medskip

\begin{proof}
Let $\lambdach_P \in \Lambdach_{G,P}$ and choose an element $\lambdach_{\tilde P} \in \Lambdach_{\tilde{G},\tilde{P}}$ in the preimage of~$\lambdach_P$ under the surjection $\Lambdach_{\tilde{G},\tilde{P}} \onto \Lambdach_{G,P}$. Furthermore denote by $\lambdach_G$ and $\lambdach_{\tilde G}$ the images of $\lambdach_P$ and $\lambdach_{\tilde P}$ in $\Lambdach_{G,G}$ and $\Lambdach_{\tilde{G},\tilde{G}}$, respectively. Then Proposition~\ref{bun-torsor} implies that the square
$$\xymatrix@+10pt{
\bBun_{\tilde{P}, \lambdach_{\tilde P}}^{or}   \ar[d]  \ar[r]  &  \Bun_{\tilde{G}, \lambdach_{\tilde G}} \ar[d]  \\ 
\bBun_{P, \lambdach_P}   \ar[r]                   &   \Bun_{G, \lambdach_G}  \\
}$$
is cartesian and that the vertical arrows are torsors for the identity component $\Bun_{Z,0}$ of the group stack $\Bun_Z$; in particular they are smooth. Since being proper can be checked smooth locally on the target, the properness of the bottom horizontal arrow follows from the properness of the top horizontal arrow.
\end{proof}

\medskip

\sssec{The stratification}
\label{The stratification}

As in the proof of Corollary \ref{properness-new} let $\lambdach_P \in \Lambdach_{G,P}$ and choose an element $\lambdach_{\tilde P} \in \Lambdach_{\tilde G, \tilde P}$ in the preimage of $\lambdach_P$ under the surjection $\Lambdach_{\tilde{G},\tilde{P}} \onto \Lambdach_{G,P}$. We now deduce the stratification result for the connected component $\bBun_{P, \lambdach_P}$ stated in Section \ref{bBunP stratification quick} from the corresponding result for $\bBun_{\tilde P, \lambdach_{\tilde P}}^{or}$.
Namely, in the notation of Section \ref{the maps j theta quick}, consider for any $\check{\theta} \in \Lambdach_{\tilde{G},\tilde{P}}^{pos} = \Lambdach_{G,P}^{pos}$ the locally closed immersion
$$\tilde{j}_{\check\theta}: \ X^{\check\theta} \times \Bun_{\tilde{P}, \lambdach_{\tilde P} + \check\theta} \ \longinto \ \bBun_{\tilde{P}, \lambdach_{\tilde P}}^{or} \, .$$
It is apparent from its definition in \cite[Sec. 1.3.3]{GeometricEisenstein} that the map $\tilde{j}_{\check\theta}$ is equivariant with respect to the action of the identity component $\Bun_{Z,0}$ of the group stack $\Bun_Z$. By Proposition \ref{bun-torsor} it thus descends to a locally closed immersion
$$j_{\check\theta}: \ X^{\check\theta} \times \Bun_{P,\lambdach_P + \check\theta} \ \longinto \ \bBun_{P, \lambdach_P} \, .$$
Furthermore, as before the diagram
\begin{equation}
\begin{aligned}
\xymatrix@+10pt{
  X^{\check\theta} \times \Bun_{P, \lambdach_P + \check\theta}    \ar[d]^{pr_2}  \ar@{^{ (}->}[r]^{ \ \ \ j_{\check\theta}}  &  \bBun_{P, \lambdach_P} \ar[d]^{\Fpb} \\ 
  \Bun_{P, \lambdach_P + \check\theta}      \ar[r]^{\Fp_P}                  &    \Bun_G
  }
\end{aligned}
\tag{\ref{The stratification}}
\end{equation}
commutes. Moreover, for $\check\theta = 0$ the map $j_0$ equals the natural inclusion $\Bun_{P, \lambdach_P} \into \bBun_{P, \lambdach_P}$. Finally, since
$\bBun_{\tilde P, \lambdach_{\tilde P}}^{or}$ is a torsor for $\Bun_{Z,0}$ over $\bBun_{P, \lambdach_P}$ by Proposition \ref{bun-torsor}, the stratification of the former descends to the latter:

\medskip

\begin{corollary}
The collection of locally closed substacks $\bBun_{P, \lambdach_P}^{\check\theta}$ corresponding to the immersions $j_{\check{\theta}}$ defines a stratification of $\bBun_{P, \lambdach_P}$ in the sense that on the level of $k$-points $\bBun_{P, \lambdach_P}$ is equal to the disjoint union
$$\bBun_{P, \lambdach_P} \ = \ \bigcup_{\check\theta \in \Lambdach_{G,P}^{pos}} \bBun_{P, \lambdach_P}^{\check\theta} \, .$$
\end{corollary}

\bigskip\bigskip\bigskip\bigskip

\ssec{Drinfeld's compactification $\widetilde{\Bun}_P$}
\label{Bun_P-tilde}

\mbox{} \\

In addition to the stack $\bBun_P$, the authors of \cite{GeometricEisenstein} also consider another relative compactification $\widetilde{\Bun}_P$ of the map $\Bun_P \to \Bun_G$. The stack $\widetilde \Bun_P$ is defined analogously to $\bBun_P$, using the affine closure of the quotient $G/U(P)$ by the unipotent radical $U(P)$ of $P$ instead of the affine closure of $G/[P,P]$. This compactification is however not used in the present article. As in the case of $\bBun_P$, the stack $\widetilde{\Bun}_P$ as defined in \cite{GeometricEisenstein} has the desired properties only under the assumption that the derived group $[G,G]$ of $G$ is simply connected. However, the exact same strategy as in Section \ref{drinfeld-new} can be carried out in this situation as well. The ``corrected'' definition of $\widetilde{\Bun}_P$ for an arbitrary reductive group $G$ then satisfies the analogue of Proposition~\ref{bun-torsor}, from which all desired properties follow just as in the case of $\bBun_P$.

\end{document}